\documentclass[11pt]{amsart}

\usepackage{subfiles}
\usepackage{lipsum}
\usepackage[font=itshape]{quoting}
\usepackage{tabu, fullpage,diagbox, booktabs}
\usepackage{placeins}
\usepackage[margin=1in]{geometry}
\usepackage[dvipsnames]{xcolor}   		             		
\usepackage{graphicx}			
\usepackage{amssymb,csquotes}
\usepackage{dsfont}
\usepackage{mathrsfs}
\usepackage{amsthm}
\usepackage{amsmath}
\usepackage{stmaryrd}
\usepackage{tikz}
\usepackage{tikz-cd}
\usepackage{accents}
\usepackage{upgreek}
\usepackage{enumerate}
\usepackage{bm}
\usepackage{mathtools}
\usepackage[all]{xy}
\usepackage{caption}

\usepackage{url}
\usepackage{float}
\usepackage{todonotes}
\usepackage{bbm}
\usepackage{easyeqn}

\usepackage[full]{textcomp}
\usepackage[cal=cm]{mathalfa}
\usepackage{xparse}
\usepackage{comment}
\usepackage{cite}

\tikzset{
	commutative diagrams/.cd, 
	arrow style=tikz, 
	diagrams={>=stealth}
}

\newcommand{\supplog}{\mathpzc{Supp}}
\newcommand{\supptrop}{\mathsf{Supp}}
\theoremstyle{definition}

\makeatletter
\newcommand{\colim@}[2]{%
	\vtop{\m@th\ialign{##\cr
			\hfil$#1\operator@font colim$\hfil\cr
			\noalign{\nointerlineskip\kern1.5\ex@}#2\cr
			\noalign{\nointerlineskip\kern-\ex@}\cr}}%
}
\newcommand{\colim}{%
	\mathop{\mathpalette\colim@{\rightarrowfill@\textstyle}}\nmlimits@
}
\makeatother

\makeatletter
\def\@tocline#1#2#3#4#5#6#7{\relax
	\ifnum #1>\c@tocdepth 
	\else
	\par \addpenalty\@secpenalty\addvspace{#2}%
	\begingroup \hyphenpenalty\@M
	\@ifempty{#4}{%
		\@tempdima\csname r@tocindent\number#1\endcsname\relax
	}{%
		\@tempdima#4\relax
	}%
	\parindent\z@ \leftskip#3\relax \advance\leftskip\@tempdima\relax
	\rightskip\@pnumwidth plus4em \parfillskip-\@pnumwidth
	#5\leavevmode\hskip-\@tempdima
	\ifcase #1
	\or\or \hskip 1em \or \hskip 2em \else \hskip 3em \fi%
	#6\nobreak\relax
	\dotfill\hbox to\@pnumwidth{\@tocpagenum{#7}}\par
	\nobreak
	\endgroup
	\fi}
\makeatother

\newcommand{\drawATwoFanWithDots}[4]{%
  \draw[thick,->] (0,0) -- (2,0) node[right] {$e_1$};
  \draw[thick,->] (0,0) -- (0,2) node[above] {$e_2$};
  \fill[blue!20,opacity=0.4] (0,0) -- (2,0) -- (0,2) -- cycle;
  \filldraw[black] (0,0) circle (2pt);

  \fill[red] (#1,#2) circle (3pt);
  \fill[red] (#3,#4) circle (3pt);
}

\newcount\currentfan
\newcount\totalfans
\newdimen\gap
\newdimen\posY

\def\drawFansRecursive#1,#2,#3,#4,#5{%
  \ifnum\currentfan>\totalfans
  \else
    \pgfmathsetmacro{\x}{(\currentfan-1)*\gap}
    \begin{scope}[shift={(\x,\the\posY)}]
      \drawATwoFanWithDots{#1}{#2}{#3}{#4}
    \end{scope}
    \advance\currentfan by 1
    \ifx\relax#5\relax
    \else
      \expandafter\drawFansRecursive#5
    \fi
  \fi
}

\newcommand{\spreadFansWithDots}[5]{%
  \totalfans=#1\relax
  \posY=#3cm\relax
  \pgfmathsetmacro{\maxwidth}{#2*#5} 
  \ifnum\totalfans>1
    \pgfmathsetlengthmacro{\rawgap}{\maxwidth/(\totalfans-1)}%
    \gap=\rawgap cm
  \else
    \gap=0pt
  \fi
  \currentfan=1
  \edef\tempcoords{#4}%
  \expandafter\drawFansRecursive\tempcoords,\relax
}

\usetikzlibrary{calc}
\usetikzlibrary{fadings}
\usetikzlibrary{decorations.pathmorphing}
\usetikzlibrary{decorations.pathreplacing}

\newcounter{marginnote}
\DeclareMathAlphabet{\mathpzc}{OT1}{pzc}{m}{it}

\usepackage[backref=page]{hyperref}
\hypersetup{
	colorlinks   = true,          
	urlcolor     = blue,          
	linkcolor    = purple,          
	citecolor   = blue             
}

\theoremstyle{definition}
\newtheorem{maintheorem}{Theorem}

\newtheorem{theorem}{Theorem}[subsection]
\newtheorem{situation}[theorem]{Situation}

\newtheorem{corollary}[theorem]{Corollary}
\newtheorem{lemma}[theorem]{Lemma}
\newtheorem{proposition}[theorem]{Proposition}
\newtheorem{goal}[theorem]{Goal}
\newtheorem{remark}[theorem]{Remark}

\newtheorem*{runningexample*}{Running example}

\newtheorem{constr}[theorem]{Construction}

\newtheorem{definition}[theorem]{Definition}
\newtheorem{thmdefn}[theorem]{Theorem/Definition}
\newtheorem{example}[theorem]{Example}

\newtheorem{proposition-definition}[theorem]{Proposition-Definition}

\newenvironment{construction}    
{%
	\pushQED{\qed}\begin{constr}}
	{\popQED\end{constr}}

\newcommand{\bcd}{\begin{center}\begin{tikzcd}}
	\newcommand{\ecd}{\end{tikzcd}\end{center}}

\newcommand{\C}{\mathbb{C}}

\newcommand{\pt}{\text{pt}}

\usepackage{accents}

\newcommand{\AX}{\mathpzc{X}}

\DeclareMathAlphabet{\mathpzc}{OT1}{pzc}{m}{it}

\NewDocumentCommand{\compatibilitydatum}{m m m m m m O{} O{} O{}}{
	\begin{equation*} \begin{tikzcd}[ampersand replacement=\&]
	\: \arrow{r} \& {#1} \arrow{r} \arrow{d}{#7} \& {#2} \arrow{r} \arrow{d}{#8} \& {#3} \arrow{r}{[1]} \arrow{d}{#9} \& \: \\
	\: \arrow{r} \& {#4} \arrow{r} \& {#5} \arrow{r} \& {#6} \arrow{r} \& \:
	\end{tikzcd} \end{equation*}}

\NewDocumentCommand{\commutingsquare}{m m m m o O{} O{} O{} O{}}{
	\begin{equation}\begin{tikzcd}[ampersand replacement=\&] \label{#5}
	#1 \arrow{r}{#6} \arrow{d}{#7} \& #2 \arrow{d}{#8} \\
	#3 \arrow{r}{#9} \& #4
	\end{tikzcd}\IfValueTF{#5}{\label{#5}}{} \end{equation}}

\NewDocumentCommand{\cartesiansquarelabel}{m m m m m O{} O{} O{} O{}}{
	\begin{tikzcd}[ampersand replacement=\&]
	#1 \arrow{r}{#6} \arrow{d}{#7} \arrow[dr, phantom, "\square"] \& #2 \arrow{d}{#8} \\
	#3 \arrow{r}{#9} \& #4
	\end{tikzcd}\IfValueTF{#5}{\label{#5}}{}
}

\NewDocumentCommand{\triangleofspaces}{m m m O{} O{} O{}}{
	\begin{tikzcd} [ampersand replacement=\&]
	#1 \arrow{r}{#4} \arrow[bend right]{rr}{#5} \& #2 \arrow{r}{#6} \& #3
	\end{tikzcd}}

\begin{document}
	\title{The Logarithmic Quot space: foundations and tropicalisation}
	\author{Patrick Kennedy--Hunt}
	\maketitle
	\begin{abstract}
		We construct a logarithmic version of the Hilbert scheme, and more generally the Quot scheme, of a simple normal crossings pair. The logarithmic Quot space admits a natural tropicalisation called the space of tropical supports, which is a functor on the category of cone complexes. The fibers of the map to the space of tropical supports are algebraic. The space of tropical supports is representable by ``piecewise linear spaces'', which are introduced here to generalise fans and cone complexes to allow non--convex geometries. The space of tropical supports can be seen as a polyhedral analogue of the Hilbert scheme. The logarithmic Quot space parameterises quotient sheaves on logarithmic modifications that satisfy a natural transversality condition. We prove that our moduli space is a separated and universally closed logarithmic algebraic space. The logarithmic Hilbert space parameterizes families of proper monomorphisms, and in this way is exactly analogous to the classical Hilbert scheme. The new complexity of the space can then be viewed as stemming from the complexity of proper monomorphisms in logarithmic geometry. Our construction generalises the logarithmic Donaldson--Thomas space studied by Maulik--Ranganathan to arbitrary rank and dimension, and the good degenerations of Quot schemes of Li--Wu to simple normal crossings geometries.
	\end{abstract}

 \tableofcontents
 
	\section*{Introduction}
    Let $X$ be a smooth scheme and consider a simple normal crossings degeneration of $X$. This paper addresses the following fundamental question.
    \begin{center}
\noindent\fbox{%
    \parbox{\linewidth - 2\fboxsep}{%
 $$\text{{How can one study the Hilbert scheme or Quot schemes of $X$ under such a degeneration?}}$$
    }%
}
\end{center} Questions about degenerations lead naturally to questions for a pair $(X,D)$ with $D$ a simple normal crossings divisor on $X$. We are led to construct the Quot schemes of a pair $(X,D)$- this means studying moduli spaces of surjections of sheaves such that the quotient sheaf is well behaved in a neighbourhood of $D$. The usual Quot scheme is recovered in the case that $D$ is empty. 
 
	Let $\underline{X}$ be a projective variety over $\mathbb{C}$ and $D$ a divisor on $\underline{X}$ with components $\{\underline{D}_i\}$. Assume that $(\underline{X},D)$ is a simple normal crossings pair and write $X$ for the associated logarithmic scheme. We define a logarithmic scheme $D_i$ by equipping $\underline{D}_i$ with divisorial logarithmic structure from its intersection with the other components of $D$. Let $\mathcal{E}$ be a coherent sheaf on $X$. We are interested in understanding how the presence of $D$ affects the Quot scheme associated to $\mathcal{E}$. 
 
 The divisor $D$ picks out an open subset of the Quot scheme as follows. There are restriction maps between Grothendieck's Quot schemes $$r_i:\mathsf{Quot}(\underline{X}, \mathcal{E})\dashrightarrow \mathsf{Quot}(\underline{{D}_i}, \mathcal{E}|_{\underline{{D}_i}})$$ defined on an open subset, but which do not extend to a morphism. There is a maximal open subscheme $U_X$ of $\mathsf{Quot}(\underline{X}, \mathcal{E})$ on which all $r_i$ are defined. We define an open subset $\mathrm{Quot}(X,\mathcal{E})^o$ of $\mathsf{Quot}(\underline{X}, \mathcal{E})$ recursively on the dimension of $X$ as the intersect of the open sets $$\mathsf{Quot}(\underline{X}, \mathcal{E})=U_X\cap \bigcap_i r_i^{-1}(\mathrm{Quot}\left(\underline{D}_i,\mathcal{E}|_{\underline{D}_i})^o\right).$$
 This is a geometric characterisation of the interior of the logarithmic Quot space.
 \begin{center}
 
\noindent\fbox{%
    \parbox{\linewidth - 2\fboxsep}{%
 $$\textrm{\textbf{Goal:} provide a moduli space and tautological morphism compactifying the data}$$  $$r^o_i:\mathsf{Quot}({X}, \mathcal{E})^o\rightarrow \mathsf{Quot}({{D}_i}, \mathcal{E}|_{{{D}_i}})^o.$$
    }%
}
\end{center}
The morphisms $r_i$ are the key ingredient required to establish degeneration and gluing formulae, generalising powerful results from Gromov--Witten theory \cite{decompdegenGWinvar,abramovich2021punctured,Ran19,YixianWu} and Donaldson--Thomas theory \cite{MR23}. 
	\subsection{Key outcomes} In this paper we construct a proper moduli stack on the category of logarithmic schemes, called the \textit{logarithmic Quot space} and denoted $\mathsf{Quot}(X,\mathcal{E})$. The morphisms $r_i$ extend to morphisms of logarithmic Quot spaces. We show that the logarithmic Quot spaces we construct are representable in a suitable sense, universally closed, and separated. The main result of \cite{Quot2} completes the proof that, after fixing numerical data, the logarithmic Quot space is proper.  
	
	In the special case $\mathcal{E} = \mathcal{O}_X$ the (logarithmic) Quot scheme coincides with the (logarithmic) Hilbert scheme. The morphisms $r_i$ are geometrically interesting. For example, they establish a link between the rich geometry of the Hilbert scheme of points on a surface, in particular Grojnowski--Nakajima theory \cite{MR1386846,MR1711344}, and the Hilbert scheme of curves on a threefold. For arbitrary $\mathcal{E}$ the morphisms $r_i$ are the key tool for developing a \textit{gluing formula} generalising similar formulae in Gromov--Witten and Donaldson--Thomas theory \cite{decompdegenGWinvar,MR2683209,kim2021degeneration,li2001degeneration, LiWu, MR23,Ran19,MNOP2}. The logarithmic Quot space has a number of further connections to well-studied geometric objects, see Section \ref{sec:ConnectionLiterature}.

    Our logarithmic Quot space has good categorical properties, see especially Section \ref{sec:QuotAsModuli}. For example, the logarithmic Hilbert scheme is the moduli space of proper monomorphisms in the category of logarithmic schemes. The logarithmic Quot space is the moduli space of \textit{logarithmic surjections of coherent sheaves}, which are closely related to  surjections of sheaves in the logarithmic \'etale topology\cite{MW23,Nak17}.  
    
    A central idea of this paper is to identify a natural tropical object to associate with a logarithmic surjection of coherent sheaves, called the \textit{tropical support}. Tropical support leads to the correct replacement for the idea of \textit{minimal or basic logarithmic structure} of the logarithmic Quot space \cite{abramovich2011stable,chen2011stable,Gillam,gross2012logarithmic}. Since the logarithmic Quot space is not a scheme with logarithmic structure, we avoid using the language of minimal logarithmic structures in the sequel. The connection between tropical support and the K-theory of toric varieties is explored further in \cite{Quot2}.
 
	\subsection{The logarithmic Quot space} We define what it means for a coherent sheaf to be \textit{logarithmically flat} in Section \ref{sec:transversality}. Suppose $\mathcal{E}$ is a sheaf on $X$ logarithmically flat over $\mathrm{Spec}(\mathbb{C})$.\footnote{By $\mathrm{Spec}(\mathbb{C})$ we will always mean a point equipped with the trivial logarithmic structure.}

    In Section \ref{sec:LogQuot} we construct a diagram of stacks over the strict \'etale site $$
		\begin{tikzcd}
		\mathcal{X} \arrow[d, "\varpi"] \arrow[r,"\pi_X"]        & X \\
		\mathsf{Quot}(X,\mathcal{E}) \arrow[r, "r_i"] & \mathsf{Quot}(D_i,\mathcal{E}|_{D_i})                   
		\end{tikzcd}$$ and a universal surjection of sheaves on $\mathcal{X}$ $$q:\pi_X^{\star} \mathcal{E}\rightarrow \mathcal{F}.$$
\begin{maintheorem}\label{Mainthm:LogQuot}
  The logarithmic Quot space $\mathsf{Quot}(X,\mathcal{E})$ contains as an open algebraic substack $$\mathsf{Quot}(X, \mathcal{E})^o\subset \mathsf{Quot}(X,\mathcal{E}).$$ Moreover $\mathsf{Quot}(X,\mathcal{E})$ is universally closed, separated, and admits a logarithmic \'etale cover by schemes with logarithmic structure. Every standard logarithmic point of our moduli space determines an expansion and surjection of sheaves on this expansion. 
\end{maintheorem}

\noindent In fact, after fixing numerical data, the logarithmic Quot space is bounded and thus proper \cite{Quot2}. 

Logarithmic algebraic spaces are algebraic spaces defined up to a choice of logarithmic birational modification. The logarithmic Picard group and logarithmic multiplicative group have similar representability properties\cite{MolchoWise,RangWise}. The choice made in constructing logarithmic Donaldson--Thomas spaces is an alternative way of handling the same phenomenon \cite{MR20}. 

An $S$ valued point of $\mathsf{Quot}(X,\mathcal{E})$ is an equivalence class of surjections of sheaves on \textit{logarithmic modifications} of $X\times S$. We call an equivalence class a \textit{logarithmic surjection of coherent sheaves} and typically denote it $q$. See Section \ref{sec:logbackground} for details. From a surjection of sheaves on a logarithmic modification, we construct a tropical object called the \textit{tropical support}. The tropical support is combinatorial in nature and records the data of a distinguished class of logarithmic modifications of $X\times S$. 

	\subsection{Representability and tropicalisation.}\label{sec:intro:rep}
	The logarithmic Quot space is not quite an algebraic stack with logarithmic structure, although it does enjoy more subtle representability properties. In particular, there is a strict map of logarithmic spaces $$\pi:\mathsf{Quot}(X,\mathcal{E}) \rightarrow \supplog(\mathpzc{X})$$ with algebraic fibres. The space $\supplog$ is combinatorial in nature and closely related to the theory of Artin fans. We now explain $\supplog$ further.

The logarithmic scheme $X$ is equipped with a map $$X \rightarrow \mathpzc{X}.$$ Here the \textit{Artin fan} $\AX$ of $X$ is a zero dimensional algebraic stack equipped with logarithmic structure \cite{AWbirational}, see Section~\ref{sec:logbackground}. The category of Artin fans is equivalent to the category of cone stacks \cite{ModStckTropCurve}. In Section~\ref{Bigsec:TropSupp} we build a moduli space on the category of cone complexes $\supptrop(\mathpzc{X})$ parametrising all possible tropical supports, called the space of tropical supports. 

The space of tropical supports is not a cone stack; and to study it we must go beyond the theory of Artin fans to study \textit{piecewise linear spaces}. The theory of piecewise linear spaces generalises the theory of cone stacks \cite{ACP,ModStckTropCurve} to accomodate non-convex geometries, see Section \ref{sec:PLcomplexes} and Section \ref{sec:PLSpaces} for details. 

We make sense of subdivisions of $\supptrop(\AX)$ in Section \ref{sec:PLSpaces}, and say a subdivision with domain a cone complex is a \textit{tropical model}. Under the equivalence of the previous paragraph, tropical models are those subdivisions which are themselves algebraic. We denote the set of tropical models of $\supptrop(\AX)$ by $$S_{\AX} = \{\supptrop_\Sigma(\AX) \rightarrow \supptrop(\AX)\}.$$

	A moduli functor on the category of cone complexes lifts to define a moduli problem over the category of logarithmic schemes, see \cite[Section 7]{ModStckTropCurve}. We denote the lift of $\supptrop(\AX)$ by $\supplog(\AX)$ and the lift of $\supptrop_\Sigma(\AX)$ by $\supplog_\Sigma(\AX)$. There is a unique cone complex $\mathrm{Trop}(X)$ which lifts to $\AX$. We define \textit{proper models of the logarithmic Quot space} to be morphisms $$\mathsf{Quot}_\Sigma(X,\mathcal{E})\rightarrow \mathsf{Quot}(X,\mathcal{E})$$ pulled back along $\pi$ from tropical models $$\supplog_\Sigma(\AX)\rightarrow \supplog(\AX).$$	

     \begin{displayquote}
 \textbf{Slogan:} The failure of the logarithmic Quot space to be an algebraic stack is controlled by the failure of $\supptrop(\mathpzc{X})$ to be a cone complex. 
 \end{displayquote}
 
    \noindent Here is one more precise instance of our slogan. Consider an open subfunctor (in the \textit{face topology}) $V$ of $\supptrop(\mathpzc{X})$ which is a cone complex and let $\mathpzc{V}$ be the associated Artin fan. The preimage $\pi^{-1}(\mathpzc{V})$ in $\mathsf{Quot}_\Sigma(X,\mathcal{E})$ is represented by a Deligne--Mumford stack with logarithmic structure.
\\
	\\

Associated to each tropical model $\supptrop_\Sigma(\mathpzc{X})\rightarrow \supptrop(\mathpzc{X})$ there is a diagram of Deligne--Mumford stacks with logarithmic structure $$
		\begin{tikzcd}
		\mathcal{X}_\Sigma \arrow[d, "\varpi"] \arrow[r,"\pi_X"]        & X  \\
		\mathsf{Quot}_\Sigma(X,\mathcal{E}) \arrow[r,"r_i"] & \mathsf{Quot}_\Sigma({D}_i,\mathcal{E}|_{D_i}).                   
		\end{tikzcd}$$
A precise version of the following theorem can be found in Section \ref{sec:LogQuot}.
\begin{maintheorem}\label{thm:LogQuot}
		The above diagrams exhibit the following properties.
		
		\noindent \textbf{Representability.} The model $\mathsf{Quot}_\Sigma(X,\mathcal{E})$ is a Deligne--Mumford stack with logarithmic structure. The logarithmic Quot space has a representable cover, in the sense that in the category of stacks on the strict \'etale site $$ \varinjlim_{\Sigma\in S_{\mathpzc{X}}}\mathsf{Quot}_\Sigma(X,\mathcal{E}) = \mathsf{Quot}(X,\mathcal{E}).$$
		
		\noindent \textbf{Universally closed.} For every choice of $\Sigma$, the underlying Deligne--Mumford stack of the moduli space $\mathsf{Quot}_\Sigma(X,\mathcal{E})$ is universally closed. 
		
		\noindent \textbf{Interpretation as a moduli space.} The logarithmic Quot space is the moduli space of \textit{logarithmic surjections of coherent sheaves}. See Section \ref{sec:RelativeSheaves} for definitions. It contains an open subscheme $$\mathsf{Quot}({X},\mathcal{E})^o\subset \mathsf{Quot}(X,\mathcal{E}).$$\\
		\\
	\end{maintheorem}
Once again \cite{Quot2} establishes boundedness of $\mathsf{Quot}_\Sigma(X,\mathcal{E})$ and thus that each model is proper. A concise summary of this section is: up to a choice of polyhedral subdivision, the logarithmic Quot space is an algebraic stack with logarithmic structure.
\subsection{Structure of the paper} The following diagram is useful for understanding how the ideas in this paper connect.
$$
\begin{tikzcd}
{\mathrm{Quot}(X,\mathcal{E})} \arrow[d, "\mathrm{tropicalisation}"', dotted] \arrow[rr, "\textrm{moduli space of}", dashed] &  & {q = [\pi_\Gamma,q_\Gamma]} \arrow[d, "\mathrm{tropicalisation}", dotted] \\
\mathpzc{Supp} \arrow[rr, "\textrm{moduli space of}", dashed]                                                                &  & \mathscr{T}(q)                                                           
\end{tikzcd} $$
Here $q$ is a logarithmic surjection of coherent sheaves, see Section \ref{sec:RelativeSheaves} and has an associated tropical support $\mathscr{T}(q)$ described in Section \ref{sec:defineTropSupp}. The moduli space of tropical supports relates to $\mathpzc{Supp}$ and is studied in Section~\ref{Bigsec:TropSupp}. To make sense of tropical support and its moduli we introduce the language of piecewise linear spaces in Section \ref{sec:PLSpaces}. 

The reader may now skip to Section \ref{sec:PLSpaces} without loss of continuity. In the remainder of the introduction we provide an overview of the ideas introduced above. 
	\subsection{Piecewise linear spaces.}\label{sec:PLcomplexes} The main combinatorial innovation of this paper is the \textit{piecewise linear space}. Piecewise linear spaces are the natural language to express the tropicalisation of a logarithmic surjection of coherent sheaves, including the case of subschemes. They are analogous to the set theoretic tropicalization studied in \cite{MaclaganSturmfels,tevelev2005compactifications}, but are more sensitive to scheme theoretic data such as embedded components. Piecewise linear spaces are also helpful for understanding the tropicalisation of the logarithmic Quot space.
 
 Geometry of fine and saturated monoids is captured by sheaves on the category \textbf{RPC} of \textit{rational polyhedral cones}, see \cite{ModStckTropCurve,MR1296725} for background on \textbf{RPC} and the \textit{face topology}. We briefly recall \textit{cone complexes} in Section~\ref{sec:backgroundconecx}. Cone complexes, via their functor of points, give rise to a distinguished class of sheaves on \textbf{RPC} which have a concrete geometric description. To make language simpler we will confuse a cone complex or piecewise linear space with its functor of points, a sheaf on \textbf{RPC}.
	
	The monoid geometry in our situation is described by {piecewise linear spaces}, a generalisation of cone complexes. A piecewise linear space, denoted $\mathscr{T}, \mathscr{S}$ or $\mathscr{G}$, is the data of a sheaf on \textbf{RPC} which may be visualised as the functor of points of a geometric object, just as with cone complexes. \textit{Piecewise linear complexes} are formed by gluing piecewise linear cones: a rational polyhedral cone is the data of a convex piecewise linear cone. The relation between piecewise linear cones and piecewise linear spaces is in a precise sense the relation between affine schemes and algebraic spaces.

\begin{figure}[ht]
	\centering
	\begin{tikzpicture}[scale=0.7]

\def\downwardCone{(0,0) -- (-2,-2) -- (2,-2) -- cycle}
\def\upwardCone{(0,0) -- (-2,2) -- (2,2) -- cycle}

\def\upRayA{(0,0) -- (2,2)}
\def\upRayB{(0,0) -- (-2,2)}
\def\downRayA{(0,0) -- (2,-2)}
\def\downRayB{(0,0) -- (-2,-2)}

\begin{scope}[shift={(0cm, 0cm)}]
  \fill[blue!10] \upwardCone;
  \draw[thick] \upRayA;
  \draw[thick] \upRayB;
  \filldraw[black] (0,0) circle (3pt);
\end{scope}

\begin{scope}[shift={(4.5cm, 0cm)}]
  \fill[blue!10, even odd rule]
    (-2,2) rectangle (2,-2)
    \downwardCone;

  \fill[white] (-2,-2.05) rectangle (2,-1.95);

  \draw[thick] \downRayA;
  \draw[thick] \downRayB;
  \filldraw[black] (0,0) circle (3pt);
\end{scope}

\begin{scope}[shift={(9cm, 0cm)}]
  \fill[blue!10, even odd rule]
    (-2,2) rectangle (2,-2)
    \downwardCone;

  \fill[white] (-2,-2.05) rectangle (2,-1.95);

  \draw[thick] \downRayA;
  \draw[thick] \downRayB;
  \draw[red, thick] (0,0) -- (0,2);
  \filldraw[black] (0,0) circle (3pt);
\end{scope}

\begin{scope}[shift={(13.5cm, 0cm)}]
  \fill[blue!10] (-2,2) rectangle (2,-2);
  \filldraw[black] (0,0) circle (3pt);
\end{scope}

\end{tikzpicture}
    \caption{Left is a rational polyhedral cone $\mathscr{G}_1$. Centre--left is the piecewise linear cone $\mathscr{G}_{-1}.$ Centre--right is a tropical model of $\mathscr{G}_{-1}$. Right is the piecewise linear cone obtained as the colimit over all fan structures on $\mathbb{R}^2$.}\label{fig:FirstPLcones}
\end{figure}
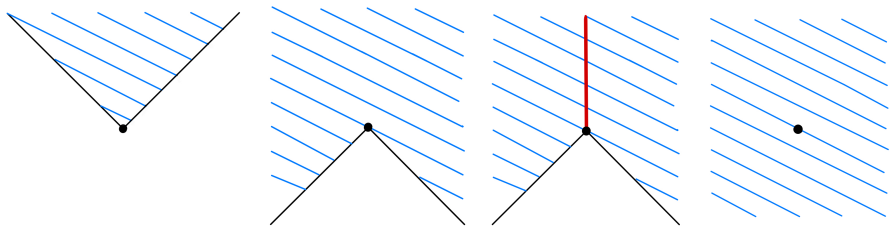

    \begin{example}
        For $k$ an integer we define subsets of $\mathbb{R}^2 = \mathbb{Z}^2 \otimes \mathbb{R}$,  $$ R_k = \{(x,y) \in \mathbb{R}^2| y\geq k|x|\}.$$ Let $J_k$ be the category of fan structures with support $R_k$. A fan structure with support ${R}_k$ defines a cone complex and thus a functor of points on the category of rational polyhedral cones. This functor of points is a sheaf with respect to the \textit{face topology}, see \cite{ModStckTropCurve}. Consider the colimit $\mathscr{G}_k$ over the functor of points of elements of $J_k$. 

        For $k>0$ note $\mathscr{G}_k$ is a rational polyhedral cone. The category $J_k$ has terminal object which is the cone with support $R_k$ and the colimit is a rational polyhedral cone.

        If $k$ is negative then $J_k$ does not have a terminal object and $\mathscr{G}_k$ cannot be a rational polyhedral cone. This is our first example of a piecewise linear cone which is not a rational polyhedral cone. For all values of $k$ we may think of $\mathscr{G}_k$ as the functor assigning to a cone $\tau$ the set of monoid maps from $\tau$ to $\mathbb{Z}^2$ whose image is contained within $R_k$. Figure \ref{fig:FirstPLcones} depicts the piecewise linear cone $\mathscr{G}_k$. for various values of $k$. 
    \end{example}
\subsubsection{How piecewise linear spaces arise}	\textit{Subdivisions} of piecewise linear spaces are defined in Section~\ref{sec:piecewiselinspaces}. A \textit{tropical model} of a piecewise linear space $\mathscr{S}$ is a morphism from a cone complex to $\mathscr{S}$ which is a subdivision. Under the slogan of Section \ref{sec:intro:rep}, tropical models correspond to logarithmic \'etale covers. In the example of Figure \ref{fig:FirstPLcones}, tropical models of $\mathscr{G}_k$ are fan structures on $\mathbb{R}^2$ with support $R_k$. 

Subdivisions of cone complexes correspond to logarithmic \'etale covers and thus piecewise linear spaces help us track a class of objects which have a logarithmic \'etale cover by logarithmic stacks, despite these objects not arising from a logarithmic structure on an algebraic stack.  In the present paper, piecewise linear spaces arise naturally in the following ways.
	\begin{enumerate}[(1)]
		\item Consider a surjection of sheaves $q: \mathcal{O}_X^n\twoheadrightarrow \mathcal{F}$ on $X$ such that $\mathcal{F}$ has no sections supported on $D$. The logarithmic modifications $$X_\Gamma \rightarrow X$$ under which the strict transform of $\mathcal{F}$ is logarithmically flat over $\mathrm{Spec}(\mathbb{C})$ are controlled by a tropical condition on $\Gamma$. The tropical condition asks $\Gamma$ is a tropical model of a certain piecewise linear space $\mathscr{T}(q)$ subdividing $\mathrm{Trop}(X)$. See \cite{Quot2} for a proof.
		\item The moduli space of subdivisions of cone complexes $\Gamma \rightarrow \mathrm{Trop}(X)$ is represented by a piecewise linear space but not by a cone complex. Moreover $\supptrop(\AX)$, which we define to be the moduli space of piecewise linear spaces subdividing $\mathrm{Trop}(X)$, is a piecewise linear space.  
		\item Our proof that the logarithmic Quot space is proper involves a version of Gr\"obner theory. Instead of a Gr\"obner fan we obtain a Gr\"obner piecewise linear space. These Gr\"obner piecewise linear spaces are built from the data of the \textit{Gr\"obner stratification} defined in \cite{Cartwright}.
	\end{enumerate}
	
	\subsubsection{Piecewise linear spaces and tropicalisation} Let $Z$ be a closed subscheme of $\underline{X}$ such that $\mathcal{O}_Z$ has no sections supported on $D$ and $Z$ is logarithmically flat over $\mathrm{Spec}(\mathbb{C})$. There is a link between the following objects. 
	\begin{enumerate}
		\item Associated to $Z$ is a subset $\mathrm{Trop}(Z)$ of the topological realisation $|\mathrm{Trop}(X)|$. We call $\mathrm{Trop}(Z)$ the \textit{tropicalisation of Z}, see \cite{MaclaganSturmfels} for details. Construction \ref{cons:} uses this data to define a piecewise linear space which is a subdivision of the cone complex $\mathrm{Trop}(X)$.
		\item The piecewise linear subdivision $\mathscr{T}(q:\mathcal{O}_X \rightarrow \mathcal{O}_Z)$.
	\end{enumerate}
	The data of (2) is a more refined version of the data of (1) and thus we may view the piecewise linear space associated to a surjection of structure sheaves as a refined version of the usual notion of tropicalisation. A key difference is that (2) detects a tropical analogue of embedded components, as the next example illustrates.
	
	\begin{example}\label{ex:keyexplspace}
		In this example we study a subscheme of $\mathbb{P}^4$ with its toric logarithmic structure. Consider $$Z = Z_1 \cup Z_2\textrm{ where }Z_1 = V(X_0 + X_1 + X_2 + X_3 + X_4)\textrm{ and }Z_2 = V(X_0-X_1,X_2-X_1, X_3-X_4).$$ The piecewise linear space $\mathscr{T}(q)$ associated to $q:\mathcal{O}_X\rightarrow \mathcal{O}_Z$ is specified by a stratification $\mathcal{P}$ of $\mathbb{R}^4$. This locally closed stratification is the common refinement of the locally closed stratifications $\{\mathrm{Trop}(Z_1), \mathrm{Trop}(Z_1)^c\}$ and $\{\mathrm{Trop}(Z_2), \mathrm{Trop}(Z_2)^c\}$ defined in \cite[Section 2.5]{MaclaganSturmfels}. See Figure~\ref{fig:KeyExample} for the piecewise linear structure on one cone of the fan of $\mathbb{P}^2$.
        The tropicalisation of $Z$ defined in \cite{MaclaganSturmfels} records only the locally closed stratification defined by $\mathrm{Trop}(Z_1)\cup \mathrm{Trop}(Z_2)$ and its complement.
	\end{example}  
	
	
	
	\subsection{The logarithmic Quot space as a moduli space}\label{sec:QuotAsModuli}
	The logarithmic Quot space is a moduli space of \textit{logarithmic surjections of coherent sheaves} of $\mathcal{E}$, defined in Section~\ref{sec:RelativeSheaves} as an equivalence class of pairs $$(\pi_\Gamma:X_\Gamma\rightarrow X, q_\Gamma: \pi_\Gamma^{\star} \mathcal{E}\rightarrow \mathcal{F})$$ where $\pi_\Gamma$ is a logarithmic modification corresponding to the subdivision $\Gamma \rightarrow \mathrm{Trop}(X)$ and $\mathcal{F}$ a logarithmically flat coherent sheaf on the underlying scheme of $X_\Gamma$ in the \'etale topology. We will write $q$ for the equivalence class. The equivalence relation is generated by pullback along certain logarithmic modifications.
	
	We associate to a logarithmic surjection of coherent sheaves, say $q=[X_\Gamma,q_\Gamma]$, a piecewise linear space $\mathscr{T}(q)$ subdividing $\mathrm{Trop}(X)$. We call this subdivision the \textit{tropical support} of $q$. The subdivision $\Gamma$ is a tropical model of $\mathscr{T}(q)$. The tropical support of a logarithmic surjection of sheaves tracks the minimal expansion $\pi_\Gamma:X_\Gamma \rightarrow X$ such that we may write $q = [\pi_\Gamma,q_\Gamma]$. 
 
	In the special case $\mathcal{E}= \mathcal{O}_X$, Grothendieck's Quot scheme is the \textit{Hilbert scheme}, a moduli space of proper monomorphisms in the category of schemes. Proper monomorphisms in the category of logarithmic schemes are compositions of two flavours of proper monomorphism: strict closed immersions and logarithmic modifications. We write $\mathsf{Hilb}(X) = \mathsf{Quot}(X, \mathcal{O}_X)$ and call it the \textit{logarithmic Hilbert space}. 
	
	Since both logarithmic modifications and strict closed immersions define proper monomorphisms in the category of logarithmic schemes, a morphism from a logarithmic scheme $B$ to the logarithmic Hilbert space specifies a proper monomorphism to $X\times B$. 
	\begin{maintheorem}\label{thm:Hilbint}
		The logarithmic Hilbert space is the moduli stack whose fibre over a scheme $B$ is the groupoid of equivalence classes of proper monomorphisms $Z \rightarrow X\times B$ such that the composition $$Z \rightarrow X\times B \rightarrow B$$ is logarithmically flat and integral. See Definition~\ref{defn:surjrelsheaves} for the equivalence relation.
	\end{maintheorem}
 The logarithmic Quot space is not invariant under logarithmic modification. However, given a logarithmic modification $\pi:\tilde{X} \rightarrow X$ and $\mathcal{E}$ a sheaf on $X$ there is a logarithmic modification $$\mathrm{Quot}(\tilde{X},\pi^\star\mathcal{E})\twoheadrightarrow\mathrm{Quot}({X},\mathcal{E}).$$

 \subsection{Connection to literature}\label{sec:ConnectionLiterature}  
In addition to directly generalizing work of Li--Wu, the logarithmic Quot scheme captures and provides a uniform framework for a number of moduli spaces that have appeared in the literature. In some of these cases, it would be worthwhile to elucidate the precise connection, but we record them here to illustrate the range of possible geometric applications:

\begin{enumerate}[(i)]
    \item \textbf{Logarithmic Gromov--Witten spaces.} There is a rich precedent for using logarithmic geometry to understand moduli spaces and their invariants. Originally developed in the context of Gromov--Witten theory \cite{MR3166394,gross2012logarithmic, Li_stablemorphisms, NishSieb06,Ran19}. Understanding logarithmic Hilbert spaces gives an alternative route to discovering logarithmic mapping spaces. In fact, the logarithmic mapping spaces constructed by Wise \cite{MR3519093} can be viewed as subfunctors of a logarithmic Hilbert scheme, by a graph construction.
    
    \item \textbf{Hyperplane arrangements and del Pezzo surfaces.} Hacking--Keel--Tevelev study two related moduli spaces of KSBA type which are equal to (or at least closely related to) logarithmic Hilbert schemes; see \cite{HKT2006,HKT2009}. In each case one studies pairs $(X,D)$ where $X\backslash D$ admits a canonical torus embedding, so can be studied via the logarithmic Hilbert scheme of a toric variety.

    \item \textbf{Logarithmic Grassmannians.} The logarithmic Grassmannian \cite{LogGrass} is a component of the logarithmic Quot space and appears to coincide with a moduli space studied in Lafforgue's influential work on Grassmannian degenerations \cite{Lafforgue2003}. The relevant substack of $\supplog$ arises from the combinatorial geometry of the Dressian \cite{MaclaganSturmfels}.

    \item \textbf{Logarithmic Hilbert schemes in Donaldson--Thomas theory.} In the special case of the logarithmic Hilbert scheme of dimension one ideal sheaves in a threefold, models are studied in logarithmic Donaldson--Thomas theory \cite{MR20}. The theory gives degeneration formulae, and ultimately a new approach to the MNOP conjectures \cite{MR23}. 
    
    \item \textbf{Degenerations for classical constructions of hyperk\"ahler varieties}. Logarithmic Hilbert schemes of points, first constructed in \cite{MR20}, and later studied by Tschanz via GIT techniques \cite{Tschanz2023,Tschanz2024}, are expected to give a robust method for constructing good degenerations for hyperk\"ahler varieties of Kummer and $\text{K3}^{[n]}$ type.

    \item \textbf{Toric hypersurfaces and the secondary polytope.} The image of the logarithmic Hilbert scheme of hypersurfaces in a toric pair inside the stack of tropical supports recovers the famous secondary polytope of Gelfand--Kapranov--Zelevinsky; see \cite{KH20}. In particular, the logarithmic linear system is a toric stack.  
\end{enumerate}

Beyond these examples, the representability properties of the logarithmic Quot space connect to themes in \cite{MolchoWise,MR3962244}. The stack of expansions was introduced in \cite{MR20}, and its geometry further developed in \cite{CarocciNabijou1,CarocciNabijou2}. In particular, \cite[Theorem~A / Theorem~1.8]{CarocciNabijou1} motivates why our moduli space has finite stabilisers, which is part of the motivation for the notion of tropical support. 

Proper monomorphisms in the category of logarithmic schemes play a role in formulating logarithmic intersection theory \cite{LogChow,HerrLogProduct}.

Finally, we observe that the logarithmic Quot space can be defined for a more general class of logarithmic schemes than the normal crossing case. In the case that $(X,D)$ admits a strict and flat map to a smooth Artin fan - the \textit{regular crossing} situation - the construction follows without further work. Indeed by \cite[Proof of Proposition 4.1.2]{Quot2}, every regular crossing pair can be obtained as a closed subscheme inside a toric variety with log structure a subset of the toric log structure. Thus the general case can be deduced from the (toric) simple normal crossings case. While we do not presently have an application in mind, we record this observation for possible future use. 

After a draft of this paper was made public, the author was made aware of independent work on an alternative approach to constructing a logarithmic Hilbert space \cite{SiebTalpoThomas}.

\subsection{Future work} Our constructions suggest a definition of a logarithmic coherent sheaf: any surjection that arises in a Quot scheme. However, a more complete study of the deformation theory and moduli of these objects is an interesting challenge. Once the right notion has been found, the interaction between logarithmic coherent sheaves and the stability conditions of Gieseker and Bridgeland is likely to be important. In the classical situation, Gieseker stability is essentially extracted from the Quot scheme using GIT. We hope that the logarithmic Quot space will shed light on these directions. A related challenge is to explore the precise connection between the logarithmic Quot space and logarithmic Picard group \cite{MolchoWise}. The hope is to use a version of GIT to construct logarithmic versions of other moduli spaces which are not yet understood. 

Recent developments permit the study of enumerative geometry by thinking about Quot schemes associated to locally free sheaves on surface geometries\cite{MR4372634}. The existence of a virtual fundamental class suggests these enumerative theories can be understood by degenerating the target and studying the logarithmic Quot space of the resulting special fibre. Enumerative geometry of moduli of higher rank sheaves provides a further direction \cite{JoyceSong}.

The \textit{derived Quot scheme} and \textit{derived Hilbert scheme} \cite{DerivedQuot} fit into the program of understanding derived versions of moduli spaces. One can attempt to define a derived logarithmic Quot space, fitting our construction in with this story. The result may have fruitful applications to enumerative geometry.  

	\subsection{Acknowledgements}
	The author would like to express sincere gratitude to his supervisor Dhruv Ranganathan for numerous helpful conversations. The author learned a great deal from numerous conversations with Sam Johnston, Navid Nabijou, Thibault Poiret and Martin Ulirsch. He also benefited greatly from conversations with Dan Abramovich, Francesca Carocci, Robert Crumplin, Samouil Molcho, Bernd Siebert, Calla Tschanz and Jonathan Wise. Martin is owed thanks for informing the author of Tevelev's unpublished proof of Proposition~\ref{prop:GenTev} and recognition for suggesting the term \textit{tropical support}. Bernd is thanked for comments and questions on a first draft of this paper. This work appeared in the PhD thesis of the author, and he thanks the examiners Mark Gross and Martin Ulirsch for numerous comments which have improved this paper.

	\section{Piecewise linear geometry}\label{sec:PLSpaces}
	In this section we generalise the theory of cone complexes to incorporate non--convex cones, as outlined in Section~\ref{sec:PLcomplexes}. The contents of this section are thus parallel to the theory of cone complexes, see \cite{ModStckTropCurve}. By tropical geometry we mean the geometry of rational polyhedral cones.
	\subsection{Cones}\label{sec:backgroundconecx} We refer the reader to \cite[Section 2]{ModStckTropCurve} for the definition of the category \textbf{RPC} of rational polyhedral cones and the category \textbf{RPCC} of rational polyhedral cone complexes. There is a fully faithful embedding of categories $$ \textbf{RPC} \hookrightarrow \textbf{RPCC}.$$
	
	A cone complex $\Sigma$ has an associated topological space called the \textit{topological realisation} which we denote $|\Sigma|$. A morphism of cone complexes $f:\Sigma' \rightarrow \Sigma$ induces a morphism of topological spaces  $$|f|:|\Sigma'| \rightarrow |\Sigma|$$ called the \textit{topological realisation of $f$}.

 	A \textit{subdivision} of cone complexes is a morphism $$f:\Sigma_1 \rightarrow \Sigma_2$$ of cone complexes such that $|f|$ is an isomorphism of topological spaces and induces a bijection of lattice points.
  
	Let $\sigma$ be a cone and $\Sigma,\Theta$ cone complexes. Consider a morphism $\Sigma \xrightarrow{h} \Theta\times \sigma \rightarrow \sigma$ where $h$ is a subdivision. Such a composition is said to be \textit{combinatorially flat} if the image of each cone of $\Sigma$ is a cone of $\sigma$. We do not impose any condition on the lattice in our definition of combinatorial flatness and so a combinatorially flat morphism of fans need not correspond to a flat map of toric varieties. 
 
 We consider the category of rational polyhedral cones \textbf{RPC} a site by declaring the inclusion of any face to be an open morphism. We call this Grothendieck topology the \textit{face topology}. Similarly define the face topology on \textbf{RPCC}.
 
	In the remainder of this section we introduce the category \textbf{PLS} of piecewise linear spaces. To do so we define piecewise linear complexes by gluing piecewise linear cones. The situation is analagous to affine schemes, schemes and algebraic spaces, as described in Table \ref{table:analogy}.
	
	\begin{table}[ht] 
		\begin{tabular}{c|c|c}
			affine variety    & polyhedral cone         & piecewise linear cone         \\
			algebraic variety &  polyhedral cone complex & piecewise linear cone complex \\
			algebraic space   & cone space                       & piecewise linear space        \\
			algebraic stack   & cone stack                       & piecewise linear stack       
		\end{tabular}
		\caption{}\label{table:analogy}
	\end{table}

	\subsection{Piecewise linear cone complexes}\label{sec:piecewiselinspaces} The local model for a piecewise linear complex is a piecewise linear cone. To make progress we make an auxilliary definition.
	
	\begin{definition}
		A \textit{local cone} $\mathfrak{s}$ of dimension $k$ is a pair $(N_\mathfrak{s},U_\mathfrak{s})$ consisting of a finitely generated rank $k$ abelian group $N_\mathfrak{s}$ and a connected open subset $U_\mathfrak{s}$ of $N_\mathfrak{s}\otimes \mathbb{R} \cong \mathbb{R}^k$ such that there is a fan on $\mathbb{R}^k$ in which $U_\mathfrak{s}$ is the union of interiors of cones. For local cones $\mathfrak{s}_1,\mathfrak{s}_2$ a morphism $\iota:\mathfrak{s}_1 \rightarrow \mathfrak{s}_2$ of local cones is a monoid morphism $$N_{\mathfrak{s}_{1}}\rightarrow N_{\mathfrak{s}_{2}}$$ inducing a map from $U_{\mathfrak{s}_1}$ to the closure of $U_{\mathfrak{s}_2}$.
	\end{definition}
 
        The interior of any rational polyhedral cone defines a local cone and any morphism of cones induces a morphism of local cones. Not all local cones arise in this way, indeed local cones need not be convex. In the sequel the closure of a subset $\kappa$ of a topological space will be denoted $\overline{\kappa}$ with the ambient topological space understood from context. 
	
	\begin{definition}\label{defn:PLspace}
		
		A \textit{piecewise linear complex} $\mathscr{S}$ is a tuple $(|\mathscr{S}|,\mathcal{P}_\mathscr{S},{A}_\mathscr{S},{C}_\mathscr{S})$ consisting of the following data.
		\begin{enumerate}[(1)]
			\item A topological space $|\mathscr{S}|$ equipped with a locally closed stratification $\mathcal{P}_\mathscr{S}$. 
			\item The set ${A}_\mathscr{S}$ consists of a homeomorphism for each $\kappa$ written $$f_\kappa:\overline{\kappa}\rightarrow \overline{U}_{\mathfrak{s}_\kappa}$$ where $\mathfrak{s}_\kappa = (N_{\mathfrak{s}_\kappa},{U}_{\mathfrak{s}_\kappa})$ is a local cone and we use a bar to denote closure. This homeomorphism identifies the closure of the stratum $\kappa$ with the closure of ${U}_{\mathfrak{s}_\kappa}$ in $N_{\mathfrak{s}_\kappa}\otimes \mathbb{R}$. We require the restriction of $f_\kappa$ to $\kappa$ defines a homeomorphism to ${U}_{\mathfrak{s}_\kappa}$.
			\item The set ${C}_\mathscr{S}$ consists of morphisms of local cones $$g_{\kappa',\kappa}:\mathfrak{s}_{\kappa'} \rightarrow \mathfrak{s}_\kappa$$ for each pair $\kappa,\kappa'$ in $\mathcal{P}_\mathscr{S}$ with $\kappa'$ a subset of $\overline{\kappa}$. We require the topological realisation of $g_{\kappa',\kappa}$ is compatible with the inclusion of $\kappa'$ as a subset of $\overline{\kappa}$. Moreover, whenever $\kappa'' \subset \overline{\kappa}'\subset \overline{\kappa}$ we have $$g_{\kappa,\kappa'}\circ g_{\kappa',\kappa''} = g_{\kappa,\kappa''}.$$
		\end{enumerate} 
		A morphism of piecewise linear complexes $\varphi:\mathscr{S}_1 \rightarrow \mathscr{S}_2$ consists of the following data.
		\begin{enumerate}[(1)]
			\item A continuous map 
			$$|\varphi|:|\mathscr{S}_1|\rightarrow |\mathscr{S}_2|$$ such that for each $\kappa_1$ in $\mathcal{P}_{\mathscr{S}_1}$ there is some $\kappa_2$ in $\mathcal{P}_{\mathscr{S}_2}$ containing $|\varphi|(\kappa_1)$.
			\item For each pair $\kappa_1,\kappa_2$ with $|\varphi|(\kappa_1)\subset \kappa_2$ a monoid morphism $\varphi_{\kappa_1,\kappa_2}:N_{\mathfrak{s}_{\kappa_1}} \rightarrow N_{\mathfrak{s}_{\kappa_2}}$ inducing a map of local cones $\mathfrak{s}_{\kappa_1}\rightarrow \mathfrak{s}_{\kappa_2}$. The $\varphi_{\kappa_1,\kappa_2}$ must be compatible with passing to closures.
		\end{enumerate}
	\end{definition}
	The reader is issued with three warnings. First, a local cone is not a piecewise linear complex - this follows from the definition of $A_\mathscr{T}$. Second, the category of piecewise linear complexes is not simply obtained by inverting subdivisions in \textbf{RPCC.} Finally, the definition of stratification states that whenever $\kappa,\kappa'$ strata of $\mathcal{P}_\mathscr{S}$ if $\kappa'$ intersects the closure $\overline{\kappa}$ of $\kappa$ then $\kappa'$ is contained in $\overline{\kappa}$. The category of piecewise linear cone complexes is denoted \textbf{PLCC}.

 \begin{definition}
    A \textit{piecewise linear cone} is a piecewise linear space $\mathscr{S}$ such that $\mathcal{P}_\mathscr{S}$ contains a dense subset of $|\mathscr{S}|$.
\end{definition}
 
	\begin{example}\label{example: cone complex as PL space}
		A cone complex $\Sigma$ specifies a piecewise linear complex $\mathscr{S}$. Writing $|\sigma|^o$ for the interior of the topological realisation of a cone $\sigma$, we specify $$|\mathscr{S}|=|\Sigma| \textrm{   and   } P_\mathscr{S} = \left\{|\sigma|^o: \sigma \textrm{ a cone of }\Sigma\right\}.$$ The closure of the stratum corresponing to $\sigma^o$ is identified with $|\sigma|$ to give the set $A_\mathscr{S}$ and the set $C_\mathscr{S}$ is defined by face inclusions. This assignment is functorial, giving a fully faithful embedding of the category of cone complexes into the category of piecewise linear complexes. 
	\end{example}

In this way we obtain fully faithful embeddings of categories 
$$
\begin{tikzcd}
\mathbf{RPC} \arrow[d] \arrow[r] & \mathbf{PLC} \arrow[d] \\
\mathbf{RPCC} \arrow[r]          & \mathbf{PLCC}.         
\end{tikzcd}$$
We say a piecewise linear complex $\Sigma$ is a \textit{cone complex} if $\Sigma$ lies in the image of the embedding of Example \ref{example: cone complex as PL space}. We say a morphism $\varphi:\mathscr{T}_1\rightarrow \mathscr{T}_2$ of piecewise linear complexes is a \textit{subdivision} if $|\varphi|$ is an isomorphism of topological spaces; the maps of local cones $\varphi_{\kappa_1,\kappa_2}$ induce a bijection between lattice points in $\kappa_1$ and lattice points in $\kappa_2 \cap |\varphi|(\kappa_1)$, and the preimage under $|\varphi|$ of every stratum of $\mathcal{P}_{\mathscr{T}_2}$ is a finite union of elements of $\mathcal{P}_{\mathscr{T}_1}$. It follows that a morphism of cone complexes is a subdivision if and only if it is a subdivision in the usual sense. We say that a subdivision is a \textit{tropical model} if the domain is a cone complex. 
	
	\subsubsection{Subdivisions from conical stratifications} For $\mathscr{S}$ a piecewise linear space, we say a locally closed stratification $\mathcal{P}$ of $|\mathscr{S}|$ refining $\mathcal{P}_\mathscr{S}$ is \textit{conical} if there is a tropical model $$\Sigma\rightarrow \mathscr{S}$$ such that each stratum of $\mathcal{P}$ is the union of interiors of images of cones of $\Sigma$. 
 
 \begin{proposition}\label{prop:} There is an initial subdivision of piecewise linear spaces $$\mathscr{S}(\mathcal{P})\rightarrow \mathscr{S}$$ such that given any tropical model $\Sigma \rightarrow \mathscr{S}$ for which each stratum of $\mathcal{P}$ is a union of interiors of cones of $\Sigma$, we may factor $$\Sigma \rightarrow \mathscr{S}(P)\rightarrow \mathscr{S}.$$ 
 \end{proposition}
	
	\begin{construction}\label{cons:}
		We define $\mathscr{S}(\mathcal{P})$. Set $|\mathscr{S}(\mathcal{P})| = |\mathscr{S}|$ and let $\sim$ be the equivalence relation on points of $|\mathscr{S}(\mathcal{P})|$ generated by $p \sim q$ whenever there is a tropical model $$\Sigma \rightarrow \mathscr{S}$$ such that $p$ and $q$ lie in the interior of the same cone. We obtain a locally closed stratification $\mathcal{P}_{\mathscr{S}(\mathcal{P})}$ of $|\mathscr{S}|$ by declaring two points to be the same stratum if they are related by $\sim$. Since each stratum of $\mathcal{P}_{\mathscr{S}(\mathcal{P})}$ is an open subset of a linear subspace of a local cone, each stratum inherits the structure of a local cone. The sets $\mathcal{A}_{\mathscr{S}(\mathcal{P})}$ and $\mathcal{C}_{\mathscr{S}(\mathcal{P})}$ are inherited from $\mathscr{S}$.
	\end{construction}
To prove Proposition \ref{prop:} it suffices to verify that Construction \ref{cons:} yielded a valid piecewise linear space. The universal property is then clear. The key claim to check is that each stratum of $\mathcal{P}_{\mathscr{S}(\mathcal{P})}$ is an open subset of a linear subspace of a local cone. 
 
\begin{proof}[Proof of Proposition \ref{prop:}] A definition will help us. Suppose we are given a conical locally closed stratification $\mathcal{Q}$ of $U_\mathfrak{s}$ for some local cone $(N_\mathfrak{s},U_\mathfrak{s})$ and a ray $\rho$ inside a stratum $\kappa$ of $\mathcal{Q}$. We think of $\kappa$ as a subset of $U_\mathfrak{s}$. We define a subgroup $G_\rho$ of $N_\mathfrak{s}$ to be the intersect of (necessarily saturated) subgroups $G_i\leq N_\mathfrak{s}$ maximal with the following property. There is an open subset $V_i$ of $U_\mathfrak{s}$ containing $\rho$ such that $G_i\otimes \mathbb{R}\cap V_i$ lies inside $\kappa\cap V_i$.

To understand why each stratum of $\mathcal{P}_{\mathscr{S}(\mathcal{P})}$ is an open subset of a linear subspace one can use the following two facts. \begin{enumerate}[(1)]
\item Each stratum of $\mathcal{P}_{\mathscr{S}(\mathcal{P})}$ is a union of interiors of cones of any tropical model $\Sigma \rightarrow \mathscr{T}$ provided the image of every cone of $\Sigma$ lies within a stratum of $\mathcal{P}$.

\item Suppose lattice points $p,q$ lie in the same stratum of $\mathcal{P}_{\mathscr{S}(\mathcal{P})}$. Necessarily $p,q$ are mapped to the same local cone $(N_\mathfrak{s},U_\mathfrak{s})$ in $\mathscr{S}$ and specify rays $\rho_p,\rho_q$ inside $N_\mathfrak{s}$.  Then we have an equality of subgroups of $N_\mathfrak{s}$ $$G_{\rho_q}= G_{\rho_p}.$$ 
 \end{enumerate}
 The first fact is clear. The second fact follows because if there is a tropical model  $\Sigma$ in which $p$ and $q$ lie in the interior of the same cone then the cones of $\Sigma$ which contain $p$ and $q$ are the same.

 Armed with these facts we will show a stratum $\kappa$ of $\mathcal{P}_{\mathscr{S}(\mathcal{P})}$ is an open subset of a linear subspace of a local cone. Fix a tropical model $\Sigma$ and choose a face $\gamma$ maximal with the property that $\gamma$ is a subset of $\kappa$. The aforementioned local cone is $(N_\kappa,U_\kappa)$ and the linear subspace $$\gamma^\mathrm{gp} \otimes \mathbb{R} \subset U_\kappa.$$ If $\kappa$ did not lie within this linear subspace then we could find $p,q\in \kappa$ in the interior of the same cone of a tropical model $\Sigma'$ and fact (2) would force $\gamma$ not to be maximal. Fact 2 also forces openness. 
\end{proof}

\subsubsection{Examples of piecewise linear spaces}	We record two examples of piecewise linear spaces that arise in nature.
	\begin{example}\label{Ex:ConeOverEx}
		Consider the cone $\sigma$ in $\mathbb{R}^3$ with rays $$\{(1,0,0), (0,1,0), (0,0,1)\}.$$ Define a piecewise linear complex which is a subdivision of $\sigma$ obtained by replacing the dense stratum $\sigma^o$ with two strata: the ray in direction $(1,1,1)$ and its complement in $\sigma^o$. See Figure \ref{fig:KeyExample2}. This is piecewise linear subdivision of a certain cone in the fan of $\mathbb{P}^4$ which appears when taking the tropical support in Example~\ref{ex:keyexplspace}.
        
        \begin{figure}[ht]
			\centering
            \begin{tikzpicture}
                \begin{scope}[xshift=0cm]
                \coordinate (A) at (0,0);
                \coordinate (B) at (3,0);
                \coordinate (C) at (0,3);
                
                \fill[blue!20] (A) -- (B) -- (C) -- cycle;
                \draw[black, thick] (A) -- (B) -- (C) -- cycle;
                
                \foreach \pt in {A,B,C} {
                  \filldraw[black] (\pt) circle (2pt);
                }
            
                \coordinate (G) at ($ (A)!.333!(B)!0.333!(C) $);
                \filldraw[black] (G) circle (2pt);
              \end{scope}
              
              \begin{scope}[xshift=5cm]
                \coordinate (P) at (0,0);
                \coordinate (Q) at (3,0);
                \coordinate (R) at (0,3);
                
                \fill[blue!20] (P) -- (Q) -- (R) -- cycle;
                \draw[black, thick] (P) -- (Q) -- (R) -- cycle;
            
                \foreach \pt in {P,Q,R} {
                  \filldraw[black] (\pt) circle (2pt);
                }
            
                \coordinate (M) at ($ (P)!0.5!(Q) $);
                \filldraw[black] (M) circle (2pt);
              \end{scope}
            \end{tikzpicture}
			\caption{Examples of a piecewise linear complexes can be obtained by taking the cones over the above (subdivided) triangles. Left is piecewise linear space discussed in Examples \ref{ex:keyexplspace} and \ref{Ex:ConeOverEx}. Right gives a different example. The only examples of piecewise linear complexes of dimension two which are subdivisions of cone complexes are cone complexes.}\label{fig:KeyExample2}
		\end{figure}
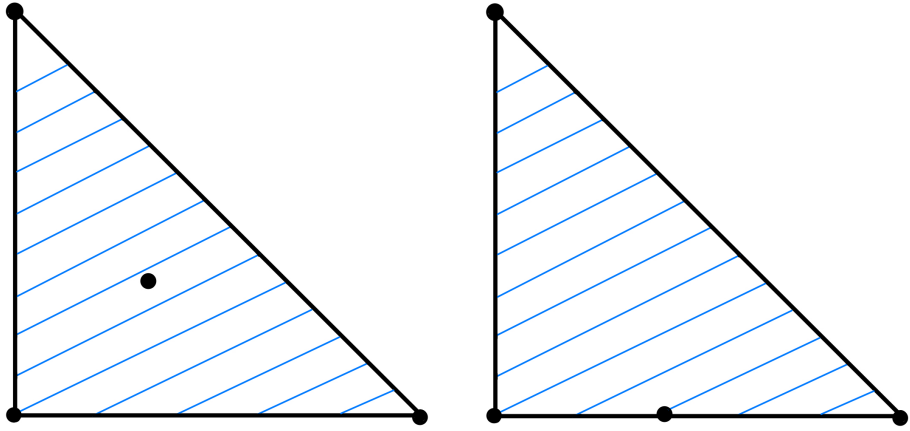
		
	\end{example}

\begin{example}[Tropical moduli spaces of points are naturally piecewise linear spaces]\label{ex:Not a cone cx}
    Let $\sigma$ be the (rational polyhedral) cone of dimension two with maximal dimension face $\mathbb{R}^2_{\geq0}$. This is simply the fan of $\mathbb{A}^2$. The space of two labelled points $p,q$ in $\sigma$ bijects with points $(p_1,p_2,q_1,q_2)\in \mathbb{R}^4_{\geq 0}$. The first two coordinates specify the location of $p$ and the second two give the location of $q$. 
    
    We define a piecewise linear structure on $\mathbb{R}^4_{\geq 0}$ by specifying a conical locally closed stratification of $\mathbb{R}^4_{\geq 0}$. Two points $(p_1,p_2,q_1,q_2)$ and $(p_1',p_2',q_1',q_2')$ lie in the same stratum if and only if:
    \begin{enumerate}[(1)]
    \item The points $(p_1,p_2)$ and $(p_1',p_2')$ lie in the interior of the same face of $\sigma$. Similarly replacing $p$ by $q$.
    \item We require $(p_1,p_2) = (q_1,q_2)$ if and only if $(p_1',p_2') = (q_1',q_2')$.
    \end{enumerate}
    The resulting locally closed stratification is readily seen to be conical. The output of Construction~\ref{cons:} is not a cone complex. Indeed, this piecewise linear space has a stratum $\kappa$ of dimension two corresponding to the locus where $p=q$ lie in the maximal dimension face of $\sigma$. This stratum does not lie in the closure of any stratum of dimension three. Since $\kappa$ lies in the closure of a stratum of dimension four, we do not have a cone complex.
\end{example}

\subsection{Base change and flatness for piecewise linear complexes} Let $\Theta$ be a cone complex. 
\subsubsection{Fibre products} Proposition \ref{prop:fibreproductPLspaces} is closely related the discussion in \cite[Section 2.2]{molcho2019universal}.
\begin{proposition}\label{prop:fibreproductPLspaces}
    Fibre products exist in the category of piecewise linear complexes.
\end{proposition}
To parse the proof it will help to note that we may characterise a piecewise linear complex as a colimit over piecewise linear cones glued along inclusions of piecewise linear cones.
\begin{proof}
The statement is local in the face topology so it suffices to consider a pre-fibre square of piecewise linear cones $$ 
\begin{tikzcd}
                        & \mathscr{T}_1 \arrow[d] \\
\mathscr{T}_2 \arrow[r] & \mathscr{T}_3.          
\end{tikzcd}$$
Write $\mathfrak{s}_i = (N_{\mathfrak{s}_i},U_{\mathfrak{s}_i})$ for the maximal local cone of $\mathscr{T}_i$. Thus $\mathscr{T}_i$ is the data of a conical locally closed stratification of a closed subset of $N_{\mathfrak{s}_i}\otimes \mathbb{R}$. There is an associated pre-fibre square of monoids whose fibre product is $$N_{\mathfrak{s}_1}\times_{N_{\mathfrak{s}_3}} N_{\mathfrak{s}_2} = N.$$ Taking the common refinement of the pullbacks of the locally closed stratifications $\mathcal{P}_{\mathscr{T}_1}$ and $\mathcal{P}_{\mathscr{T}_2}$ defines a conical locally closed stratification of $N\otimes \mathbb{R}$ and thus a piecewise linear cone. Note any conical locally closed stratification pulls back to a conical locally closed stratification: the linear (in)equalities cutting each stratum out pull back to linear (in)equalities. It is easy to see this piecewise linear cone is the fibre product. 
\end{proof}
\subsubsection{Combinatorial flatness for piecewise linear spaces} To make sense of combinatorial flatness for morphisms of piecewise linear spaces we require Construction~\ref{cons:BddPLSpace}. 

Let $\sigma$ be a cone, $\Theta$ a cone complex and consider the cone complex $\sigma \times \Theta$. Let $\mathcal{P}^o$ be a locally closed stratification of the open subset $|\sigma|^o \times |\Theta|$ which refines the locally closed stratification pulled back from $\Theta$ and such that there exists a subdivision of $\sigma\times \Theta$ for which each stratum of $\mathcal{P}^o$ is a union of interiors of cones.
 
	\begin{construction}\label{cons:BddPLSpace}
		 We define a locally closed stratification $\mathcal{P}$ of $|\sigma \times \Theta|$. Two points lie in the same locally closed stratum of a locally closed stratification $\mathcal{P}'$ if they lie in the interior of the same cone of $\sigma \times \Theta$ and they lie in the closure of precisely the same strata of $\mathcal{P}^o$. Define the final stratification $\mathcal{P}$ by replacing strata of $\mathcal{P}'$ with their connected components. 
	\end{construction}
	Observe there is a cone complex structure on $\sigma \times \Theta$ such that each stratum of $\mathcal{P}$ is a union of cones. We need a notion of continuity for families of piecewise linear spaces. Our definition generalises the definition of combinatorial flatness, discussed for example in \cite[proof of Lemma~3.3.5]{MR20}. 
	
	\begin{definition}\label{defn:CombFlat}
		Fix $\sigma$ a rational polyhedral cone and $\Theta$ a cone complex. Consider a subdivision $\mathscr{T}$ of $\Theta \times \sigma$ and denote the composition $$ p:\mathscr{T}\rightarrow \Theta \times \sigma\rightarrow \sigma.$$
		\begin{enumerate}[(1)]
			\item We say $p$ satisfies axiom F1 if for every stratum $\kappa$ in $\mathcal{P}_\mathscr{T}$ the image $|p|(\kappa)$ lies in $\mathcal{P}_{\sigma}$.
			\item Consider faces $\tau_1\leq \tau_2$ of $\sigma$ and denote the restriction of $p$ to the preimage of $\tau_i$ by $$p_i: \mathscr{T}_i\rightarrow \Theta\times \tau_i\rightarrow \sigma.$$ The restriction of $\mathcal{P}_{\mathscr{T}_i}$ to the preimage of the interior of $\tau_i$ defines a locally closed stratification $\mathcal{P}_{\mathscr{T}_i}^o$. We say $p$ satisfies axiom F2 if for all choices of $\tau_1 \leq \tau_2$ the output of $\mathcal{P}_{\mathscr{T}_2}^o$ under Construction \ref{cons:BddPLSpace} restricts on the preimage of the interior of $\tau_1$ to $\mathcal{P}_{\mathscr{T}_1}^o$. 
		\end{enumerate}    
		We say $p$ is \textit{combinatorially flat} if it satisfies axioms F1 and F2.
	\end{definition}
 
Note if $\mathscr{T}$ is a cone complex then axiom F1 simply asks the image of each cone is a cone and axiom F2 is automatic. We simplify checking axiom F2 with a lemma.

\begin{lemma}\label{lem:F2}
    Assume $$p:\mathscr{T}\rightarrow \Theta \times \sigma\rightarrow \sigma$$ satisfies axiom F1. Assume moreover that axiom F2 holds for $\tau_2 = \sigma$ and any choice of $\tau_1$. Then $p$ is combinatorially flat.
\end{lemma}
In the sequel for $\tau$ a face of ${\sigma}$ we write $\mathcal{P}_\mathscr{T}(\tau)$ for the subset of $\mathcal{P}_\mathscr{T}$ mapped under $p$ to the interior of $\tau$.
\begin{proof}
    If the lemma were false then we could find a pair of faces $\tau_1 \leq \tau_2$ of $\sigma$ and strata $\kappa,\kappa'\in \mathcal{P}_\mathscr{T}(\tau_1)$ such that every stratum $\kappa''\in \mathcal{P}_{\mathscr{T}}(\tau_2)$ contains $\kappa$ in its closure if and only if it contins $\kappa'$ in its closure. 

    Since F2 holds for $\tau_2 = \sigma$ and $\tau_1$ there exists a stratum  $\hat{\kappa}$ in $\mathcal{P}_{\mathscr{T}}(\sigma)$ containing $\kappa'$ in its closure but not containing $\kappa$ (swapping $\kappa,\kappa'$ if needed). Write $\overline{\kappa}$ for the closure of $\hat{\kappa}$ in $\Theta \times \sigma$. We will show that $\overline{\kappa}\cap p^{-1}(|\tau_2^o|)$ contains $\kappa'$ in its closure. Note $\kappa$ cannot lie in the closure of $\overline{\kappa}\cap p^{-1}(|\tau_2^o|)$ as it does not lie in $\overline{\kappa}$. This completes the proof because certainly $\overline{\kappa}\cap p^{-1}(|\tau_2^o|)$ is a union of strata so by the definition of piecewise linear space, $\kappa'$ must lie entirely within the closure of one such stratum.  

    It remains to show that the closure in $\Theta \times \sigma$ of $\overline{\kappa}\cap p^{-1}(|\tau_2^o|)$ contains $\kappa'$. Without loss of generality replace $\hat{\kappa}$ with a stratum of $\mathcal{P}_\mathscr{T}(\sigma)$ lying in the closure of $\hat{\kappa}$ which is minimal with the property that its closure contains $\kappa'$. We think of $\hat{\kappa}$ as a subset of $\mathbb{R}^n$ for some $n$, obtained by removing polyhedral cones $\alpha_1,...,\alpha_n$ from the intersect of a polyhedral cone $\alpha_0$ with $p^{-1}(|\sigma|^o)$. Note $\alpha_0$ necessarily surjects to $\sigma$ and since axiom F2 holds for morphisms of cones we know $\kappa$ lies in the closure of $\alpha_0 \cap p^{-1}(|\tau_2^o|)$. To complete the proof note minimality of $\hat{\kappa}$ means $\kappa'$ does not lie in the closure of $\alpha_i$ for any $i>0$. We are done because $\kappa'$ lies in the closure of $\alpha_0\cap |\tau_2|^o$, but not in the closure of any other $\alpha_i$.
\end{proof}
\subsubsection{Combinatorial flatness and base change}\label{CombFlatandBaseChange} Let $\tau \rightarrow \sigma$ be a morphism of rational polyhedral cones. Let $\mathscr{T}$ be a subdivision of $\sigma \times \Theta$ and consider the following diagram where all squares are cartesian. 
$$
\begin{tikzcd}
\mathscr{T}_\tau \arrow[d] \arrow[r] & \tau \times \Theta \arrow[d] \arrow[r] & \tau \arrow[d] \\
\mathscr{T} \arrow[r]                & \sigma \times \Theta \arrow[r]         & \sigma        
\end{tikzcd}
$$

 \begin{lemma}
 Assume $\mathscr{T}\rightarrow \sigma$ is combinatorially flat. Then $\mathscr{T}_\tau\rightarrow \tau$ is also combinatorially flat.
 \end{lemma}
 \begin{proof}
Strata of $\mathscr{T}_\tau$ biject with pairs $(\kappa,\kappa')\in \mathcal{P}_\mathscr{T}\times \mathcal{P}_\tau$ whose image in $\sigma$ lie in the same stratum. Axiom F1 follows from the fact that surjections are preserved under pullback in the category of sets. For axiom F2 first observe a stratum $(\kappa_1,\kappa_1')$ lies in the closure of $(\kappa_2,\kappa_2')$ if and only if $\kappa_1$ lies in the closure of $\kappa_2$ and $\kappa_1'$ lies in the closure of $\kappa_2'$. If axiom F2 were violated we could find faces $\tau_1\leq\tau_2$ of $\tau$ and strata $(\kappa,|\tau_1^o|), (\gamma,|\tau_1^o|)$ of $\mathscr{T}_\tau(\tau_1)$ which lie in the closure of precisely the same strata of $\mathscr{T}_\tau(\tau_2)$. Without loss of generality $(\kappa,|\tau_1^o|)$ lies in the closure of $(\gamma,|\tau_1^o|)$. We have two cases.

\textbf{Case 1.} If the interiors of $\tau_1,\tau_2$ are mapped to the interior of the same cone of $\sigma$ then $(\kappa, |\tau_2^o|)$ contains $(\kappa,|\tau_1^o|)$ in its closure but does not contain $(\gamma,|\tau_1^o|)$.

\textbf{Case 2} If the interiors of $\tau_1,\tau_2$ are mapped to the interior of cones $\sigma_1\leq \sigma_2$ respectively then we appeal to Lemma \ref{lem:F2} to find a stratum $\hat{\kappa}$ of $\mathscr{T}(\sigma_2)$ whose closure contains $\kappa$ but does not contain $\gamma$. Then $(\hat{\kappa}, |\tau_2^o|)$ contains $(\kappa,|\tau_1^o|)$ in its closure but does not contain $(\gamma,|\tau_1^o|)$. 
\end{proof}

We say a subdivision $\mathscr{T} \rightarrow \mathscr{G}\times \Theta\rightarrow \mathscr{G}$ is \textit{combinatorially flat} if pulling back any map from a cone $\sigma$ to $\mathscr{G}$ yields a combinatorially flat morphism in the sense of Definition \ref{defn:CombFlat}.
 \subsection{Piecewise linear spaces}\label{sec:IntroPLSpaces} In this subsection we develop the piecewise linear analogue of an algebraic space. We will use the theory of geometric contexts, see \cite[Section 1.1]{ModStckTropCurve} or \cite{simpson1996algebraic} for background. Geometric contexts are the minimal structure necessary to define a version of algebraic spaces. 
	
	We define a Grothendieck topology $\tau_{str}$ on the category of piecewise linear complexes generated by morphisms from piecewise linear cones $$\iota: \mathfrak{T} \rightarrow \mathscr{S}$$ which are isomorphisms to their images. We say a morphism $\varphi$ of piecewise linear spaces is \textit{strict} if $\varphi_{\alpha,\beta}$ is an isomorphism of local cones for all $\alpha,\beta$. We denote the set of strict morphisms by $\mathbb{S}$.
	
	\begin{lemma}
		The triple $(\textbf{PLCC},\tau_{str},\mathbb{S})$ forms the data of a geometric context.
	\end{lemma}
	
	\begin{proof}
		We check the axioms of a geometric context. To confirm \textbf{PLCC} contains all finite limits we appeal to Proposition \ref{prop:fibreproductPLspaces}. Disjoint unions of piecewise linear spaces are clearly piecewise linear spaces.
		
		The topology $\tau_{str}$ is subcanonical because piecewise linear cone complexes can be defined as a colimit over a diagram of local cones with all morphisms open.
		
		The definition of strict is closed under composition since the composition of isomorphisms is an isomorphism, and examining the construction of fibre product given above, being strict is preserved under base change. Isomorphisms and open maps are strict. 
		
		Finally strictness is a local property because any open neighbourhood covering a point $p$ contains the local cone in which $p$ lies.
	\end{proof}
	
	We now define a \textit{piecewise linear space} to be a geometric space in the context $(\textbf{PLCC},\tau_{str},\mathbb{S})$, see \cite[Definition 2.7]{ModStckTropCurve}. A morphism of piecewise linear spaces $$f:\mathscr{T}\rightarrow \mathscr{S}$$ is called a \textit{subdivision}/\textit{tropical model} if, given any morphism from a piecewise linear cone $\mathscr{S}'\rightarrow \mathscr{S}$, the pullback of $f$ is a subdivision/ tropical model of piecewise linear complexes.
	
	\subsection{Properties of piecewise linear spaces.} In this section we prove basic properties of piecewise linear spaces. 
 
 \subsubsection{Subdividing piecewise linear spaces} A \textit{subcomplex} of a piecewise linear complex $\mathscr{S}$ is a morphism of piecewise linear complexes $\mathscr{T}\rightarrow \mathscr{S}$ which is an isomorphism to its image. A subcomplex of a piecewise linear space is a morphism of piecewise linear spaces such that the base change along any morphism from a piecewise linear complex is a subcomplex in the above sense.
	\begin{proposition}\label{prop:PLsubdivisions}
		Every piecewise linear space $\mathscr{G}$ admits a tropical model $\Sigma\rightarrow \mathscr{G}$. Moreover if $\mathscr{G} \leq \mathscr{G}'$ is a subcomplex and $\Sigma \rightarrow \mathscr{G}$ a tropical model then there exists a tropical model $\Sigma'\rightarrow \mathscr{G}'$ such that the following diagram commutes:
		$$
		\begin{tikzcd}
		\Sigma' \arrow[r]             & \mathscr{G}'                \\
		\Sigma \arrow[r] \arrow[u, hook] & \mathscr{G} \arrow[u, hook]
		\end{tikzcd}
		$$
		In particular the vertical morphisms are both inclusions of subcomplexes.
	\end{proposition}

A basic example of a piecewise linear space is obtained by defining the free action of a finite group $G$ on a piecewise linear cone complex $\mathscr{T}_G$. This induces a groupoid presentation of $[\mathscr{T}_G/G]$. A \textit{quotient piecewise linear cone} is defined to be a piecewise linear space which may be written $[\mathscr{T}_G/G]$ such that the image of some stratum of $\mathcal{P}_{\mathscr{T}_G}$ is dense. Any piecewise linear space may be obtained by taking a colimit over a diagram of quotient piecewise linear cones where morphisms are inclusions of subcomplexes. This follows from the argument of \cite[TAG 0262]{stacks-project}.

	\begin{lemma}\label{lem:onecone}
		We prove the following local analogue of Proposition \ref{prop:PLsubdivisions}.
		\begin{enumerate}[(1)]
			\item Set $\mathscr{G}$ a quotient piecewise linear cone. There exists a tropical model $\Sigma\rightarrow \mathscr{G}$.
			\item Consider a subcomplex $$\Sigma' \hookrightarrow \mathscr{G}$$ where $\Sigma'$ is a cone complex. There is a tropical model $$\Sigma \rightarrow \mathscr{G}$$ extending the original embedding of cone complexes. 
		\end{enumerate}
	\end{lemma}

\begin{proof}[Proof of Proposition \ref{prop:PLsubdivisions}] 
The subdivision can be constructed recursively on the dimension of each quotient piecewise linear cone in our colimit, starting with cones of dimension zero (where no subdivision is needed). To handle cones of dimension $k$ subdivide each quotient piecewise linear cone using Lemma \ref{lem:onecone} without further subdividing any strata of dimension less than $k$. We avoid further subdivision of lower dimensional strata by including the previously constructed subdivision of these strata in the data of $\Sigma'$.
\end{proof}
 
\begin{proof}[Proof of Lemma \ref{lem:onecone}]
    We induct on the dimension of $\mathscr{G}$. For the base case note dimension zero cones are a single point and there is nothing to do. For the inductive step choose a presentation $$\mathscr{G}_G \rightarrow \mathscr{G} = [\mathscr{G}_G/G]$$ for $G$ a finite group. To specify a tropical model of $\mathscr{G}$ it suffices to specify a $G$ equivariant tropical model of $\mathscr{G}_G$. The subcomplex $\Sigma'\rightarrow \mathscr{G}$ pulls back to specify a G equivariant subcomplex $\Sigma'_G \rightarrow \mathscr{G}_G$. There is a subcomplex $\partial \mathscr{G}$ of $\mathscr{G}$ obtained as the quotient $\partial \mathscr{G}_G \rightarrow \partial \mathscr{G} = [\partial \mathscr{G}_G/G]$ where $\partial \mathscr{G}_G$ is the subcomplex of $\mathscr{G}_G$ obtained by discarding all strata of maximal dimension in $\mathcal{P}_{\mathscr{G}_G}$. Define $\partial \Sigma'$ as the preimage of $\partial \mathscr{G}$ in $\Sigma'$. By the inductive hypothesis find a subdivision $\partial \Sigma \rightarrow \partial \mathscr{G}$ extending $\partial \Sigma'.$ Write $\partial \Sigma_G$ for the $G$ equivariant subcomplex of $\mathscr{G}_G$ obtained by pullback.

    To complete the inductive step fix a maximal stratum of $\mathscr{G}_G$ and let $\mathscr{T}$ be the corresponding piecewise linear cone. Write $\Sigma_\mathscr{T}', \partial \Sigma_\mathscr{T}$ for the restriction to $\mathscr{T}$ of $\Sigma'_G, \partial \Sigma_G$ respectively. We think of $\mathscr{T}$ as a sudivision of $N_\mathfrak{t}\otimes \mathbb{R}$ where the local cone of the dense stratum of $\mathscr{T}$ is $(N_\mathfrak{t},U_\mathfrak{t})$. We may assume that $\Sigma_\mathscr{T}'$ contains $\partial \Sigma_\mathscr{T}$ as a subcomplex.

    We now have in particular the data of a cone complex $\Sigma_\mathscr{T}'$ embedded in $N_\mathfrak{t} \otimes \mathbb{R}$. Since every toric variety admits an equivariant compactification \cite{MR2248429}, we can extend this to a complete fan on $N_\mathfrak{t} \otimes \mathbb{R}$. Restricting this fan to the closure of $U_\mathfrak{t}$ defines a tropical model $\Sigma_\mathscr{T}\rightarrow \mathscr{T}$. This tropical model defines a subdivision of $\mathscr{G}_G$ by subdividing $\mathscr{T}$ and leaving the remainder unchanged. This subdivision is not equivariant, however applying the \textit{averaging trick} explained in \cite[Lemma~3.3.7]{MR20} we obtain an equivariant tropical model of $\mathscr{G}_G$. Since $\partial \Sigma_G$ and $ \Sigma'_G$ are $G$ equivariant, the averaging trick does not subdivide them.
\end{proof}
 \subsubsection{Sheaves on \textbf{RPC}} Piecewise linear spaces are a geometric way to understand a particular collection of sheaves on \textbf{RPC}. We denote the category of sheaves on $\textbf{RPC}$ with the face topology by $\textbf{Sh}(\textbf{RPC})$.
 
 \begin{proposition}\label{prop:PLasSheaves}
    There are fully faithful embeddings of categories
    $$ \textbf{RPCC} \hookrightarrow \textbf{PLS} \hookrightarrow \textbf{Sh}(\textbf{RPC}).$$
 \end{proposition}

 \begin{proof}
    \underline{\textbf{RPCC} to \textbf{PLS}.} The assignment of Example \ref{example: cone complex as PL space} defines a fully faithful embedding $$\textbf{RPCC} \rightarrow \textbf{PLS}.$$ Indeed Definition \ref{defn:PLspace} specialises to the definition of a morphism of cone complexes. The topology $\mathbf{\tau}_\mathrm{str}$ restricts to the face topology.
	
	\underline{\textbf{PLS} to \textbf{Sh}(\textbf{RPC}).} We denote the image of a piecewise linear space $\mathscr{S}$ under the Yoneda embedding by $H_\mathscr{S}$, a sheaf on \textbf{PLS}. We denote the restriction of $H_\mathscr{S}$ to $\textbf{RPC}$ $h_\mathscr{S}$. Since the topology on \textbf{RPC} is pulled back from the topology on the category of piecewise linear spaces, $h_\mathscr{S}$ is a sheaf. The assignment $$\mathscr{S} \rightarrow h_\mathscr{S}$$ defines our second map. We now verify this map is fully faithful.

		 Given two morphisms in \textbf{PLS} $$\varphi_1,\varphi_2: \mathscr{S}_1\rightarrow \mathscr{S}_2$$ there are tropical models $$\Sigma_1 \rightarrow \mathscr{S}_1, \Sigma_2 \rightarrow \mathscr{S}_2$$ for which there are commutative squares $$
		\begin{tikzcd}
		\mathscr{S}_1 \arrow[r, "\varphi_1"]       & \mathscr{S_2}      & \mathscr{S}_1 \arrow[r, "\varphi_2"]       & \mathscr{S}_2     \\
		\Sigma_1 \arrow[r, "\varphi_1'"] \arrow[u] & \Sigma_2 \arrow[u] & \Sigma_1 \arrow[u] \arrow[r, "\varphi_2'"] & \Sigma_2 \arrow[u]
		\end{tikzcd}.$$ Note $\varphi_1=\varphi_2$ if and only if $\varphi_1'=\varphi_2'$ and since the Yoneda embedding $\textbf{RPC}\hookrightarrow \textbf{Sh}(\textbf{RPC})$ is faithful, our functor is faithful. 
		
		We verify our embedding is full. Since tropical models are monomorphisms and the definition of morphism is local on cones, it suffices to consider a natural transformation $h_\sigma\rightarrow h_\mathscr{T}$ where $\sigma$ is a cone complex consisting of a single cone and its faces. The image of $h_\sigma(\sigma)$ defines the morphism from $\sigma$ to the piecewise linear space $\mathscr{T}$.
	\end{proof}

 \subsubsection{Prorepresentability and piecewise linear spaces}
	Proposition \ref{prop:PLasSheaves} implies the functor $h_\mathscr{S}$ is represented by a cone complex if and only if $\mathscr{S}$ is a cone complex. The next lemma shows in general a piecewise linear space is prorepresentable by cone complexes. We denote the system of tropical models of a piecewise linear complex $\mathscr{S}$ by $S_\mathscr{S}$.
	
	\begin{lemma}
		There is an equality of sheaves on \textbf{RPC} with the face topology $$\varinjlim_{\Sigma \in S_\mathscr{S}} h_\Sigma= h_\mathscr{S}.$$
	\end{lemma}
	\begin{proof}
		There is a compatible system of morphisms from $S_\mathscr{S}$ to $\mathscr{S}$ so the universal property of colimit defines a morphism $$\varinjlim_{S_\mathscr{S}}h_\Sigma \rightarrow h_\mathscr{S}.$$ To prove the lemma we write down the inverse morphism. Let $\sigma$ be the piecewise linear complex associated to an element of \textbf{RPC}. Any morphism $\sigma \rightarrow \mathscr{S}$ factors through a tropical model $\Sigma \rightarrow \mathscr{S}$ by Proposition \ref{prop:PLsubdivisions}. In this way a map from $\sigma$ to $\mathscr{S}$ defines an element of $h_\Sigma(\sigma)$. Composing with the canonical map $h_\Sigma \rightarrow \varinjlim_{\Sigma \in S_\mathscr{S}}h_\Sigma$ we have defined an element of $\varinjlim_{\Sigma \in S_\mathscr{S}}h_\Sigma(\sigma)$. The element is readily seen to be independent of the choice of $\Sigma$. In this way we define a morphism of sheaves $$h_\mathscr{S} \rightarrow \varinjlim_{\Sigma \in S_\mathscr{S}}h_\Sigma.$$ Our two maps are inverse.
	\end{proof}
 \section{The space of tropical supports}\label{Bigsec:TropSupp}
Logarithmic modifications of the logarithmic Quot space will be controlled by a piecewise linear space called the \textit{moduli space of tropical supports}.
Let $\Theta$ be a cone complex. 
	\begin{definition}\label{defn:FamilyTropSupp}
		A \textit{family of tropical supports} over a cone $\sigma$ is a subdivision $\mathscr{T}$ of $\Theta \times \sigma$ combinatorially flat over $\sigma$. 
	\end{definition}
Define the \textit{moduli space of tropical supports}, denoted $\supptrop(\Theta)$, to be the sheaf on $\textbf{RPC}$ assigning to a cone $\sigma$ the collection of families of tropical supports over $\sigma$. We imagine $\supptrop(\Theta)$ is a tropical version of a Hilbert scheme.
	
	\subsection{The moduli space of tropical supports is a piecewise linear space}
	We routinely confuse the space of tropical supports with its functor of points.
	
	\begin{theorem}\label{thm: Universal property Supptrop}
		There is a diagram of piecewise linear spaces $$
		\begin{tikzcd}
		\mathscr{X} \arrow[r] \arrow[rd] & \Theta \times \supptrop(\Theta) \arrow[d, "\varpi"] \\
		& \supptrop(\Theta)                                            
		\end{tikzcd}$$
		such that given a family of tropical supports on $\Theta$ over a rational polyhedral cone $\sigma$, say $$\mathscr{X}(\sigma) \rightarrow \Theta \times \sigma\rightarrow \sigma,$$ there is a unique morphism $\sigma \rightarrow \supptrop(\Theta)$ along which $\mathscr{X}$ pulls back to $\mathscr{X}_\sigma$.
	\end{theorem}
 
	\subsubsection{PL polyhedral structures.} Given a morphism of cone complexes $\Theta \times \sigma \rightarrow \sigma$ there is a correspondence between tropical models of $\Theta \times \sigma$ and polyhedral subdivisions of $\Theta$ with edge lengths metrised by $\sigma$. We now explain the piecewise linear analogue of this correspondence. 
 
 Consider a combinatorially flat piecewise linear subdivision $$\mathscr{S}\rightarrow \Theta \times \mathbb{R}_{\geq 0} \textrm{ and write }\pi_2:\Theta \times \mathbb{R}_{\geq 0} \rightarrow \mathbb{R}_{\geq 0}$$ for the second projection map. 
	
	\begin{definition}
		A \textit{PL polyhedral structure} of $\Theta$ is a locally closed stratification $\mathscr{T}_1$ of $|\Theta|$ such that there exists a subdivision $\mathscr{S}$ with notation in the previous paragraph such that $\mathcal{P}_\mathscr{S}$ pulls back along the inclusion $$|\Theta|\times \{1\}  \hookrightarrow |\Theta\times \mathbb{R}_{\geq 0}|$$ to $\mathscr{T}_1$.
	\end{definition}
	
	Observe a family of tropical supports on $\Theta$, written $$\mathscr{T}\xrightarrow{\pi} \Theta \times \sigma \rightarrow \sigma$$ specifies for each point $p$ of $|\sigma|$ a PL polyhedral structure $\mathscr{T}_p=\mathcal{P}_\mathscr{T} \cap |\pi^{-1}(p)|$. Points of the topological realisation of the piecewise linear space representing $\supptrop(\Theta)$ biject with PL polyhedral structures of $\Theta$. 

    Note if $\mathscr{T}$ is a cone complex then $\mathscr{T}_p$ is a polyhedral subdivision. A stratum of a PL polyhedral structure on $\Theta$ consisting of a single point is called a \textit{finite corner}. In the special case $\mathscr{T}$ is a polyhedral structure on $\Theta$, finite corners are the vertices. 
    \subsubsection{Tropical degree} Given a family of tropical supports $(\mathscr{T},\varphi)$ on $\Theta$ over a fan $\sigma$ we obtain a piecewise linear subdivision of $\Theta$ by pulling back $\mathscr{T}$ to the preimage of $0$ in $\sigma$. We call the resulting tropical support  $(\Delta,\varphi_\Delta)$ over a point the \textit{tropical degree}. Clearly tropical degree is constant in families.

    The tropical degree of a PL polyhedral structure $\mathscr{T}_p$ is the tropical degree of any family of tropical supports over some cone $\sigma$ for which $\mathscr{T}_p = \mathscr{T}_q$ for some point $q$ in $\sigma$. The \textit{infinite corners} of  $\mathscr{T}_p$ are the dimension one strata in $\mathcal{P}_\Delta$ where $(\Delta,\varphi_\Delta)$ is the tropical degree of $\mathscr{T}_p$. The infinite corners are thus the minimal unbounded strata.
    
    The set of \textit{corners} of a polyhedral structure is the union of the finite corners and the set of infinite corners. 
	\subsubsection{Discrete data - a first shot at combinatorial type} We construct the piecewise linear space $\supptrop$ by gluing piecewise linear cones. The gluands are indexed by data called the \textit{combinatorial type} of a tropical support. We now define the \textit{discrete data} of a tropical support- there will be finitely many combinatorial types for fixed discrete data.

    Let $\mathscr{T}\rightarrow \Theta \times \sigma$ be a family of tropical supports on $\Theta$. Choose a point $q$ and let $\mathscr{T}_q$ be the associated PL polyhedral structure.
	\begin{definition}
		The \textit{discrete data of a PL polyhedral structure}  on $\Theta$ is the following information.
		
		\begin{enumerate}[(1)]
			\item The set of $C$ of corners of $\mathscr{T}_q$.
			\item An assignment of corners of $\mathscr{T}_q$ to cones of $\Theta$. A corner $p$ is taken to the cone $\theta$ for which $p$ lies in the preimage of the interior of $\theta$ under the projection map to $\Theta$.
			\item A subset $W$ of the power set of $C$. For a stratum $\kappa$ in $\mathcal{P}_\mathscr{T}$ which maps to the interior of $\sigma$ we define $S_\kappa$ to be the collection of corners in the closure of $\kappa$.  The set $W$ is the set of $S_\kappa$ for $\kappa$ in $\mathcal{P}_\mathscr{T}$.
            \item For each set $S_\kappa$ a subgroup of $\theta^\mathrm{gp}$ where the interior of $\kappa$ is a subset of $\theta$ a cone of $\Theta$. This subgroup is generated by $\{s-s'|s,s' \in N_{\mathfrak{s}_\kappa}\}$.
		\end{enumerate}
		The discrete data of PL polyhedral structures $\mathscr{T}_1$ and $\mathscr{T}_2$ are the same if there is a bijection of finite corner sets $\varphi: C_1 \rightarrow C_2$ satisfying the hypotheses
		\begin{enumerate}[(1)]
			\item The corners $\varphi(c)$ and $c$ are assigned the same cone of $\Theta$.
			\item The map $\varphi$ induces a well defined bijection $$W_1 \rightarrow W_2$$ which preserves the associated subgroups of $\theta^\mathrm{gp}$.
		\end{enumerate}
	\end{definition}
	
	We now verify that discrete data is constant on the interiors of cones.
	
	\begin{proposition}\label{prop:Tropconnectedness}
		Consider a family of tropical supports $$\mathscr{T}\times \Theta \rightarrow \sigma$$ and let $p,q$ be points in the interior of $|\sigma|$. The discrete data of $\mathscr{T}_p$ and $\mathscr{T}_q$ coincide.
	\end{proposition}
	\begin{proof}
		Suffices to think about the finite corner sets. Finite corners of $\mathscr{T}_p$ are paired with finite corners of $\mathscr{T}_q$ if they lie in the same stratum of $\mathcal{P}_\mathscr{T}$, inducing a bijection $\varphi$ from the finite corners of $\mathscr{T}_p$ to the finite corners of $\mathscr{T}_q$ which preserves the cone of $\Theta$. The second property in the definition of isomorphism of discrete data follows from continuity.
	\end{proof}
	\begin{lemma}\label{lem:justneedcorner}
		Strata of a piecewise linear subdivision of a cone complex $\Theta \times \sigma$ are of the form $$\bigcup \mathrm{Int}(\mathrm{Conv}(C_i))$$ for strictly convex sets $C_i$ of rays.
	\end{lemma} 
	\begin{proof}
		Consider a stratum $s$ equipped with homeomorphism to local cone $\mathfrak{s}_1$ which is itself an open subset of $\mathbb{R}^k$ for some $k$. By part (1) of the definition of piecewise linear complexes, the closure of $s$ is a union of strata. The convex hull of $s$ is the convex hull of finitely many points $p_1,...,p_k$. All points $p_i$ are strata of $\mathcal{P}_S$; this follows by induction on dimension. The stratum $s$ is an open subset of the interior of $\mathrm{conv}(p_1,...,p_k)$. Necessarily any such open set is the complement of closed sets which are also convex hulls of rays. 
	\end{proof}
	
	\begin{corollary}
		A family of tropical supports on $\Theta$ over a cone $\sigma$ $$\mathscr{T}\rightarrow \Theta \times \sigma \rightarrow \sigma$$ is uniquely specified by two pieces of data.
		\begin{enumerate}
			\item The discrete data of $\mathscr{T}_p$ for $p$ any point in the interior of $\sigma$.
			\item The elements of $\mathcal{P}_\mathscr{T}$ corresponding to finite corners $\mathscr{T}_p$.
		\end{enumerate}
	\end{corollary}
	\begin{proof}
		Such a  family of tropical supports is specified by the data of a locally closed stratification of $$|\Theta \times \sigma|.$$ By Lemma \ref{lem:justneedcorner} such a  locally closed stratification is specified by the corners and which convex hulls are the same strata. This latter datum is specified by the discrete data. 
	\end{proof}
	
	\subsubsection{Local model} Let $D$ be a discrete datum. We say two finite corners $c_1,c_2$ of $D$ are \textit{similar} if there is an isomorphism of discrete data with associated bijection of finite corner sets $\varphi$ satisfying $\varphi(c_1) =c_2$. Suppose PL polyhedral structures in the combinatorial type $[\mathscr{T}]$ have $n$ finite corners and let $\mathrm{Sym}([\mathscr{T}])$ be the subgroup of the symmetric group on $n$ letters which maps each corner to a similar corner. The position of each finite corner gives a map of sets
	$$\Phi:\{\textrm{PL-polyhedral structures with discrete data }D\}\rightarrow |\Theta|^n/\mathrm{Sym}(D).$$
	Our strategy is to upgrade this map of sets to a map of piecewise linear spaces.
	
	\begin{proposition}\label{prop:localmodel}
		There is a disjoint union of local cones $\sqcup_i \mathfrak{s}_i$ and a $\mathrm{Sym}(D)$ equivariant embedding $$\bigsqcup_i\mathfrak{s}_i\hookrightarrow \Theta^n$$ such that the inclusion of topological realisations $$|\mathfrak{s}_i|/\mathrm{Sym}(D)\hookrightarrow|\Theta|^n/\mathrm{Sym}(D)$$ 
		is the image of $\Phi$. Moreover the image is disjoint from the big diagonal of $|\Theta|^d$.
	\end{proposition} 
	\begin{proof}
		We set $r: |\Theta^n| \rightarrow |\Theta|^n/\mathrm{Sym}(D)$ the quotient map. We will show $r^{-1}(\mathrm{Im}(\Phi))$ defines a finite union of local cones. Since the action of $\mathrm{Sym}(\mathscr{T})$ preserves which cone a point lies in, there are a sequence of cones $\sigma_1,...,\sigma_n$ of $\Theta$ such that $\mathscr{U}_D=r^{-1}(\mathrm{Im}(\Phi))$ lies in the topological realisation of the product $$\sigma_1 \times \sigma_2 \times ... \times \sigma_n \subset \Theta^d$$
		The product of cones is a cone and thus we may consider $$\sigma = \sigma_1 \times \sigma_2 \times ... \times \sigma_m \subset N$$ for some finitely generated free abelian group $N$. 
		
		The image of $\Phi$ is not the whole of $\sigma$. The third and fourth points in the discrete data cut out a subset. The fourth piece of data defines linear equalities. We thus determine a (saturated) subgroup $N_{\mathfrak{s}} \subset N$. We will take $N_{\mathfrak{s}}$ to be the torsion free abelian group in the definition of the local cones we construct.

        We are left to identify open subsets $U_{\mathfrak{s}_i}$ of $N_{\mathfrak{s}}\otimes \mathbb{R}$. The set $W$ imposes an open condition. Indeed fixing which sets lie in the closure of a stratum is an open condition - we are asking no corner lies in the closure of a stratum it should not lie in the closure of. Thus we have specified an open subset $U_\mathfrak{s}$ of $N_{\mathfrak{s}}\otimes \mathbb{R}$. The open subset is not obviously connected - set $U_{\mathfrak{s}_i}$ to be the connected components. 

        We define $\mathfrak{s}_i = (N_{\mathfrak{s}},U_{\mathfrak{s}_i})$ and claim we have specified local cones. By induction on dimension of $\Theta$ we know the closed subset removed to define $U_{\mathfrak{s}_i}$ is the support of a cone complex, and thus $U_{\mathfrak{s}}$ is the union of interiors of cones.
	\end{proof}
	
	We say PL subdivisions $\mathscr{T}_p$ and $\mathscr{T}_q$ are of the same \textit{combinatorial type} if we can take $p,q$ to lie in the same $\mathfrak{s}_i$. Write $[\mathscr{T}_p]=[\mathscr{T}_q]$ for the combinatorial type of $\mathscr{T}_p$
	
	Associated to any local cone $\mathfrak{s} = \mathfrak{s}_i$ is a piecewise linear complex which we now describe. Take the closure of $\mathfrak{s}$ in $\mathbb{R}^k$. This topological space has a conical locally closed stratification with two strata, one of which is $\mathfrak{s}$. The output of Construction~\ref{cons:} is a piecewise linear cone $\mathscr{S}_{[\mathscr{T}]}$. 
	
	We must subdivide the piecewise linear complex $\mathscr{S}_{[\mathscr{T}]}$ without subdividing its dense stratum. This is because multiple combinatorial types may appear in a single boundary stratum. To obtain the right subdivision we need to construct a universal family. In the sequel whenever $\mathscr{T}$ has discrete data $D$ we write $\mathrm{Sym}(D) = \mathrm{Sym}([\mathscr{T}])$.
	
	\subsubsection{Universal family over local model} We now construct the universal family associated to $\mathscr{S}_{[\mathscr{T}]}$ as a subdivision of $\mathscr{S}_{[\mathscr{T}]}\times \Theta$. In light of Construction \ref{cons:BddPLSpace} it suffices to construct a conical locally closed stratification $\mathcal{P}_{uni}$ of $$|\mathscr{S}_{[\mathscr{T}]}\times \Theta|\subset |\Theta^n\times \Theta|.$$ This locally closed stratification, and thus the universal family we construct, is $\mathrm{Sym}([\mathscr{T}])$ equivariant.
	
	We now explain how to construct $\mathcal{P}_{uni}$. The map $$\Theta^n \times \Theta\rightarrow \Theta^n$$ has $n$ universal sections $s_1,...,s_n$. We have fixed a bijection between $\{s_i\}$ and the set of finite corners in the discrete data $D$ of $\mathscr{T}_p$ when we thought of $\mathscr{S}_\mathscr{T}$ as a subset of $\Theta^n$. Infinite corners naturally biject with dimension one strata of $\mathcal{P}_{\mathscr{T}_0}$. Thus associated to each corner we have a locally closed stratum inside $\Theta^n \times \Theta$. 
 
 Write $\{x_i\}$ for the set of strata corresponding to corners. We say a subset $T\subset \{x_i\}$ is strictly convex if for all $x_i \in T$ the convex hull of $T \backslash \{x_i\}$ is not equal to the convex hull of $T$. We take strata over the interior of $\mathscr{S}_{[\mathscr{T}]}$ to be unions of convex hulls of the sections corresponding to the strictly convex subsets. Which unions to take is specified by the discrete data $D$. Strata which lie neither over the interior of $\mathscr{S}_{[\mathscr{T}]}$ nor the zero stratum in $\mathscr{S}_{[\mathscr{T}]}$ are dictated by combinatorial flatness.
	
	Denote the resulting map of piecewise linear spaces by $\mathscr{X}_{[\mathscr{T}]}\rightarrow \mathscr{S}_{[\mathscr{T}]}$. This map is not in general combinatorially flat because the image of each stratum of $\mathcal{P}_{\mathscr{X}_{[\mathscr{T}]}}$ may not be an entire stratum, violating axiom F1. Pushing forward the stratification on  $\mathscr{X}_{[\mathscr{T}]}$ defines a conical refinement of the stratification $\tilde{\mathcal{P}}$ on $\mathscr{S}_{[\mathscr{T}]}$. Pulling back $\tilde{\mathcal{P}}$ defines a conical refinement of $\mathcal{P}_{\mathscr{X}_{[\mathscr{T}]}}$. Applying Construction \ref{cons:} to both conical stratifications, we obtain a new morphism of piecewise linear spaces $$\varpi:\tilde{\mathscr{X}} \rightarrow \tilde{\mathscr{S}}_{[\mathscr{T}]}$$
	This map is combinatorially flat: axiom F1 holds by construction and we deduce F2 from Lemma~\ref{lem:F2}. On the interior of ${\mathscr{S}}_{[\mathscr{T}]}$ no subdivision was made to yield $\tilde{\mathscr{S}}_{[\mathscr{T}]}$.
	
	\subsection{Passing to the quotient.} Consider $|\mathrm{Sym}([\mathscr{T}])|$ copies of $\tilde{\mathscr{S}}_{[\mathscr{T}]}$ labelled by the elements of $\mathrm{Sym}([\mathscr{T}])$. Glue these complexes as follows. Whenever $\sigma_1,\sigma_2$ are elements of $\mathrm{Sym}([\mathscr{T}])$ such that $\sigma_1\sigma_2^{-1}$ fixes a subcomplex $\mathscr{G}\hookrightarrow \mathscr{T}$, we glue along the subcomplex $\mathscr{G}$. Denote the resulting piecewise linear complex $\mathscr{G}$ and observe $\mathrm{Sym}([\mathscr{T}])$ acts freely upon it.

Define a subcomplex $R$ of $\mathscr{G}\times \mathscr{G}$ whose strata are pairs $(\kappa,\sigma(\kappa))$ whenever $\kappa$ lies in $\mathcal{P}_\mathscr{G}$ and $\sigma$ in $\mathrm{Sym}([\mathscr{T}])$. This is a strict equivalence relation and thus the quotient is a piecewise linear space
 $\mathscr{U}_{[\mathscr{T}]}.$
	\subsubsection{Gluing local patches.} We now construct a diagram $$g:J \rightarrow \textbf{PLS} $$ of piecewise linear spaces with all morphisms open: the colimit over this diagram will be $\mathrm{Supp}(\Theta)$ and since all morphisms will be inclusions of subcomplexes, the result is a piecewise linear space. 

Objects of $J$ are pairs consisting of a copy of $\mathscr{U}_{[\mathscr{T}]}$ and a $\mathrm{Sym}([\mathscr{T}])$ equivariant weight function $$\varphi_{[\mathscr{T}]}:\mathcal{P}_{\tilde{\mathscr{X}}_{[\mathscr{T}]}} \rightarrow \bigcup_\theta G_\theta$$ where $\varphi_{[\mathscr{T}]}$ satisfies the conditions of Definition \ref{defn:FamilyTropSupp}. There is a morphism of pairs $$(\mathscr{U}_{[\mathscr{T}]},\varphi_{[\mathscr{T}]})\rightarrow (\mathscr{U}_{[\mathscr{T}']},\varphi_{[\mathscr{T}']})$$ whenever $[\mathscr{T}']$ arises as the combinatorial type of $\mathscr{T}_p$ for $p$ not in the dense stratum of $\tilde{\mathscr{S}}_{[\mathscr{T}]}$.

The image of a pair $(\mathscr{U}_{[\mathscr{T}]}, \varphi_{[\mathscr{T}]})$ under $g$ is the piecewise linear space $\mathscr{U}_{[\mathscr{T}]}$.  We now define the image of a morphism $$(\mathscr{U}_{[\mathscr{T}]},\varphi_{[\mathscr{T}]})\rightarrow (\mathscr{U}_{[\mathscr{T}']},\varphi_{[\mathscr{T}']})$$ under $g$. Certainly we know $[\mathscr{T}] \ne [\mathscr{T}']$. The proof of Proposition \ref{prop:localmodel} shows that there is a subcomplex $$ \iota:\tilde{\mathscr{S}}_{[\mathscr{T}']}\rightarrow \tilde{\mathscr{S}}_{[\mathscr{T}]}\textrm{ which descends to the quotient to define a subcomplex }\iota':\mathscr{U}_{[\mathscr{T}']}\hookrightarrow \mathscr{U}_{[\mathscr{T}]}.$$
And this is our morphism - a subcomplex is an open morphism.

The colimit over the piecewise linear spaces yields a new piecewise linear space $\supptrop(\Theta)$. There is a corresponding diagram of universal families with colimit denoted $\mathscr{X}$ and a morphism $\mathscr{X} \rightarrow \Theta \times \supptrop(\Theta)\rightarrow \supptrop(\Theta)$. There is a unique function $$\varphi: \mathcal{P}_\mathscr{X} \rightarrow \bigcup_\theta G_\theta$$ which pulls back to each of the $\varphi_{[\mathscr{T}]}$.
	
	\begin{proof}[Proof of Theorem \ref{thm: Universal property Supptrop}]
		We check the piecewise linear space $\supptrop(\Theta)$ satisfies the universal property of Theorem \ref{thm: Universal property Supptrop}. Consider a family of tropical supports over a cone $\sigma$ say $$\mathscr{X}_\sigma\rightarrow \sigma \times \Theta \rightarrow \sigma, \varphi_\sigma.$$ Enumerating the $n$ corners of the PL subdivision corresponding to $\mathscr{X}_\sigma$ gives a section of the map $\Theta^n\times \sigma\rightarrow \sigma$ which we think of as a map $\sigma \rightarrow \Theta^n$. The combinatorial type of a PL polyhedral structure arising from a point in the interior $\sigma$ dictates a choice of $\mathscr{U}_D$ and we have the composition $$\sigma \rightarrow \Theta^n \rightarrow \mathscr{U}_D.$$ The weight function thus defines a map from $\sigma$ to the image under $g$ of a unique pair $(\mathscr{U}_D,\varphi)$. The function $\varphi$ is determined uniquely by $\varphi_\sigma$. In this way we obtain a map to $\supptrop(\Theta)$ along which universal families pull back.
	\end{proof}

\section{Logarithmic flatness for toric varieties}
In this section we develop our understanding of the transversality condition captured by \textit{logarithmic flatness.} In the sequel we will define what it means for a coherent sheaf $\mathcal{F}$ on a logarithmic $S$ scheme $X/S$ to be logarithmically flat over $S$. In this section we restrict to the case that $S$ is a point with the trivial logarithmic structure and $X$ is a toric variety equipped with divisorial logarithmic structure from the toric boundary. 

\begin{situation}\label{sit:ToricWorld}
Let $X$ be an affine toric variety with dense torus $X^o$ and cocharacter lattice $M_X$. Define $\mathcal{E} = \mathcal{O}_X^n$ a locally free coherent sheaf on $X$. Consider a short exact sequence of sheaves on $X$ $$0 \rightarrow \mathcal{G} \rightarrow \mathcal{O}_X^n \xrightarrow{q} \mathcal{F}\rightarrow 0 \textrm{ with global sections } 0 \rightarrow G \rightarrow \C[X]^n \xrightarrow{q} {F}\rightarrow 0.$$ Write $G^o$ for $\mathcal{G}(X^o)$ and note since open immersions are flat, $\mathcal{G}(X^o)$ is a submodule of $\mathbb{C}[X^o]^n$. Assume $\mathcal{F}$ has no non--zero sections with support contained within the toric boundary of $X$.
\end{situation}

\subsection{The tropical support of $q$.} The tropical support of a surjection of sheaves $q$ on a toric variety $X$ is a piecewise linear space $\mathscr{T}$ subdividing $\mathbb{R}^{\mathrm{dim}(X)} = M_X \otimes \mathbb{R}$. Tropical support will play an important role in the sequel where it is defined in terms of torus actions. In this section we instead take the Gr\"obner theory approach.

\subsubsection{Monomials}
Observe both $\mathbb{C}[X^o]^n = \mathcal{O}_X^n(X^o)$ and $\mathbb{C}[X]^n = \mathcal{O}_X^n(X)$ admit natural actions of $X^o$. A monomial of $X^o$ (respectively $X$) is a character occuring in the $X^o$ representation $\mathbb{C}[X^o]^n$ (respectively $\mathbb{C}[X]^n$). Write $\mathrm{mon}(\mathcal{E}), \mathrm{Mon}(\mathcal{E})$ for the set of monomials of $X^o$ and $X$ respectively.

\subsubsection{Term orders} A \textit{term order} on $\mathrm{mon}(\mathcal{E})$ respectively $\mathrm{Mon}(\mathcal{E})$ is a total pre--order.
\begin{remark}\label{rem:Grobnerweak}
This is weaker than the notion of term order considered in Gr\"obner theory in two ways. First the relation need not be anti-symmetric. A more important subtlety is that for $M_1,M_2,P$ monomials with $M_1 \leq M_2$ it need not be true that $$M_1P\leq M_2P.$$ 
\end{remark}
Instead of unpacking this definition we construct all term orders which will be of interest to us in Example \ref{example:termorder}.
\begin{example}\label{example:termorder}
    The co-character/ character pairing assigns to a cocharacter $w$ of $X^o$ functions $$\mathrm{Mon}(\mathcal{E}) \rightarrow \mathbb{Z} \quad \mathrm{mon}(\mathcal{E}) \rightarrow \mathbb{Z}.$$ Pulling back the total order on $\mathbb{Z}$ assigns to $w$ a term order on $\mathrm{Mon}(\mathcal{E}),\mathrm{mon}(\mathcal{E})$ respectively. For $k$ a positive integer the term order associated to $w$ and $kw$ is the same. This allows us to define a term order associated to every $w'\in M\otimes \mathbb{Q} = \mathbb{Q}^n$: choose a positive integer $\ell$ such that $\ell w'$ lies in $\mathbb{Z}^n$ and use the term order associated to $\ell w$.
\end{example}

\subsubsection{Initial forms}
Any element $g^o \in \mathbb{C}[X^o]^n$ respectively $g \in \mathbb{C}[X]^n$ may be expressed as a $\mathbb{C}$ linear combination of monomials $$g^o = \sum_{m \in \mathrm{mon}(\mathbb{C}[X^o]^n)} a_m^o m,   \quad g=\sum_{m \in \mathrm{Mon}(\mathbb{C}[X]^n)} a_m m   \textrm{ where }a_m^o, a_m \in \mathbb{C}.$$
We define new elements of $\mathbb{C}[X^o]^n$ and $\mathbb{C}[X]^n$ respectively $$\mathrm{in}_w(g^o) = \sum_{w(m) \textrm{ maximal}} a_m^o m \quad \textrm{ and }\quad  \mathrm{In}_w(g) = \sum_{w(m) \textrm{ maximal}} a_m m.$$

\begin{remark}
    The coefficients $a_m^o, a_m$ are not well defined. Indeed for $\mathbb{C}[X] = k[x]$ and $n=2$ we have $$(x,0) +(0,x) = (x,x)$$ and note $(x,x), (x,0)$ and $(0,x)$ are all monomials. The initial form $\mathrm{In}_w(g)$ is none the less well defined.
\end{remark}

\begin{remark}
    The initial forms need not be monomials. For example setting $\mathbb{C}[X] = \mathbb{C}[x,y],n=1$ and $w=(1,1)$ we find $$x+y = \mathrm{in}_w(x+y).$$
\end{remark}
\subsubsection{Initial submodules}
Given a submodule $$G\leq \mathbb{C}[X]^n \textrm{ or }G^o\leq \mathbb{C}[X^o]^n$$ we define $\mathrm{In}_w(G)$ the submodule of $\mathbb{C}[X]^n$ generated by $\mathrm{In}_w(g)$ for $g\in G$. Similarly $\mathrm{in}_w(G^o)$ is the submodule of $\mathbb{C}[X^o]^n$
 generated by $\mathrm{in}_w(g)$ for $g\in G^o$.
\subsubsection{The Gr\"obner stratification} We now work in Situation \ref{sit:ToricWorld}. As $w$ varies in $M_X \otimes \mathbb{Q}$ the intial ideal $\mathrm{in}_w(G^o)$ also varies. There is a conical locally closed stratification of $M \otimes \mathbb{Q}$ on which this initial ideal is constant. Following \cite{Cartwright}, we call this stratification the \textit{Gr\"obner stratification}. The \textit{tropical support} of $q$ is the piecewise linear space $\mathscr{T}(q)$ associated to the conical stratification obtained from the common refinement of the fan of $X$ and the Gr\"obner stratification.
\subsection{Notions of transversality}\label{sec:transversalitynew} We begin by stating several notions of transversality. The remainder of this section is largely dedicated to the connections between these definitions and how we can ensure they hold.
 \subsubsection{Logarithmic flatness}\label{sec:toriclogflat} Logarithmic flatness will be defined in general in Section \ref{subsec:LogFlat}. For now we handle the special case relevant to this section. Denote the projection and torus multiplication maps $$X\xleftarrow{\pi_1} X\times X^o \xrightarrow{m} X.$$ A sheaf $\mathcal{F}$ on a toric variety $X$ is called \textit{logarithmically flat} over $\mathrm{Spec}(\mathbb{C})$ if $m^\star \mathcal{F}$ is flat over $X$ along the map $\pi_1$. 
\subsubsection{Transversality}\label{sec:toricTransversality} For $\sigma$ a cone in the fan of $X$ denote the associated torus orbit $O(\sigma)$ and set $V(\sigma)$ the closure of $O(\sigma)$. We say a sheaf $\mathcal{F}$ on $X$ is \textit{transverse} if the pullback of $\mathcal{F}$ to $V(\sigma)$ has no sections supported on the compliment of $O(\sigma).$

\begin{example}
Set $X = \mathbb{A}^2$, and note the fan of $X$ has two rays, one zero cone and one two dimensional cone. Left is a curve with transverse structure sheaf. Centre is not transverse because setting $\sigma$ any ray, the structure sheaf has a section supported on the orbit supported at the maximal cone - i.e. $(0,0)$. Right is not transverse by setting $\sigma$ the zero cone. 
\end{example}
\begin{figure}[ht]
			\centering
			\includegraphics[width= 120mm]{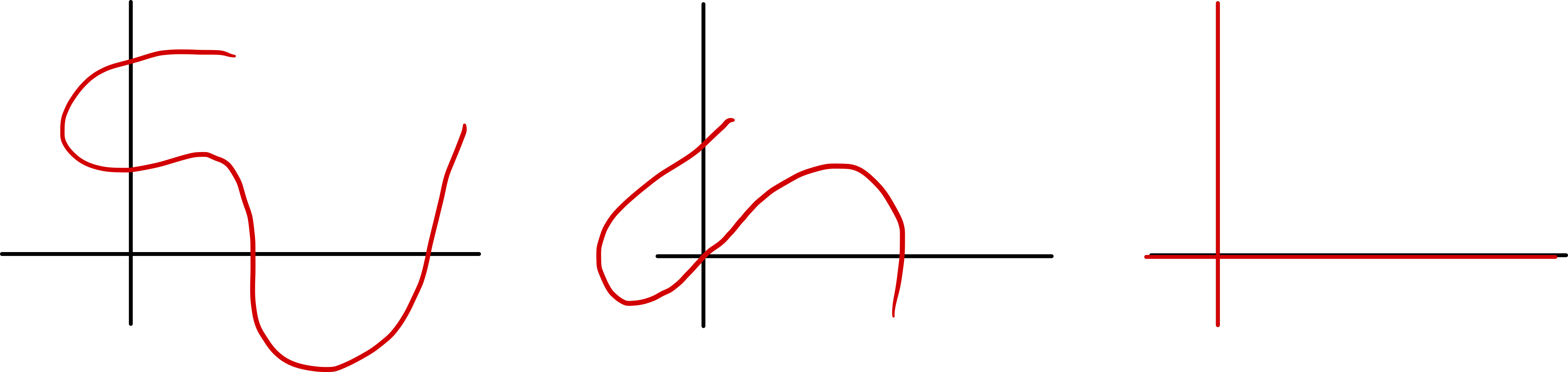}
			\caption{Curves in $\mathbb{A}^2$ depicted in red. The divisors $X=0,Y=0$ are depicted as black lines.}\label{fig:KeyExample}
		\end{figure}

\subsubsection{Strict and total transform}\label{sec:toric Strict and total} We say a sheaf $\mathcal{F}$ on $X$ has the strict and total transform property if given any toric modification $$\pi:X_\Gamma \rightarrow X$$ the strict transform $\pi^!\mathcal{F}$ and total transform $\pi^\star\mathcal{F}$ coincide.
\subsubsection{Tropical support}\label{subsec: toric tropical support} Our final condition requires that $X_\Gamma$ is a tropical model of the tropical support of $q$. 

\subsection{Tevelev's theorem} In the sequel we focus on the following upgrade of Situation \ref{sit:ToricWorld}.

\begin{situation}\label{sit:toricflattening}
    Continuing Situation \ref{sit:ToricWorld}, let $\pi_\Gamma: X_\Gamma\rightarrow X$ be a toric modification. Writing $\mathcal{F}_\Gamma$ for the strict transform of $\mathcal{F}$ there is a short exact sequence of coherent sheaves on $X_\Gamma$ $$0 \rightarrow \mathcal{G}_\Gamma \rightarrow \pi_\Gamma^\star\mathcal{E} \xrightarrow{q_\Gamma} \mathcal{F}_\Gamma\rightarrow 0 \textrm{ with global sections } 0 \rightarrow G_\Gamma \rightarrow C[X_\Gamma]^n \xrightarrow{q_\Gamma} {F}_\Gamma\rightarrow 0.$$ 
We assume that $\mathcal{F}_\Gamma$ is logarithmically flat over $\mathrm{Spec}(\mathbb{C})$.
\end{situation}

The importance of this situation reflects the following theorem first proved by Tevelev in unpublished notes \cite{TevNotes}, generalising his published result \cite{tevelev2005compactifications}. The theorem was first stated in this form in the literature by Ulirsch \cite{UlirschThesis}. We recall Tevelev's proof as we are unable to find a suitable reference.

\begin{theorem}\label{thm:OldTev}
    Let $\mathcal{F}$ be a sheaf on toric variety $X$ with no sections supported on the toric boundary. There is a logarithmic modification $X_\Gamma \rightarrow X$ such that the strict transform of $\mathcal{F}$ is logarithmically flat. 
\end{theorem}

\begin{proof}
    We show that $m^\star \mathcal{F}$ can be flattened by strict transform under an equivariant blowup of $X$. We think of $\mathbb{G}_m^k$ as the dense torus of $\mathbb{P}^k$ and we consider the projection map $$\mathbb{P}^k \times X\rightarrow X $$ The pushforward of $m^\star\mathcal{F}$ to $\mathbb{P}^k \times X$ is equivariant and quasicoherent but may not be coherent. It contains a coherent equivariant subsheaf. We denote this equivariant coherent sheaf $\mathcal{G}.$

    We now consider Grothendieck's Quot scheme $$Q = \mathsf{Quot}(\mathbb{P}^k \times (X)/X,\mathcal{G})$$ which is a union of projective varieties over $X.$

 By generic flatness $\mathcal{G}$ is flat over an open subscheme $U\subset X$ and so the map $Q \rightarrow Y$ admits a section $U \rightarrow Q$ over $U$. Take $\overline{X}$ a scheme theoretic image of $U$ and observe the map $\overline{X}\rightarrow X$ is birational as it is an isomorphism over $Q$ and the domain integral because $U$ is integral. By a theorem of Grothendieck it is a blowup \cite[2.3.5]{EGA3} of a closed subscheme.

 The morphism $\overline{X}\rightarrow Q$ carries a universal surjection $\overline{\mathcal{G}}\rightarrow \mathcal{N}$ whose kernel is supported outside of $U$ and thus $\mathcal{N}$ is the strict transform of $\mathcal{G}$ and is necessarily flat.

  It remains to check the map $\overline{X}\rightarrow X$ is a logarithmic modification. We know it is a blowup in a torus equivariant subscheme. Any such blowup is a logarithmic modification.
\end{proof}

\subsection{Types of transversality} We explain the connection between the transversality conditions introducted in Section \ref{sec:transversalitynew}. 

\begin{theorem}\label{thm:CharacteriseLogFlat}
    The following are equivalent
    \begin{enumerate}[(1)]
        \item The sheaf $\mathcal{F}_\Gamma$ is logarithmically flat, see Section \ref{sec:toriclogflat}.
        \item The sheaf $\mathcal{F}$ has the strict and total transform property, see Section \ref{sec:toric Strict and total}.
    \end{enumerate}
    Either of these conditions implies 
    \begin{enumerate}[(1)]
        \setcounter{enumi}{2}
        \item The sheaf $\mathcal{F}_\Gamma$ is transverse, see Section \ref{sec:toricTransversality}.
    \end{enumerate}
    Any of the above implies our next condition.
    \begin{enumerate}[(1)]
        \setcounter{enumi}{3}
        \item The fan of $X$ is a subdivision of $\mathscr{T}(q)$, see Section \ref{subsec: toric tropical support}. 
    \end{enumerate}
\end{theorem}
The following theorem is helpful for building intuition about tropical support. Its proof does not apppear in this paper

\begin{theorem}[\cite{Quot2}]\label{conj:CharacteriseLogFlat}
    Assume also that $X$ is smooth. Condition (4) implies Condition (1) and thus the conditions in Theorem~\ref{thm:CharacteriseLogFlat} are equivalent for smooth $X$.
\end{theorem}

The proof of Theorem \ref{thm:CharacteriseLogFlat} will occupy the remainder of this subsection. 
\subsubsection{Consequences of logarithmic flatness} We start our proof of Theorem \ref{thm:CharacteriseLogFlat} by showing that logarithmic flatness implies two other conditions.
\begin{proof}[Proof $(1)\Rightarrow (2)$]
This is a cosmetic modification of \cite[TAG 080F]{stacks-project}.
\end{proof}

\begin{proof}[Proof $(1)\Rightarrow (3)$] 
Assume $\mathcal{F}$ were to have a section $s$ supported on the compliment of $O(\sigma)$ in $V(\sigma)$ for some cone $\sigma$. Since flatness and transversality are local conditions, without loss of generality $V(\sigma)$ is affine. Choose a non-zero function $f$ pulled back from the algebraic stack $[V(\sigma)/\mathbb{G}_m^{\mathrm{dim}(X)}]$ such that $V(f)$ contains the toric boundary of the toric variety $V(\sigma)$. Necessarily the zero set of such an $f$ is supported on the toric boundary. By flatness the map $(f^k) \otimes F_\Gamma\rightarrow F_\Gamma$ is injective. This implies that $f^ks$ is not zero for any integer $k$ and consequently the localisation map from $\mathcal{O}_{V(\sigma)}$ obtained by inverting $f$ does not map $s$ to zero. It follows that the restriction of $s$ to $O(\sigma) = V(\sigma) \backslash V(f)$ is not zero and so the support of $s$ was not contained within the toric boundary.  
\end{proof}

\subsubsection{Test cocharacters} Given a cocharacter $w\in M_X$ we define a morphism $$\varphi_w^o:S^o=\mathrm{Spec}(\mathbb{C}((t)))\rightarrow X^o$$ which on the level of coordinate rings is specified by $$(\varphi_w^o)^\#:\mathbb{C}[X_1^{\pm 1},...,X_n^{\pm 1}] \rightarrow \mathbb{C}((t))\quad X_i \mapsto t^{w_i}.$$
Assuming $w$ lies in the support of the fan of $X$, the map $\varphi_w^o$ extends to a map $$\phi_w:S=\mathrm{Spec}(\mathbb{C}[[t]])\rightarrow X.$$ 

The short exact sequence $$ 0 \rightarrow \mathcal{G} \rightarrow \mathcal{O}_X^n\rightarrow \mathcal{F} \rightarrow 0$$ pulls back along $m$ to a short exact sequence $$ 0 \rightarrow \mathcal{G}' \rightarrow \mathcal{O}_{X\times X^o}^n\rightarrow \mathcal{F}' \rightarrow 0.$$ We set $s$ the special point of $S$ and pull back along $s\rightarrow X$ to obtain a surjection of sheaves $$0 \rightarrow \mathcal{G}_w \rightarrow \mathcal{O}_{X^o}^n\rightarrow \mathcal{F}_w \rightarrow 0$$ on $X^o$. Evaluating this sheaf on global sections we obtain $$0 \rightarrow G_w \rightarrow \mathbb{C}[X^o]^n\rightarrow F_w \rightarrow 0.$$

\begin{lemma}\label{lem:inw is Gw}
    There is an equality of submodules of $\mathbb{C}[X^o]^n$ $$G_w = \mathrm{in}_w(G^o).$$
\end{lemma}
\begin{proof}
    This is a special case of Proposition \ref{prop:initialideal}.
\end{proof}
\begin{proof}[Proof $(1) \Rightarrow (4)$]
    The submodule $G_w$ depends only upon $s$. The point $s$ depends only upon the cone $\sigma$ for which $w$ lies in the interior of $\sigma$. The result now follows from Lemma \ref{lem:inw is Gw}.
\end{proof}

\subsubsection{Transversality and the strict/total transform property} We work in Situation \ref{sit:toricflattening} but impose that Condition (4) of Theorem \ref{thm:CharacteriseLogFlat} is false for $q_\Gamma$ and impose Condition (3) instead of logarithmic flatness. We will show that Condition (3) implies Condition (4) by finding a contradiction. Pick $\sigma$ a minimal cone in the fan of $X$ such that $\mathrm{in}_w(G^o)$ is not constant for $w$ in the interior of $X$ and let $\tau$ be any facet of $\sigma$.

Pull $q$ back to $O(\tau)$ and take global sections to obtain a surjection of modules $$\mathbb{C}[O(\tau)]^n \rightarrow \mathcal{F}_\sigma$$ with kernel denoted $G_\sigma$ a $\mathbb{C}[O(\tau)]$ module. Let $w_D$ be a point in the interior of $\sigma$ lying in the stratum whose closure contains $\tau$.

    \begin{lemma}\label{lem:GsigmaIsinwD}
        If $\mathcal{F}_\Gamma|_{V(\tau)}$ has no sections supported on $V(\sigma)$ then $$G_\sigma = \mathrm{in}_{w_D}(G_\tau)$$
    \end{lemma}
    \begin{proof}
        It suffices to handle the case that $\tau$ is the zero cone and $\sigma$ a ray. The result is now Lemma~\ref{lem:inw is Gw}.
    \end{proof}

\begin{proof}[Proof $(3) \Rightarrow (4)$]
    In the above setup we claim that $\mathcal{F}_\Gamma|_{V(\tau)}$ has a section supported on $V(\sigma)$, which is in the compliment of $O(\tau)$. Thus whenever (4) is false (3) is also false. The contrapositive proves our claim.

    We argue $G_\sigma <\mathrm{in}_{w_D}(G_\tau)$ is a proper submodule contradicting the conclusion of Lemma \ref{lem:GsigmaIsinwD}. Indeed for any $w'$ in the interior of $\sigma$ we have $G_\sigma \leq \mathrm{in}_{w'}(G)$. Choose $w'$ not in the same stratum $\kappa$ as $w_D$, but ask $w'$ lies in the closure of $\kappa$. Lemma~\ref{lem:inw is Gw} applied to a sufficiently fine toric modification of $X_\Gamma$ now shows $\mathrm{in}_w(G)$ admits a degeneration to $\mathrm{in}_{w_D}(G)$ and so it is not possible that $\mathrm{in}_{w_D}(G)<\mathrm{in}_{w'}(G)$ unless they are equal. Since we assumed no equality, there is some element of $\mathrm{in}_{w_D}(G)$ which is not in $\mathrm{in}_{w'}(G)$ and thus not in $G_\sigma$.
\end{proof}

\subsubsection{Strict transforms, total transforms and logarithmic flatness} We deduce that $(2) \Rightarrow (1)$ from the following proposition.

\begin{proposition}\label{prop:Strict/Total implies log flat}
    Assume that in Situation \ref{sit:toricflattening} the strict and total transform of $\mathcal{F}$ under $\pi_\Gamma$ coincide. Then $\mathcal{F}$ is logarithmically flat.
\end{proposition}

\begin{proof}
    The valuative criterion for flatness \cite[11.8.1]{EGA4} states that to check $\mathcal{F}$ is flat over $X$ it suffices to check that for $S$ any trait and given a map $p: S \rightarrow X$ we have the following property. The sheaf $\mathcal{F}_S$ on $S \times X^o$ obtained by pulling back $\mathcal{F}$ along the induced map $S \times X^o\rightarrow X$ is logarithmically flat over $S$. 

    In this paragraph we show that it suffices to handle the case that the generic point of $S$ maps to the interior of the dense torus of $X$. Let $\sigma$ be a cone of $X$. We claim now that $\mathcal{F}$ satisfying the hypotheses of the theorem implies that $\mathcal{F}|_{V(\sigma)}$ also satisfies these hypotheses. It suffices to handle the case that $\sigma$ is a ray as the argument may be iterated for the general case. For $\sigma$ a ray note that the logarithmic modification $X_\Gamma \rightarrow X$ induces a logarithmic modification $Y \rightarrow V(\sigma)$. Since flatness is preserved by base change the fact $\mathcal{F}_\Gamma$ is logarithmically flat implies that the pullback of $\mathcal{F}|_{V(\sigma)}$ to $Y$ is logarithmically flat. This implies that the strict and total transforms coincide because logarithmic flatness implies no sections supported on the compliment of $X^o$.

    To finish the proof we appeal to the valuative criterion for flatness. A map $S \rightarrow X$ such that the generic point of $S$ intersects the dense torus of $X$ specifies a cocharacter of $X$ and thus a ray $\rho$ in the fan of $X$. Choose a logarithmic modification $$X_{\Gamma'} \rightarrow X_\Gamma\rightarrow X$$ such that $\rho$ is a ray of $\Gamma'$. By the universal property of blowup the map from $S\rightarrow X$ factors through $X_{\Gamma'}$. Since pullback is functorial we know $\mathcal{F}_S$ is pulled back from a sheaf $\mathcal{F}''$ on $X_{\Gamma'}\times X^o$. Observe $\mathcal{F}''$ is flat over $X$ by construction. Flatness is preserved under base change and we have verified the valuative criterion for flatness.
\end{proof}

\begin{proof}[Proof $(2) \Rightarrow (1)$]
    Theorem \ref{thm:OldTev} implies that there is a logarithmic modification $\pi: X_\Gamma \rightarrow X$ on which the strict transform of $\mathcal{F}$ is transverse. The strict and total transform property implies we are then in the situation of Proposition \ref{prop:Strict/Total implies log flat}. 
\end{proof}

	\section{Proper monomorphisms and logarithmic surjections of coherent sheaves}\label{sec:RelativeSheaves}

    In this section we define logarithmic surjections of coherent sheaves. In Section \ref{sec:LogQuot} we will define the logarithmic Quot space as a moduli space of (logarithmically flat and integral) logarithmic surjections of coherent sheaves. We draw the readers attention to Section \ref{sec:getmap} which gives a geometric perspective on the importance of the definition of logarithmic flatness.
    
    Logarithmic surjections of coherent sheaves are modelled on surjections of module valued sheaves in a certain topology \cite{MW23,Nak17} - we believe this is an intrinsic perspective on the logarithmic Quot space. We finish the section by connecting logarithmic surjections from the structure sheaf to proper monomorphisms in the category of logarithmic schemes. This gives an alternative intrinsic perspective on the logarithmic Hilbert scheme.
    
    We will have cause to fix a sheaf $\mathcal{E}$ on $X$. To simplify notation we use the same symbol $\mathcal{E}$ to denote the pullback of the fixed sheaf $\mathcal{E}$ on $X$ to $X \times S$.

	\subsection{Artin fans and logarithmic modification}\label{sec:logbackground}
 See \cite[Definition 1.2]{LogFlat} for the definition of fine and saturated logarithmic schemes and further background. In this paper all logarithmic schemes are fine and saturated; in the sequel we write logarithmic scheme with fine and saturated being understood.
	
	\subsubsection{Tropical geometry and Artin fans} Stacks over \textbf{RPC} can be lifted to define stacks on the category of logarithmic schemes, see \cite[Section 6]{ModStckTropCurve}. Given a category fibered in groupoids $C\rightarrow \textbf{RPC}$ one defines a category fibered in groupoids $\mathcal{A}_C$ over the category of logarithmic schemes by setting $$\mathcal{A}_C(S) = C(\Gamma(X,M_X)).$$ Stackifying $\mathcal{A}_C$ yeilds $\mathpzc{C}$. If $C$ is a cone stack, see \cite[Section 2]{ModStckTropCurve} then $\mathpzc{C}$ is a zero dimensional algebraic stack. We call algebraic stacks arising in this way \textit{Artin fans}, see \cite{abramovich2015skeletons,AWbirational} for the theory of Artin fans. The map from cone stacks to Artin fans defines an equivalence of categories. 
	
	For a piecewise linear space $\mathscr{T}$ we write the corresponding zero dimensional stack by $a^\star\mathscr{T}$. In the special case of the moduli space of tropical supports we write $a^\star \supptrop(\mathpzc{X})  = \supplog(\mathpzc{X})$. We do not assert that this zero dimensional stack is algebraic, although it has a logarithmic \'etale cover by Artin fans.
	
	\subsubsection{The Artin fan of a logarithmic scheme.} Let $X$ be a logarithmic scheme with locally connected logarithmic strata. There is an initial strict morphism from $X$ to an Artin fan with faithful monodromy \cite{AWbirational}. We denote this map $$X \rightarrow \mathpzc{X}$$ and call $\mathpzc{X}$ the \textit{Artin fan} of $X$. Write $\mathrm{Trop}(X)$ for the stack on \textbf{RPC} such that $a^\star\mathrm{Trop}(X) = \mathpzc{X}.$
	
	The assignment of an Artin fan to a logarithmic scheme is not functorial in general, but there is a substitute for functorality, see \cite[Section 5]{abramovich2015skeletons}. Given a morphism of logarithmic schemes $X \rightarrow Y$ there is a relative Artin fan $\mathpzc{X}_Y$ and a commutative diagram
	$$
\begin{tikzcd}
X \arrow[d] \arrow[r] & \mathpzc{X}_Y \arrow[d] \\
Y \arrow[r]           & \mathpzc{Y}.            
\end{tikzcd}$$
	Indeed $X\rightarrow \mathpzc{X}_Y$ is the initial strict map from $X$ to an Artin fan (with faithful monodromy), through which the map $X \rightarrow \mathpzc{Y}$ factors.
	\subsubsection{Logarithmic modification} A \textit{logarithmic modification} $Y \rightarrow X$ of $X$ is a morphism that is \'etale locally on $X$ the base change of a morphism $$a^\star\Gamma \rightarrow a^\star\sigma$$ along a strict map $X \rightarrow a^\star \sigma$ where $\Gamma \rightarrow \mathrm{Trop}(B)$ is a morphism of cone complexes which is an isomorphism of topological realisations. Notice the morphism $\Gamma \rightarrow \mathrm{Trop}(B)$ need not be a tropical model and thus our definition of logarithmic modification permits lattice changes.
	
\subsubsection{From Artin fan to locally closed stratification} Assume $X$ has locally connected logarithmic strata. There is a bijection between cones $\sigma$ of $\mathrm{Trop}(X)$ and (stacky) points $a_\sigma$ of the Artin fan $\mathpzc{X}=a^\star\mathrm{Trop}(X)$. Thus the map $X\rightarrow \mathpzc{X}$ defines by pulling back points a locally closed semi-stratification of $X$. We denote the stratum associated to a cone $\sigma$ by $X_\sigma$. Fixing a tropical model $\Gamma \rightarrow \mathrm{Trop}(X)$ defining logarithmic modification $X_\Gamma\rightarrow X$, note to each cone $\gamma$ of $\Gamma$ there is a locally closed stratum $O(\gamma)$ of $X_\Gamma$. We denote the preimage of the closure of $a_\gamma$ by $V(\gamma)$.
	
	\subsubsection{Star piecewise linear spaces} Consider a subdivision $\mathscr{G}\rightarrow \mathrm{Trop}(X)$ and let $\gamma$ be an element of $\mathcal{P}_\mathscr{G}$ such that the image of $\gamma$ lies in the interior of a cone $\sigma$ of $\mathrm{Trop}(X)$. Suppose $\sigma$ has associated local cone $(N_\sigma, U_\sigma)$ and note $$N_\sigma \cong \mathbb{Z}^{\mathrm{dim}(\sigma)} \cong \sigma^\mathrm{gp}.$$ We let $\gamma$ have associated local cone $(N_\gamma,U_\gamma)$ and observe there is a natural inclusion $N_\gamma \hookrightarrow N_\sigma$.  The \textit{star piecewise linear space} $\mathrm{St}_\gamma$ of $\gamma$ is the data of a piecewise linear structure on $N_\sigma(\gamma) = N_\sigma/N_\gamma$. To specify such a piecewise linear structure it is enough to specify the associated locally closed stratification $\mathcal{P}_\gamma$ of $N_\sigma(\gamma)\otimes \mathbb{R}$. 
 
Strata of $\mathcal{P}_\gamma$ biject with strata $\kappa$ of $\mathcal{P}_\mathscr{G}$ such that $\gamma$ lies in the closure of $\kappa$. The stratum associated to $\kappa$ in $\mathcal{P}_\mathscr{G}$ is the image of $\kappa$ under the quotient map $M_\sigma \rightarrow M_\sigma/M_\gamma$. Write $\mathrm{St}_\gamma(\kappa)$ for the cone of $\mathrm{St}_\gamma$ corresponding to $\kappa$ in $\mathcal{P}_\mathscr{G}$. For example,  $\mathrm{St}_\gamma(\gamma)$ is the zero cone. We warn the reader that if $\sigma$ is a face of a cone $\sigma'$ then the star fan of $\gamma$ in $\mathscr{G}(\sigma)$ does not detect cones in $\mathscr{G}(\sigma')$ which contain $\gamma$ in their closure.

	\subsection{Logarithmic Flatness}\label{sec:transversality}
Example \ref{ex:NeedLogQuot} demonstrates our claim from the introduction that pullback does not describe a map $$ r_i:\mathrm{Quot}(X,\mathcal{E})\rightarrow \mathrm{Quot}(D_i,\mathcal{E}|_{D_i}).$$ In this section we understand the open set on which $r_i$ is well defined. The correct technical condition to ask for is \textit{logarithmic flatness}. An important subtlety is that being logarithmically flat over a point is a non--trivial condition generalising strong transversality \cite{MR20}. The special fibre in Example \ref{ex:NeedLogQuot} does not satisfy this condition.

\begin{example}\label{ex:NeedLogQuot}
Consider a strict map$$V(X+t(Y+Z)) \rightarrow \mathbb{P}^2\times \mathbb{A}^1$$ where the target is a logarithmic scheme with toric logarithmic structure. This morphism is logarithmically flat over $\mathbb{A}^1$ away from $0$ where logarithmic flatness fails. Pulling back the universal surjetion $$\mathcal{O}_X \rightarrow \mathcal{O}_Z$$ to the divisor $X = 0$ defines a subschme of $\mathbb{P}^1 \times \mathbb{A}^1$. This subscheme is not flat (in the usual sense) over $\mathbb{A}^1$: the fibre over $0 \in \mathbb{A}^1$ is dimension one whereas the fibre over every other closed point of $\mathbb{A}^1$ has dimension zero.
\end{example}

\subsubsection{Logarithmic flatness for subschemes over a point} Let $Z$ be a strict closed subscheme of $X$. Every such scheme is flat over $\mathrm{Spec}(\mathbb{C})$ but not all choices of $Z$ are \textit{logarithmically flat} in the sense of the following rephrasing of the definition presented in \cite[Section 1.10]{LogFlat}. \begin{definition}\label{defn:logflatoverC}
    We say $Z$ is logarithmically flat over $\mathrm{Spec}(\mathbb{C})$ if $Z$ is flat over the Artin fan of $X$.
\end{definition}

In the special case that $X$ is a toric variety with divisorial logarithmic structure from its toric boundary, Definition \ref{defn:logflatoverC} coincides with the notion from Section \ref{sec:toriclogflat}. There is also a version of the transversality of Section \ref{sec:toricTransversality}.

\begin{definition}\label{defn:transverse}
    A subscheme $Z$ of $X$ is \textit{transverse} if for every logarithmic stratum $O(\sigma)$ the closure of $Z\cap O(\sigma)$ in $V(\sigma)$ coincides with $V(\sigma) \cap Z$.
\end{definition}

Definition \ref{defn:transverse} asks that $Z\cap V(\sigma)$ has no embedded component supported on the compliment of $O(\sigma)$. Imposing that $Z$ is logarithmically flat over $S$ is a transversality condition closely related to strong transversality defined in \cite{MR20}. Figure \ref{fig:KeyExample} has examples of how transversality can fail.

	\subsubsection{Logarithmic flatness in the language of Artin fans}\label{subsec:LogFlat} Let $\pi: X\rightarrow B$ be a morphism of logarithmic schemes and let $\mathcal{F}$ be a coherent sheaf on $X$. We upgrade Kato's definitions of logarithmic flat and integral for a morphism of logarithmic schemes to define what it means for a coherent sheaf to be logarithmically flat and integral over a base, see \cite[Section~1.10]{LogFlat}. We also recast Kato's definition in the language of Artin fans.
	
	\begin{definition}\label{defn:logflat}
		We say $\mathcal{F}$ is \textit{logarithmically flat} over $B$ if $\mathcal{F}$ is flat over  $$B_X  = B \times_{\mathpzc{B}}{\mathpzc{X}_B}$$ where $X$ is considered a $B_X$ scheme in the obvious way. We say $\mathcal{F}$ is \textit{integral} over $B$ if the map $\mathpzc{X}_B \rightarrow \mathpzc{B}$ is flat as a morphism of underlying algebraic stacks.
	\end{definition}
An important special case is to understand sheaves logarithmically flat over the standard log point $\mathsf{pt}^\dagger$. These are certain sheaves on logarithmic modifications of $X \times \mathsf{pt}^\dagger$. We call the underlying scheme of such a logarithmic modification an \textit{expansion}. The space of tropical supports will be upgraded to become something very similar to the stack of expansions introduced in \cite{MR20}. The link between the geometry of the stack of expansions and tropical geometry is explained in \cite{CarocciNabijou1,CarocciNabijou2}. 

We define a log point to be a logarithmic scheme with a single point. Our next lemma gives a necessary condition for a sheaf on $X\times S$ to be logarithmically flat over $S$. For $\sigma$ a cone of $\mathrm{Trop}(X)$ we write $V(\sigma)$ for the preimage of the closure of the point $a_\sigma$ of $\mathpzc{X}$ in $X\times S$. 
	\begin{lemma}\label{lem:transversality and closure} Set $\mathcal{F}$ a sheaf on the underlying scheme of $X$ where $X$ is a logarithmic scheme over a log point $S$. Assume $\mathcal{F}$ and $\mathcal{O}_X$ are logarithmically flat over $S$ then we have the following.
		\begin{enumerate}[(1)]
			\item For every cone $\sigma$ of $\mathrm{Trop}(X)$, the restriction $\mathcal{F}|_{V(\sigma)}$ has no sections with support contained in the complement of $O(\sigma)$.
			\item In the special case $\mathcal{F} = \underline{\mathcal{O}_Z}$ is the structure sheaf of a closed subscheme $Z$, if $\underline{\mathcal{O}_Z}$ is logarithmically flat then $Z\cap V(\sigma)$ is the closure of $Z \cap {O}(\sigma)$.
		\end{enumerate}
	\end{lemma}

	\begin{proof} We can investigate support locally, so restrict attention to a cone $\tau$ of $\mathrm{Trop}(X)$ such that $\sigma\leq \tau$.  Flatness is preserved under base change so we can assume $\sigma$ is the zero cone. Write $Y_\tau$ for the toric variety corresponding to the cone $\tau$ and observe there is a natural map $Y_\tau \rightarrow \mathrm{Trop}(X)$. Since flatness is preserved under base change we consider a Cartesian square $$
\begin{tikzcd}
X_\tau \arrow[d] \arrow[r,"g"] & X \arrow[d]                                         \\
S \times Y_\tau \arrow[r]  & S \times \mathpzc{X}
\end{tikzcd}$$
If $\mathcal{F}$ has a non--zero section supported on $V(\sigma)\backslash O(\sigma)$ then for the right choice of $\tau$ this pulls back to a non-zero section $s$ of $g^\star \mathcal{F}$. Define $Y_\tau^o$ the dense torus of $Y_\tau$. The support of $s$ is contained in the complement of the preimage in $X$ of $S\times Y_\tau^o$.
 
Assume for contradiction that the support of $s$ were contained in the preimage $Z$ of a closed subscheme $V(f)$ of $S\times Y_\tau$. Since the stalk of $s$ vanished away from $Z$, we know $s$ must vanish on complement of $Z$. Without loss of generality $X_\tau$ is affine and global sections of $\mathcal{F}$ are a module $M$. Thus passing to an affine patch $\mathrm{Spec}(A)$ of $S\times Y_\tau$, $s$ vanishes in the localisation of $M$ (considered an $A$ module) obtained by inverting $f$. In other words $f^ks=0$ for some $k$. But by flatness the morphism $ (f^k) \otimes M \rightarrow M$ sending $f^ka \otimes b \mapsto f^kab$ is injective.
	\end{proof}

For $S$ a log point, we say a sheaf on $X\times S$ is \textit{transverse} if it satisfies the first condition in Lemma~\ref{lem:transversality and closure}. For $S$ any logarithmic scheme say $\mathcal{F}$ is transverse over $S$ if the pullback of $\mathcal{F}$ is transverse over every strict closed point of $S$. Transversality is an interesting condition first because it is easier to check than logarithmic flatness. Second, we will see that transversality is the weakest condition to ensure the morphism $r_i$ exists. Third a theorem of Tevelev shows that the difference between transversality and logarithmic flatness is not something our moduli space detects. 
%
\begin{remark} Lemma \ref{lem:transversality and closure} is similar to the proof that $(1) \Rightarrow (3)$ in Theorem \ref{thm:CharacteriseLogFlat} except that $X$ need not be toric and we work in the relative situation.
\end{remark}
 
	\subsubsection{From logarithmic flatness to morphisms between Quot schemes}\label{sec:getmap} Let $X$ be a scheme and $D$ a divisor on $X$. For $S$ noetherian consider a coherent sheaf $\mathcal{F}$ on $X\times S$ which is flat over $S$. Let $$\iota_D: D \times S \rightarrow X \times S$$
	be the natural inclusion.
	
	\begin{proposition}\label{prop:transversality}
		The sheaf $\iota_D^{\star} \mathcal{F}$ is flat over $S$ if for each closed point $s$ of $S$ the pull back $\mathcal{F}_s$ has no sections supported on $D$.
	\end{proposition}

	\begin{proof}
		Flatness is local so we pass to affine patches. Let $M$ be a $B$ module where $B$ is an $A$ algebra. Assume $A,B$ are Noetherian and $M$ is finitely generated over $B$. Let $f\in B$ such that for any maximal ideal $m$ of $A$ multiplication by $f$ is an injective map on $M/mM$. Then $M$ is flat over $A$ implies $M/f$ is flat over $A$ by \cite[Theorem 22.6]{matsumuraCommutativeRings}. 
	\end{proof}
	
	Proposition \ref{prop:transversality} is evidence that transversality is the natural condition to consider if one hopes to extend the morphisms $$r_i:\mathsf{Quot}(X,\mathcal{E})\rightarrow \mathsf{Quot}(D,\mathcal{E}|_{D})$$ discussed in the introduction. In the sequel we fix a sheaf $\mathcal{E}$ on $X$.
	
	\begin{corollary}\label{Corr:MapsBetweenModuliSpaces}
		Let $\mathsf{Quot}(X,\mathcal{E})^{o}$ be the moduli space whose fibre over a logarithmic scheme $S$ consists of surjections of sheaves $q:\mathcal{E} \rightarrow \mathcal{F}$ on $X \times S$ such that $\mathcal{F}$ is transverse over $S$. 
		\begin{enumerate}[(1)]
			\item There is an open inclusion $$\mathsf{Quot}(X,\mathcal{E})^{o}\hookrightarrow \mathsf{Quot}(\underline{X},\mathcal{E}).$$
			\item Pullback to $D_i$ defines a morphism $$\mathsf{Quot}(X,\mathcal{E})^o\rightarrow \mathsf{Quot}(D_i,\iota_{D_i}^{\star} \mathcal{E})$$
		\end{enumerate}
	\end{corollary}
	\begin{proof}
		For statement (1) we need to verify the map is open which follows from \cite[Theorem~53]{MatsumuraHideyuki1970Ca/H}. Statement (2) follows directly from Proposition~\ref{prop:transversality}. 
	\end{proof}

 \subsection{Logarithmic surjections of coherent sheaves}
Let $\mathcal{E}$ be a coherent sheaf on a fine and saturated logarithmic scheme $X$ defined over $\mathrm{Spec}(\mathbb{C})$. The moduli space $\mathsf{Quot}(X,\mathcal{E})^{o}$ is not proper. The solution is to study a moduli space whose $S$ points are \textit{logarithmic surjections of coherent sheaves} of the pullback $\mathcal{E}_S$ on $X\times S$ flat over $S$. To simplify notation we drop the subscript $S$ from $\mathcal{E}_S$.
	\begin{definition}\label{defn:surjrelsheaves}
		A {logarithmic surjection of coherent sheaves }on $X \times S$ which is flat over $S$ is written  $[\pi_\Gamma,q_\Gamma]$ and is defined to be an equivalence class of pairs $$(\pi_\Gamma:(X\times S)_\Gamma \rightarrow X\times S, q_\Gamma:\pi_\Gamma^{\star} \mathcal{E} \twoheadrightarrow \mathcal{F}_\Gamma)$$ where $\pi_\Gamma$ is a logarithmic modification and $q_\Gamma$ is a surjection of coherent sheaves on $(X\times S)_\Gamma$. We require both $(X\times S)_\Gamma$ and $\mathcal{F}_\Gamma$ are logarithmically flat and integral over $S$.

        Given two pairs $((X\times S)_\Gamma,q_\Gamma)$ and $((X\times S)_{\Gamma'},q_{\Gamma'})$ set $\Gamma''$ to be the common refinement of $\Gamma,\Gamma'$. There are two logarithmic modifications $$\pi_{\Gamma'',\Gamma}:X_{\Gamma''}\rightarrow X_\Gamma \quad \pi_{\Gamma'',\Gamma'}:X_{\Gamma''}\rightarrow X_{\Gamma'}.$$ The equivalence relation is the smallest which identifies $(\pi_\Gamma,q_\Gamma)$ with $(\pi_{\Gamma'},q_{\Gamma'})$ whenever $$ \pi^{\star} _{\Gamma'',\Gamma}q_\Gamma = \pi^{\star} _{\Gamma'',\Gamma'}q_{\Gamma'}.$$
	\end{definition}
 
 We remark that the logarithmic modification of $X \times S$ corresponding to the common refinement of $\Gamma$ and $\Gamma'$ need not be integral over $S$. The next proposition shows that logarithmic flatness is preserved.
	\begin{proposition}
		Let $V$ be logarithmically flat over $S$. If $\mathcal{F}$ a coherent sheaf on $V$ is logarithmically flat over $S$, then given a logarithmic modification $$V_\Gamma \rightarrow V\textrm{ where }a^\star \Gamma \rightarrow \mathpzc{S}\textrm{ is flat,}$$ the pullback of $\mathcal{F}$ to a sheaf $\mathcal{F}_\Gamma$ on $V_\Gamma$ is also logarithmically flat over $S$. 
	\end{proposition}
	
	\begin{proof}
		Logarithmic flatness of $\mathcal{F}$ implies $\mathcal{F}$ is flat over $S\times_\mathpzc{S} \mathpzc{V}$. We are required to check $\mathcal{F}_\Gamma$ is flat over $a^\star \Gamma \times _\mathpzc{S} S$. A diagram chase shows all squares in the following diagram are cartesian. 
$$\begin{tikzcd}
V_\Gamma \arrow[r, "h'"] \arrow[d] & S\times_\mathpzc{S} a^\star\Gamma \arrow[d] \arrow[r] & a^\star\Gamma \arrow[d] \\
V \arrow[r, "h"]                   & S\times_\mathpzc{S} \mathpzc{V} \arrow[r]      &\mathpzc{V}            
\end{tikzcd} $$
Since flatness is preserved under base change \cite[TAG 01U8]{stacks-project}, we deduce $\mathcal{F}_\Gamma$ is flat over $S\times_\mathpzc{S} a^\star\Gamma$ and the result is proved.
	\end{proof}

\subsection{Proper monomorphisms and logarithmic surjections of coherent sheaves.}
The functor of points of Grothendieck's Hilbert scheme associated to $X$ assigns to a scheme $S$ the set of closed immersions to $X\times S$ which are flat over $S$. In the category of schemes closed immersions are precisely proper monomorphisms. Thus the data of an ideal sheaf is the same as the data of a proper monomorphism.
In the category of logarithmic schemes a strict closed immersion is a proper monomorphism \cite[Proposition 1.4, Property (v)]{MonomLogSch}, but there is a second class of proper monomorphism.

\begin{lemma}
    Logarithmic modifications are proper monomorphisms.
\end{lemma}
\begin{proof}
    A logarithmic modification is the base change of a tropical model. It is easy to see tropical models are monomorphisms and being a monomorphism is preserved under base change. The map of logarithmic stacks corresponding to a tropical model is always proper and since properness is preserved under base change, the result is proved.
\end{proof}

Let $X$ be a logarithmic scheme whose Artin fan is $a^\star\mathrm{Trop}(X)$ for some cone complex $\mathrm{Trop}(X)$. The next lemma shows that any proper monomorphism to $X$ can be understood in terms of logarithmic modification and strict closed immersion.

\begin{lemma}\label{lem:ProperMonom}
    Let $Z \rightarrow X$ be a proper monomorphism in the category of logarithmic schemes. There is a commutative square $$ 
\begin{tikzcd}
\tilde{Z} \arrow[r, "\iota_{\tilde{Z}}", hook] \arrow[d, "\pi_Z"'] & \tilde{X} \arrow[d, "\pi_X"] \\
Z \arrow[r, "g"]                                                        & X                           
\end{tikzcd}$$ where $\pi_Z$ and $\pi_X$ are logarithmic modifications and $\iota_{\tilde{Z}}$ is a strict closed immersion.
\end{lemma}

Lemma \ref{lem:ProperMonom} is asserting that an injection of logarithmic coherent sheaves on $X$ $$I_Z \hookrightarrow \mathcal{O}_X$$ is the same data as an equivalence class of proper monomorphism to $X$.

\begin{proof}
If $Z$ is atomic, the morphism $g$ induces a morphism of cone complexes $$\mathrm{Trop}(Z)\rightarrow \mathrm{Trop}(X)$$ and thus a piecewise linear subdivision of $\mathrm{Trop}(X)$. If $Z$ is not atomic chose an open cover of $Z$ by atomic schemes $Z_i$. Each $Z_i$ induces a subdivision. Take the common refinement to obtain the subdivision $$\mathscr{T}(Z)\rightarrow \mathrm{Trop}(X).$$  Choose a tropical model of $\mathscr{T}(Z)$ say $$\Sigma \rightarrow \mathrm{Trop}(X)$$ and write $\pi_X:\tilde{X}\rightarrow X$ for the corresponding logarithmic modification. Consider the fibre product of cone stacks of $\mathrm{Trop}(Z)$ with $\Sigma$ over $\mathrm{Trop}(X)$ and write $\pi_Z:\tilde{Z}\rightarrow Z$ for the corresponding logarithmic modification. Note the map $\tilde{Z}\rightarrow X$ factors through $\tilde{X}$ by the universal property of fibre product and by construction the map $\tilde{Z} \rightarrow \tilde{X}$ is strict.

We know both $g$ and $\pi_Z$ are monomorphisms. It follows that the composition of $\pi_X$ with $\iota_{\tilde{Z}}$ is a monomorphism and thus $\iota_{\tilde{Z}}$ is a strict proper monomorphism and is thus a closed immersion \cite[Proposition 1.4]{MonomLogSch}.
\end{proof}
We have thus shown that a logarithmic surjection of coherent sheaves with domain $\mathcal{O}_X$ is the data of an equivalence class of proper monomorphisms.
\subsection{The space of tropical supports} We turn our attention to $a^\star\supptrop(\mathpzc{X})$ which we denote $\supplog(\mathpzc{X})$. Where there is no ambiguity, we drop $\mathpzc{X}$ when writing $\supplog$ and $\supptrop$. Tropical models of $\supptrop$ are denoted by a subscript, $$\supptrop_\Sigma \rightarrow \supptrop.$$ We will write $a^\star\supptrop_\Sigma = \supplog_\Sigma$. 

Consider a diagram with all horizontal morphisms tropical models and all vertical morphisms combinatorially flat.
$$
\begin{tikzcd}
\mathcal{X}_\Sigma \arrow[r] \arrow[d] & \mathscr{X} \arrow[d] \\
\supptrop_\Sigma \arrow[r]         & \supptrop        
\end{tikzcd}
$$
\subsubsection{A flat map of Artin fans}\label{sec:aflatmapArtinFans} The morphism of Artin fans $$a^\star\pi: a^\star\mathcal{X}_\Sigma \rightarrow a^\star \supptrop_\Sigma$$ need not be flat. In this subsection we modify $\mathcal{X}_\Sigma$ in a canonical way to make the map flat. 

Our modification is specified on the level of cone complexes. In light of \cite[Theorem 2.1.4]{molcho2019universal} a morphism of Artin fans is flat if the corresponding morphism of cone complexes $\mathcal{X}_\Sigma \rightarrow \supptrop_\Sigma$ has the following two properties.
\begin{enumerate}[(1)]
\item The image of every element of $\mathcal{P}_{\mathcal{X}_\Sigma}$ lies in $\mathcal{P}_{\supptrop_\Sigma}$.
\item For every cone $\sigma$ of $\mathscr{X}_\sigma$ mapping to cone $\tau$ of $\supplog_\Sigma$ defines a surjection from the lattice of $\sigma$ to the lattice of $\tau$. 
\end{enumerate}
Our definition of combinatorial flatness ensures that (1) always holds. We define a diagram of logarithmic schemes $$\mathcal{X}_\Sigma\rightarrow \mathscr{X}_\Sigma \rightarrow \supplog_\Sigma.$$
where the first morphism is a logarithmic modification defined by modifying the lattice and the second has properties (1) and (2).

\begin{construction}\label{cons:flatmapArtinfan}
    We define the morphism $$\mathcal{X}_\Sigma\rightarrow \mathscr{X}_\Sigma \rightarrow \supplog_\Sigma.$$ We do so by replacing the local cones of $\mathcal{X}_\Sigma$. Fix $\mathfrak{s} = (N_\mathfrak{s},U_\mathfrak{s})$ a local cone of $\mathcal{X}_\Sigma$. The map to $\supplog_\Sigma$ defines a map from $\mathfrak{s}$ to a piecewise linear cone $\mathfrak{t} = (N_\mathfrak{t},U_\mathfrak{t})$. This is the data of a monoid morphism $$N_\mathfrak{s} \rightarrow N_\mathfrak{t}.$$ We define a new piecewise linear cone $(N_{\mathfrak{s}'},U_{\mathfrak{s}'})$ as follows. Whenever the image of $\mathfrak{s}$ in $\supptrop$ has the same dimension as $\mathfrak{s}$, we define the group $N_{\mathfrak{s}'}$ to be the subgroup of $N_{\mathfrak{s}'}\otimes \mathbb{Q}$ whose image in $N_\mathfrak{t}\otimes \mathbb{Q}$ lies in $N_\mathfrak{t}$. Identifying $N_\mathfrak{t}\otimes \mathbb{R} = N_{\mathfrak{t}'}\otimes \mathbb{R}$, we define $U_{\mathfrak{t}'} = U_\mathfrak{t}$. The abelian group corresponding to other cones is taken to be minimal compatible with the above choice.

    We define a topological space $|\mathscr{X}_\Sigma|=|\mathcal{X}_\Sigma|$ and $\mathcal{P}_\mathcal{X} = \mathcal{P}_\mathscr{X}$. We replace the homeomorphisms $$f_\kappa: \overline{\kappa} \rightarrow \overline{U}_{\mathfrak{t}_\kappa} \textrm{ with }f_\kappa': \overline{\kappa} \rightarrow \overline{U}_{\mathfrak{t}_\kappa'} .$$ The remaining data of the piecewise linear space $\mathscr{X}_\Sigma$ is inherited from $\mathcal{X}$.
\end{construction}

The algebro--geometric justification for making this change of lattice is that in the definition of logarithmic surjection of coherent sheaves we work up to logarithmic modification, and in particular changing the lattice. Thus, applying a logarithmic modification to the universal family over $\supplog$ is merely a convenience provided it preserves combinatorial flatness.

\subsubsection{Moduli of logarithmic modification}
Suppose one is interested in studying the moduli problem of proper monomorphisms. Our logarithmic Hilbert scheme has nice properties, but we pay the price of working up to logarithmic modification. 

A related question is to study the moduli space of logarithmic modifications of $X$ with no equivalence relation (and not allowing more generl proper monomorphisms). The resulting moduli space is the open subfunctor of $\supptrop$ obtained by discarding all strata admitting a map from a cone $\sigma$ such that the universal family does not pull back to a cone complex over $\sigma$.
 \section{Tropical support for constant degeneration}\label{sec:defineTropSupp}
Consider the projective, logarithmically flat and integral morphism of fine and saturated logarithmic schemes $V=X\times W\rightarrow W$. Let $\mathcal{E}$ be a sheaf on $V$ logarithmically flat over $W$. In this section we restrict attention to the case that $W = \mathrm{Spec}(A)$ is an affine atomic logarithmic scheme with Artin fan $\mathpzc{W} = a^\star \mathrm{Trop}(W)$ for $\mathrm{Trop}(W)$ a cone. We moreover impose that the image of $W$ in its Artin fan $\mathpzc{W}$ is a single point. In particular there exists a cone complex $\mathrm{Trop}(V)$ such that there is an equality of Artin fans $$\mathpzc{V}=\mathpzc{X}\times \mathpzc{W} = a^\star \mathrm{Trop}(V).$$ 

\begin{goal}
 To a logarithmic surjection of coherent sheaves $$q = [\pi_\Gamma:V_\Gamma \rightarrow V,q_\Gamma:\pi_\Gamma^\star \mathcal{E}\rightarrow \mathcal{F}_\Gamma]$$ on $V$ which is flat over $W$ we associate a subdivision $$\mathscr{T}(q)\rightarrow \mathrm{Trop}(V).$$ 
 \end{goal}
We call the piecewise linear space $\mathscr{T}(q)$ the \textit{tropical support} of $q$. A useful heuristic is that tropical support tracks the minimal logarithmic modification $$\pi_\Gamma: V_\Gamma \rightarrow V\textrm{ such that we can write }q = [\pi_\Gamma, q_\Gamma].$$ Morally none of our hypotheses on $W$ are necessary and we handle the general case in Section \ref{sec:TropSuppNotConstDegen}.
	\subsection{Defining tropical support}\label{sec:DefnTropSupp} 
Fix a representative $(\pi_\Gamma,q_\Gamma)$ for $q$ such that $V_\Gamma\rightarrow W$ is projective and write $\Gamma \rightarrow \mathrm{Trop}(V)$ for the tropical model corresponding to $\pi_\Gamma$. Note $V$ has a locally closed stratification pulled back from its Artin fan with strata $V_\sigma$ of $V$ indexed by cones $\sigma$ of $\mathrm{Trop}(V)$. Similarly $V_\Gamma$ has a locally closed stratification with strata $O(\gamma)$ indexed by cones $\gamma$ of $\Gamma$. We write the set of cones of $\Gamma$ whose interior is mapped to the interior of a cone $\sigma$ of $\mathrm{Trop}(V)$ by $\Gamma(\sigma)$.
	
\subsubsection{Torus actions and subdivisions} Fix a dimension $k$ cone $\gamma$ of $\Gamma(\sigma)$ and note there is an inequality $\mathrm{dim}(\sigma) = \ell\geq k$. There is a cartesian diagram $$
	\begin{tikzcd}
	O(\gamma) \arrow[r] \arrow[d]  & V_\sigma \arrow[d] \\
	B\mathbb{G}_m^{k} \arrow[r,"g"] & B\mathbb{G}_m^{\ell}. 
	\end{tikzcd}$$
	The cocharacter lattice of $\mathbb{G}_m^k$ may be identified with ${\gamma^\mathrm{gp}}$. The inclusion of $\gamma^\mathrm{gp}$ into $\sigma^\mathrm{gp}$ defines a morphism of tori $\mathbb{G}_m^{k} \rightarrow \mathbb{G}_m^{\ell}$ inducing the map $g$. There is a free action of a torus $T_\gamma \sim \mathbb{G}_m^{\ell-k}$ on $O(\gamma)$ defined by automorphisms of $\mathbb{G}_m^{\ell}$ up to the action of $\mathbb{G}_m^{k}$. The cocharacter lattice of the torus $T_\gamma$ is thus identified with the lattice $M_\sigma(\gamma) = \sigma^\mathrm{gp}/\gamma^\mathrm{gp}$. 
	
	A subtorus of $T_\gamma$ is then the data of a saturated subgroup of $M_\sigma(\gamma)$. The torus $T_\gamma$ acts on the set of surjections 
	$$\{q_{\Gamma,\gamma}:\pi_\Gamma^{\star} \mathcal{E}|_{V(\gamma)} \rightarrow \mathcal{F}|_{V(\gamma)}\}$$ by pullback. Denote the maximal subtorus of $T_\gamma$ stabilising $q_{\Gamma,\gamma}$ by $T_\gamma(q_\Gamma)$ and the corresponding subgroup of $M_\sigma(\gamma)$ by $L_\gamma(q_\Gamma)$. Observe there is an identification between cones in $\Gamma(\sigma)$ of which $\gamma$ is a face and cones of the star fan $\mathrm{St}_\gamma$.
	
	\begin{thmdefn}\label{lem:defnTropSupp}
		Fix $q_\Gamma:\pi_\Gamma^{\star} \mathcal{E}\rightarrow \mathcal{F}_\Gamma$ a surjection of sheaves on a logarithmic modification $\pi_\Gamma: V_\Gamma \rightarrow V$ with $\mathcal{F}_\Gamma$ logarithmically flat over $W$. There is an initial piecewise linear subdivision $\mathscr{T}(q_\Gamma)\rightarrow \mathrm{Trop}(V)$ with the following properties. 
		\begin{enumerate}[(1)]
			\item It is possible to factor $\Gamma \rightarrow \mathscr{T}(q_\Gamma) \rightarrow \mathrm{Trop}(V)$.
			\item Let $\gamma\leq \gamma'$ be cones of $\Gamma(\sigma)$ in $\mathrm{Trop}(V)$ where $a_\sigma$ lies in the image of the map $V \rightarrow \mathpzc{X}\times \mathpzc{W}$. The cone $\mathrm{St}_\gamma(\gamma')$ is contained in $L_\gamma(q)$ if and only if $\gamma'$ and $\gamma$ lie in the same stratum of $\mathcal{P}_{\mathscr{T}(q_\Gamma)}$.
		\end{enumerate}
		Moreover $\mathscr{T}(q_\Gamma)$ depends only on the logarithmic surjection of coherent sheaves $q = [q_\Gamma, V_\Gamma]$. We define the piecewise linear complex $\mathscr{T}(q)=\mathscr{T}(q_\Gamma)$ to be the \textit{tropical support} of $q$. 
	\end{thmdefn}
	Proof of this theorem occupies the remainder of this section. It will suffice to check the theorem Zariski locally on $V$. Indeed, given $\mathscr{T}(q|_{U_i})$ for $U_i$ a cover of $V$, the piecewise linear space $\mathscr{T}(q)$ is the common refinement of the piecewise linear subdivisions $\mathscr{T}(q|_{U_i})$. The reason we imposed an initial piecewise linear subdivision is that the map $V \rightarrow \mathpzc{V}$ may not be surjective. The initial condition ensures uniqueness.
\subsection{Logarithmic flatness and Gr\"obner theory}\label{sec:Grobner} 
We consider the following diagram in which the vertical maps are flat $$
\begin{tikzcd}
V_\Gamma \arrow[r] \arrow[d]                  & V \arrow[d]                     \\
W \times_\mathpzc{W} a^\star \Gamma \arrow[r] & W \times_\mathpzc{W}(\mathpzc{X}\times \mathpzc{W}).
\end{tikzcd}$$
Choosing cones $\gamma \leq \gamma'$ of $\Gamma(\sigma)$, we are interested in understanding the relation between $q|_{O(\gamma)}$ and $q|_{O(\gamma')}$. Our strategy is to work in coordinates where the problem reduces to Gr\"obner theory. 
 
 \subsubsection{Notation} The cone $\mathrm{St}_\gamma(\gamma')$ defines an affine toric variety $Y=Y(\gamma')$ over $\mathbb{C}$. Note $\gamma'$ identifies an open subset $U=U(\gamma')$ in $V(\gamma)$ which contains $O(\gamma')$ as a locally closed subscheme. Shrinking $V$ such that $V_\sigma = \mathrm{Spec}(A)$ is a sufficiently small affine, we may write a map of rings $$A \rightarrow A \otimes \mathbb{C}[Y]\textrm{ corresponding to }U \rightarrow V_\sigma.$$ The pullback of $\mathcal{E}$ to $V_\sigma$ has global sections defined by an $A$ module $E$. In  the sequel we often write $A[Y] = A \otimes \mathbb{C}[Y]$.

Restricting the surjection $q$ to $U(\gamma')$ and taking global sections defines a short exact sequence $$0 \rightarrow G \rightarrow E \otimes_A A[Y] \xrightarrow{q|_{U}} F \rightarrow 0.$$ Here $G$ is defined such that the sequence is exact.
Write $Y^o$ for the dense torus of $Y$ and let $G^o$ be the $A[Y^o]$ module generated by the image of $G$ under the localisation map $$E\otimes_A A[Y]\rightarrow E\otimes_A A[Y^o].$$ Taking global sections of $q|_{O(\gamma)}$ fits into the short exact sequence $$0 \rightarrow G^o \rightarrow E \otimes_A A[Y^o] \xrightarrow{\Gamma(q|_{O(\gamma)})} F^o \rightarrow 0.$$
Write $A[Y] = A[X_1,...,X_k,X_{k+1}^{\pm 1},...,X_{n}^{\pm 1}]$ and $A[\overline{Y}] = A[X_{k+1}^{\pm 1},...,X_{n}^{\pm 1}]$. We can then write a short exact sequence $$0 \rightarrow G' \rightarrow E \otimes_A A[\overline{Y}] \xrightarrow{\Gamma(q|_{O(\gamma')})} F' \rightarrow 0.$$
Our task is to relate $G^o$ and $G'$.

	\subsubsection{Initial degenerations for modules}  We define a \textit{monomial} of $A[Y] = A[X_1,...,X_k,X_{k+1}^{\pm 1},...,X_n^{\pm 1}]$ to be an element of the form $a X_i^j$ for some $a\in A$ and a primitive monomial if $a=1$. A \textit{monomial} of $E\otimes_AA[Y]$ is any element of the form $e\otimes m$ for $e\in E$ and $m$ a monomial of $A[Y]$.  
 
 We define a \textit{term order} $w$ to be a morphism of monoids from the set of monomials under multiplication to the real numbers under addition. The \textit{initial form} of an element $f = \sum_i m_i \in E\otimes_A A[Y^o]$ and a submodule $G^o \subset E\otimes_A A[Y^o]$, where $m_i$ are monomials, with respect to a term order $w$ are respectively $$\mathrm{in}_w(f) = \sum_{w(m_i) \mathrm{ maximal}} m_i \in E\otimes A[Y] \textrm{ and }\mathrm{in}_w(G^o) \textrm{ the }A[Y] \textrm{ module generated by } \{\mathrm{in}_w(f)|f\in G^o\}.$$ 
The same definitions work replacing $G^o$ by $G$.

Note a monomial on a toric variety $Y$ specifies a character of $Y^o$. Given a cocharacter $w$ the pairing between characters and cocharacters assigns to every monomial an integer. An element of the cocharacter lattice of $Y^o$ thus specifies a term order on monomials of $\mathbb{C}[Y^o]$ and thus on monomials of $E\otimes_A A[Y^o]$. 
	
	\subsubsection{Pullback to a trait}
	Let $S = \mathrm{Spec}(R)$ be a trait with generic point $\eta$ and closed point $s$. For a cone $\gamma$ of $\Gamma$ write $a_\gamma$ for the associated point of $a^\star \Gamma$. Choose a morphism $\varphi': S \rightarrow a^\star \Gamma$ such that the image of the generic point is $a_\gamma$ and the image of the special point is $a_{\gamma'}$. We probe flatness by studying the following enhancement of the diagram considered at the start of this section. $$
\begin{tikzcd}
V_s \arrow[r, "\hat{\iota}"] \arrow[d] \arrow[rr, "\iota", bend left] & V_S \arrow[r] \arrow[d,"\varpi"]          & V_\Gamma \arrow[r] \arrow[d]                  & V \arrow[d]                     \\
W \times_\mathpzc{W} s\arrow[r]                                                & W \times_\mathpzc{W} S \arrow[r,"\mathrm{id}\times \varphi'"] & W \times_\mathpzc{W} a^\star \Gamma \arrow[r] & W \times_\mathpzc{W}(\mathpzc{X}\times \mathpzc{W})
\end{tikzcd}$$ 
 This diagram is defined by declaring all squares cartesian.

We will assume $\gamma$ is the zero cone. The argument is identical in the general case, except that every short exact sequence must be tensored by the flat $\mathbb{C}$ module $\mathbb{C}[T_1^{\pm 1},...,T_{\mathrm{dim}(\gamma)}^{\pm 1}]$ at every step. This has no effect on our argument except to complicate notation.

We use two facts about this diagram. First $\varpi$ is flat because it is the base change of a flat map. Second pulling back $q$ to $V_s$ yields a short exact sequence $$0 \rightarrow G_\varphi = G' \otimes_\mathbb{C}M \rightarrow E \otimes_A A[\overline{Y}]\otimes_\mathbb{C}M \xrightarrow{q_s} F' \otimes_\mathbb{C}M \rightarrow 0$$
Here $M$ is the structure sheaf of a torus of rank $\mathrm{dim}(\gamma') - \mathrm{dim}(\gamma)$. It arises because the image of $s$ is stacky. Note that $G_\varphi$ depends only on the restriction of $\varphi$ to its special point. We choose an identification $$A[\overline{Y}]\otimes M\rightarrow A[Y^o]$$ extending the natural map $A[\overline{Y}]\rightarrow A[Y^o]$. In the sequel we think of $G_\varphi$ as a submodule of $E \otimes A[Y^o]$.

\subsubsection{Initial degenerations and traits} The morphism $\varphi$ specifies a cocharacter of the torus $Y^o$ and thus an element of the cocharacter lattice $m_\varphi \in M_{Y^o}$ inducing term order $w_\varphi$.
	
	\begin{proposition}\label{prop:initialideal}
		If $\mathcal{F}$ is logarithmically flat then there is an equality of submodules $$G_\varphi = \mathrm{in}_{w_\varphi}({G}^o) \leq \Gamma({\iota}^{\star}  \pi_\Gamma^\star \mathcal{E},Y_s) = E\otimes_A A[Y^o].$$
	\end{proposition}
	\begin{proof}
		Let $m': U \times Y^o \rightarrow U \times Y^o$ be the torus multiplication map. Consider a diagram 
		$$
\begin{tikzcd}
Y_S=S \times Y^o \arrow[d] \arrow[r]            & U \times Y^o \arrow[r, "m'"'] \arrow[d, "\pi_Y"] \arrow[rr, "m", bend left] & U \times Y^o \arrow[r, "\pi_U"'] \arrow[d, "\pi_Y\circ (m')^{-1}"] & U \arrow[d]                      \\
W\times S \arrow[r, "\varphi"] & W\times Y \arrow[r, "\mathrm{id}"]                             & W\times Y \arrow[r]                                   & W\times_{\mathpzc{W}}a^\star \Gamma
\end{tikzcd}$$
		Here we have used the fact $U = V_\sigma \times Y$ to define $\pi_Y$ as the composition $U \times Y^o \rightarrow U \rightarrow W\times Y$. Pulling back $q$ along $\pi_U$ defines a short exact sequence of sheaves which on the level of global sections is given by $$0 \rightarrow \mathrm{ker}(\pi_U^{\star} q)\rightarrow E\otimes_A A[Y] \otimes_\mathbb{C} \mathbb{C}[Y^o]\xrightarrow{\pi_U^{\star} q} F \rightarrow 0$$ Observe $\mathrm{ker}(\pi_1^{\star} q)$ is the submodule of $E\otimes_A A[Y] \otimes_\mathbb{C} \mathbb{C}[Y^o]$ generated by $g \otimes f$ where $g \in G$ and $f$ lies in $\mathbb{C}[Y^o]$.
		On the level of coordinate rings $m'$ is specified by $$ m'^\#:{A}[Y]\otimes_\mathbb{C} \mathbb{C}[Y^o]\rightarrow A[Y]\otimes_\mathbb{C} \mathbb{C}[Y^o].$$ $$  m\otimes f \mapsto m \otimes mf \textrm{ whenever }m\textrm{ a monomial in } \mathbb{C}[Y].$$ Consequently we have a short exact sequence of global sections 
  $$0 \rightarrow \mathrm{ker}(m^{\star} q) \rightarrow E \otimes_A A[Y]\otimes_\mathbb{C} \mathbb{C}[Y^o] \rightarrow {F}\otimes_A A[Y^o]  \rightarrow 0$$ where we have 
  $$\mathrm{ker}(m^{\star} q)= \left\langle\sum_i (e_i \otimes a_im_i) \otimes m_i| m_i \textrm{ a monomial and }\sum_i e_i \otimes a_im_i \in G^o \right \rangle. $$ Pulling back along $\varphi$ yields the short exact sequence of global sections of sheaves on $S\times Y^o$ \begin{equation}\label{eqn:ExSeq}
		0 \rightarrow {G}_S \rightarrow E \otimes_A A[X^o]\otimes R \xrightarrow{q_S}  {F}_S\rightarrow 0 .
		\end{equation} 
and we learn $\mathrm{ker}(q_S)$ is generated by $\sum_i e_i \otimes a_i m_i \otimes \varphi^\# (m_i)$ where $\sum_i e_i \otimes a_im_i$ is an element of $G$ and $m_i$ are monomials. Write $\pi$ for the uniformiser of $R$. Flatness of $F_S$ over $R$ ensures if $\pi g \in G_S$ for $g \in R[Y^o]$ then $g\in G_S$. Restricting to the special fibre of $S$, that is quotienting by the ideal $(\pi\otimes 1)$ yields the ideal $\textrm{in}_w({G}^o)$.
	\end{proof}
	

	\begin{corollary}\label{corr:inwconst}
		As $w$ varies within the interior of any cone of the star fan of $\gamma$ the module $\mathrm{in}_w(G^o)$ is constant.
	\end{corollary}
	\begin{proof}
		Consider $w$ a cocharacter of $Y^o$ specifying a point of $M_Y$ within the fan of $Y$. This is the data of a morphism $S^o \rightarrow (\mathbb{C}^{\star} )^n$. Since $w$ is contained in the support of the fan of $Y$ the morphism from $S^o$ extends to a morphism from $S$ to $Y$. We thus identify $G_\varphi$ with $\mathrm{in}_w(G^o)$ but the former depends only upon the cone in which $w$ lies.
	\end{proof}	
 \subsection{Gr\"obner theory and torus actions} There is a close link between Gr\"obner theory and torus actions. We continue notation from the previous section. Cocharacters of a torus specify the data of a one dimensional subtorus. Given cocharacter $w$ of $Y^o$ we write $T_w$ for the associated subtorus.
 
\begin{lemma}\label{lem:GrobIsTorus}
The morphism  $q|_{V(\gamma)}:\pi_\Gamma^\star E|_{V(\gamma)} \rightarrow \mathcal{F}_\Gamma|_{V(\gamma)}$ is invariant under the action of the one dimensional subtorus $T_w$ if and only if $$\mathrm{in}_w\left(\Gamma(O(\gamma),\mathrm{ker}(q|_{V(\gamma)}))\right)=\Gamma(O(\gamma),\mathrm{ker}(q|_{V(\gamma)})).$$ 
\end{lemma}
\begin{proof}
First note $\Gamma(\mathrm{ker}(O(\gamma),q|_{V(\gamma)}) = G^o$. The morphism $q|_{V(\gamma)}$ is characterised by its kernel so we need understand when $\mathrm{ker}(q|_{V(\gamma)})$ is fixed by our torus action. The submodule $G^o$ is fixed if and only if $\mathrm{in}_w(G^o) = G^o.$
\end{proof}
 
\begin{corollary}\label{Corr:UnionOfCones}
The set of points $w$ in the cocharacter lattice of $Y^o$ such that $\mathrm{in}_w(G^o) = G^o$ forms a saturated subgroup. If $\mathcal{F}$ is logarithmically flat then this subgroup is a union of cones.
\end{corollary}
\begin{proof}
    First sentence from Lemma \ref{lem:GrobIsTorus}. Second sentence also uses Corollary \ref{corr:inwconst}.
\end{proof}
We work in the setup of Section \ref{sec:DefnTropSupp}. The next lemma is important because it suggests how to prove Theorem \ref{lem:defnTropSupp}. Meditating on the consequences of this lemma suggests the right definition of piecewise linear space. Lemma \ref{lem:ConstDegenUnionCones} is saying we can do a sort of global version of Gr\"obner theory - thinking about cones without thinking explicitly about coordinate patches.
	\begin{lemma}\label{lem:ConstDegenUnionCones}
		There are cones $\kappa_1,...,\kappa_n$ in $\mathrm{St}_\gamma$ such that there is an equality of submonoids  $$\bigcup_i\left(\kappa_i\cap M_\sigma(\gamma)\right) = L_\gamma(q)\leq M_\sigma(\gamma).$$
	\end{lemma}
	\begin{proof}[Proof of Theorem \ref{lem:defnTropSupp}]
		Observe $V_\sigma$ contains an open subset $V_\sigma^o$ which is the preimage of the point $a_\sigma$ in $\mathpzc{V}$ associated to the cone $\sigma$. The preimage $V'(\gamma)$ of $V_\sigma^o$ in $V(\gamma)$ is a toric variety bundle over $V^o_\sigma$ Take an open cover $\{U_{i}\}$ of $V^o_\sigma$ whose preimage in $V(\gamma)$ is isomorphic to $U_{\sigma_i}\times Y$ for $Y$ a toric variety. Let $L_i$ be the maximal subtorus of $T_\gamma$ which stabilises $q|_{U_i}$. In the case $V_\sigma^o = V_\sigma$ we know $T_\gamma(q) = \cap_i L_i$; in fact this holds in general since being torus fixed is a closed condition and $V$ is logarithmically flat. It thus suffices to check each $L_i$ is a union of cones. Without loss of generality $U_i$ is affine and the result follows from Corollary \ref{Corr:UnionOfCones}. 
	\end{proof} 
\subsection{Construction}We now explain how to construct the piecewise linear space $\mathscr{T}(q_\Gamma)$. We do so by specifying a locally closed stratification of the interior of each cone $\sigma$ of $\mathrm{Trop}(V)$. 


\begin{construction}\label{Cons:TropSupp} For each choice of $\sigma$ a cone of $\mathrm{Trop}(V)$ with $a_\sigma$ in the image of $V \rightarrow \mathpzc{X}\times \mathpzc{W}$ we define a locally closed stratification $P_\sigma$ of $|\sigma|^o$ as follows. Define an equivalence relation $\sim$ on polyhedra of $\Gamma(\sigma)$ such that $\gamma \sim \gamma'$ if $\gamma'$ is a face of $\gamma$ and the image of $\gamma'$ in $\mathrm{St}_\gamma$ lies in $L_\gamma(q_\Gamma)$. We specify a locally closed stratification $\mathcal{P}$ of the interior of $\sigma$ by declaring $p,q \in |\sigma|^o$ lie in the same stratum if and only if $p,q$ lie in the interiors of cones related by $\sim$.

For each choice of $\sigma$ a cone of $\mathrm{Trop}(V)$ with $a_\sigma$ not in the image of $V \rightarrow \mathpzc{X}\times \mathpzc{W}$ we define a locally closed stratification $P_\sigma$ of $|\sigma|^o$ as follows. Define an equivalence relation $\sim$ on polyhedra of $\Gamma(\sigma)$ such that $\gamma \sim \gamma'$ if whenever $\sigma$ is a face of $\sigma'$ and $\sigma'$ is as in the first paragraph, $\gamma$ and $\gamma'$ lie in the closure of the same strata of $\mathcal{P}_{\sigma'}$. 

We define $\mathscr{T}(q)$ to be the unique piecewise linear space such that the subset of $\mathcal{P}_{\mathscr{T}(q)}$ mapping to the interior of $\sigma$ is $\mathcal{P}_\sigma$ for all $\sigma$.
\end{construction}

We next verify that the piecewise linear space ${\mathscr{T}(q)}$ built in Construction \ref{Cons:TropSupp} exists. Our task is to check the $\mathcal{P}_\sigma$ fit together to define a locally closed stratification. We first handle strata in $\Gamma(\sigma)$ for a fixed $\sigma$.

\begin{lemma}\label{lem:LooksLikePL}
Fixing any cone $\sigma$, there is a piecewise linear space subdividing $\sigma$ whose restriction to the interior of $\sigma$ is $\mathcal{P}_\sigma$.
\end{lemma}
\begin{proof}
It suffices to handle the case that $a_\sigma$ lies in the image of $V \rightarrow \mathpzc{X}\times \mathpzc{W}$. Lemma \ref{lem:ConstDegenUnionCones} implies each stratum of $\mathcal{P}_\sigma$ naturally carries the structure of a piecewise linear cone. It remains to check $\mathcal{P}_\sigma$ is a bone fide stratification. This follows by observing that $\gamma\leq \gamma'$ elements of $\Gamma(\sigma)$ and $q_\Gamma|_{O(\gamma)}$ is fixed by a subtorus of $Y_\gamma^o$ then $q_{\Gamma}|_{O(\gamma')}$ is fixed by the same subtorus. This property can be checked on any affine cover of $X$ and thus follows from our Gr\"obner theory characterisation of being torus fixed. 
\end{proof}

It remains to check that the locally closed stratifications $\mathcal{P}_\sigma$ fit together to form a piecewise linear complex as claimed. The only work we must do is verify that the locally closed stratification $\mathcal{P}_{\mathscr{T}(q)}$ is a stratification, not just a semistratification.

\begin{lemma}\label{lem:PLspaceConstDegen}
    Let $\kappa,\kappa'$ be strata of $\mathcal{P}_{\mathscr{T}(q)}$ such that $\kappa$ intersects the closure of $\kappa'$. Then $\kappa$ lies in the closure of $\kappa'$.
\end{lemma}
\begin{proof} We split the proof into three cases.

\noindent \textbf{Case I.} If $\kappa,\kappa'$ lie in the interior of the same cone $\sigma$ of $\mathrm{Trop}(V)$ then the result follows from Lemma~\ref{lem:LooksLikePL}. 

\noindent \textbf{Case II.} We consider the case that $\kappa,\kappa'$ lie in the interiors of different cones $\sigma\leq \sigma'$ both of whose interiors map to the interior of the same cone of $\mathrm{Trop}(W)$. Assume moreover that the associated point of $\mathpzc{W}$ lies in the image of the map $W \rightarrow \mathpzc{W}$. We must check that there are not cones $\gamma \leq \gamma'$ of $\Gamma(\sigma), \Gamma(\sigma')$ respectively with the following property. Observe a one dimensional sublattice of $\sigma^\mathrm{gp}$ specifies a one dimensional sublattice of $(\sigma')^\mathrm{gp}$ by inclusion. Thus it makes sense to define our property as $L_{\gamma'}(q) <L_\gamma(q)$. 

We define $K = \mathrm{ker}(\pi_\Gamma^\star \mathcal{E}\rightarrow \mathcal{F}_\Gamma)$ and note this kernel characterises $q_\Gamma$. We know there are no sections of $\mathcal{F}_\Gamma|_{V(\gamma)}$ whose support is contained within $\mathcal{F}_\Gamma|_{V(\gamma')}$. Knowing $K|_{O(\gamma)}$ thus determines $K|_{V(\gamma)}$ as the maximal subsheaf of $\pi_\Gamma^\star \mathcal{E}$ whose restriction to $O(\gamma)$ is $K|_{O(\gamma)}$. Thus if $K|_{O(\gamma)}$ is preserved when pulling back along some isomorphism, so is $K|_{V(\gamma)}$ and thus $K|_{V(\gamma')}$.

\noindent \textbf{Case III.} Now suppose $\kappa,\kappa'$ lie in different cones $\sigma\leq \sigma'$ which map to different cones of $\mathrm{Trop}(W)$. Then the result is a property of piecewise linear spaces following from point set topology.  
\end{proof}

\subsection{Proof of Theorem} We are now in a position to verify that the piecewise linear space of Construction~\ref{Cons:TropSupp} satisfies the hypothesis of Theorem \ref{lem:defnTropSupp}.
 	\begin{proof}[Proof of Theorem \ref{lem:defnTropSupp}]
		Properties (1) and (2) are clear. We check the tropical support is independent of the choice of $\Gamma$. Consider a refinement of subdivisions  $$ \Gamma' \rightarrow \Gamma \rightarrow \mathrm{Trop}(V)$$ with corresponding logarithmic modification  $$V_{\Gamma'} \xrightarrow{\pi_{\Gamma',\Gamma}} V_\Gamma \xrightarrow{\pi_\Gamma} V$$ 
		It suffices to check $\mathscr{T}(\pi^{\star} _{\Gamma,\Gamma'}q_\Gamma) = \mathscr{T}(q_{\Gamma})$. Equivalently the stratifications $\mathcal{P}_{\mathscr{T}(q_\Gamma)}$ and $\mathcal{P}_{\mathscr{T}(\pi^{\star} _{\Gamma',\Gamma}q_\Gamma)}$ of $|\mathrm{Trop}(V)|$ coincide. Let $\gamma'$ be a stratum of $\Gamma'$ with image contained in $\gamma$ a stratum of $\Gamma$. If $\Gamma$ has the same dimension as $\gamma'$ the strata $O(\gamma)$ and $O(\gamma')$ are isomorphic and the associated surjections of sheaves may be identified. Consequently $L_\gamma(q) = L_{\gamma'}(q')$. In general, since $q_{\Gamma'}$ is pulled back from $q_\Gamma$ we have $L_\gamma(q) = L_{\gamma'}(q')$. 
	\end{proof} 
 \section{Tropical support in families}\label{sec:TropSuppNotConstDegen}
We continue with the notation introduced at the start of Section \ref{sec:DefnTropSupp} except we drop all assumptions on $W$. For semantic convenience throughout we pretend $W$ has locally connected logarithmic strata and thus admits an Artin fan. In Remark \ref{rem:NOAFAN} we explain how to adapt our discussion to remove this hypothesis.

 \begin{goal}\label{goal:1}
 To a logarithmic surjection of coherent sheaves $$q = [\pi_\Gamma:(X\times W)_\Gamma \rightarrow X\times W,q_\Gamma:\pi_\Gamma^\star \mathcal{E}\rightarrow \mathcal{F}_\Gamma]$$ on $X$ flat over $W$ we associate a map $$\mathpzc{W} \rightarrow \supplog.$$
 \end{goal}

\subsection{Tropical support over atomic logarithmic schemes} All of the difficulty in achieving Goal~\ref{goal:1} appears when $W$ is atomic and has locally connected logarithmic strata. In this situation the Artin fan of $X \times W$ is $\mathpzc{X}\times \mathpzc{W}$. Our strategy is to use the universal property of $\supptrop$. The work we do in this section comes down to checking the definition of tropical support generalises nicely from the case that the image of $W$ in $\mathpzc{W}$ is a single point.

We now define tropical support for $W$ atomic with Artin fan $a^\star \sigma$ for some cone $\sigma$. The preimage of the closed point in $\sigma$ is a non--empty closed subscheme $V_W(\sigma)$ of $W$. For $q$ a logarithmic surjection of coherent sheaves flat over $W$ define $\mathscr{T}(q)$ to be $\mathscr{T}(q|_{V_W(\sigma)})$ as defined in the last section. By definition this is a combinatorially flat morphism of piecewise linear spaces, defining a map $$\mathpzc{W} \rightarrow \supplog.$$
It is far from clear that this construction satisfies descent in the strict \'etale topology as we vary $W$. In the remainder of this section we check strict descent.
\subsection{The relative Quot scheme trick} 
The relative Quot scheme trick is a local version of Maulik and Ranganathan's technique for constructing logarithmic Donaldson--Thomas spaces \cite{MR20}. Assume $W$ is atomic and fix a tropical model $\Gamma \rightarrow \mathrm{Trop}(V)$ such that $V_\Gamma$ is integral over $S$. There is a cartesian diagram $$
\begin{tikzcd}
V \arrow[r] \arrow[d] & \mathcal{X}_\Gamma=a^\star \Gamma \times_\mathpzc{X} X \arrow[d] \\
W \arrow[r]                   & \mathpzc{W}                                  
\end{tikzcd}$$
 and the sheaf $\pi_\Gamma^\star \mathcal{E}$ is pulled back from a sheaf $\mathcal{E}_\Gamma$ on $\mathcal{X}_\Gamma$. Logarithmic flat and integral over a base imply flatness, so a representative of a logarithmic surjection of coherent sheaves on $V$ is specified by a $\underline{W}$ point of Grothendieck's relative Quot scheme $Q=\mathrm{Quot}(\underline{\mathcal{X}_\Gamma}/\mathpzc{W},\mathcal{E}_\Gamma)$. Here we adopt the notation for any logarithmic scheme $V$, $\underline{V}$ is the underlying scheme.
 
 Not all $\underline{W}$ points of $Q$ give rise to logarithmic surjections of coherent sheaves because logarithmic flatness is stronger than flatness. We will often restrict attention to the open substack $Q^o$ of $Q$ whose points specify logarithmically flat quotients. This technique is useful because tropical support can be understood in terms of locally closed substacks of $Q^o$.

\subsection{Strict descent and base change} A strict map of atomic logarithmic schemes $f:V \rightarrow W$ induces a map of Artin fans $a^\star f_\mathrm{trop}:\mathpzc{V} = a^\star \tau \rightarrow \mathpzc{W} = a^\star \sigma$. A family of tropical supports $$\mathscr{T}\rightarrow \mathrm{Trop}(X)\times \sigma \rightarrow \sigma$$ over $\sigma$ pulls back along the map of cones $f_\mathrm{trop}$ to define a family of tropical supports $$f^\star_\mathrm{trop} \mathscr{T}\rightarrow \mathrm{Trop}(X)\times \tau \rightarrow \tau$$ over $\tau$. 

 \begin{proposition}\label{prop:Family gives Family}
 The subdivision $\mathscr{T}(q)$ respects strict base change. More precisely in the situation of the previous paragraph $$ f^\star_\mathrm{trop} \mathscr{T}(q) = \mathscr{T}(f^\star q).$$
 \end{proposition}
To prove Proposition \ref{prop:Family gives Family} we develop the relative Quot trick introduced in the last section. 

\subsubsection{Constant degeneration implies constant tropical support}\label{sec:ConstDegen} Assume now that the image of $W$ in $\mathpzc{W}$ consists of a single (stacky) point $a_s$. There is a diagram with all squares cartesian $$ 
\begin{tikzcd}
(\mathrm{Trop}(X)\times W)_\Gamma \arrow[r] \arrow[d] & \mathcal{X}_s \arrow[d] \arrow[r] & \mathcal{X}_\Gamma \arrow[d] \\
W \arrow[r]                                  & a_s \arrow[r, hook]                          & \mathpzc{W}.                 
\end{tikzcd}$$
To avoid stack issues pull $\mathcal{X}_s$ back along a map from $\mathrm{Spec}(\mathbb{C})$ to $a_s$ to define a scheme $X_s$. In the relative Quot scheme trick we consider morphisms $W \rightarrow \mathrm{Quot}(X_s,\mathcal{E}_\Gamma)=Q$. Observe any map from a scheme $S$ to the open subset $Q^o$ of $Q$ includes the data of a map from $S$ to $\mathpzc{W}$ which specifies a logarithmic structure on $S$ by pullback. In this way a point of $Q^o$ gives the data of a logarithmic surjection of coherent sheaves flat over a log point. The tropical support of a point of the scheme $Q^o$ is the tropical support of this logarithmic surjection of coherent sheaves.

\begin{lemma}\label{lem:TropSuppConst}
Proposition \ref{prop:Family gives Family} holds whenever the image of $W$ in $\mathpzc{W}$ consists of a single point.
\end{lemma}
Lemma \ref{lem:TropSuppConst} reinterprets and generalises \cite[Lemma 4.3.2]{MR20}. It establishes even without the hypothesis that the image of $W$ is a point, there is a locally closed stratification of $Q^o$ on which tropical support is well understood. Understanding tropical support in families then amounts to understanding how these locally closed pieces fit together. 

\begin{remark} Since the image of $W$ in $\mathpzc{W}$ is a single point, we know $W$ is connected. The data of $q$ gives a map $\underline{W}\rightarrow Q^o$. The tropical support of $W$ could be characterised as the tropical support of any point in the connected component of $Q^o$ in which the image of $\underline{W}$ lies. 
\end{remark}

The remainder of Section \ref{sec:ConstDegen} is a proof of Lemma \ref{lem:TropSuppConst}. We start by sketching the proof idea. Note first the substack of $Q^o$ consisting of points with fixed tropical support is constructible. Indeed, we are imposing that $q_\Gamma$ is fixed by certain torus actions (a closed condition) and not fixed by other torus actions (an open condition). A constructible subset is closed/open if and only if it is preserved under specialisation/generization respectively. Being torus fixed is a closed condition, so our task is to show that being torus fixed is also open in our situation.

Write $V(\gamma)$ for the closed stratum of $X_\Gamma$ corresponding to cone $\gamma$ and $O(\gamma)$ for the locally closed stratum. The automorphism torus associated to $O(\gamma)$ is denoted $T_\gamma$ and we denote the restriction of $\pi_\Gamma$ to $V(\gamma)$ by $$\pi_\gamma: V(\gamma)\rightarrow X.$$ The restriction of a quotient of $\mathcal{E}$ denoted $q_\Gamma$ to $V(\gamma)$ is written $$q_\gamma:\pi_\gamma^\star \mathcal{E}\rightarrow \mathcal{F}_\gamma.$$

We write $\overline{Q}_\gamma = \mathrm{Quot}(\underline{V(\gamma)}, \pi_\Gamma^\star \mathcal{E}|_{V(\gamma)})$ and $Q^o_\Gamma$ for the open logarithmically flat subscheme. By Corollary~\ref{Corr:MapsBetweenModuliSpaces}, there is a map $$Q^o \rightarrow Q^o_\gamma = \mathrm{Quot}(\underline{V(\gamma)}, \pi_\Gamma^\star \mathcal{E}|_{V(\gamma)}) \textrm{ sending } q_\Gamma \mapsto q_\gamma.$$ Let $S$ be a trait and chose a morphism from $S$ to $Q^o$. Write $s$ for the special point of $S$ and $\eta$ for the generic point. Our map specifies a surjection of sheaves on $X\times S$ written $q(S):\pi_X^\star \mathcal{E}\rightarrow \mathcal{F}_S$. The restriction of $q(S)$ to $V(\gamma)$ is $q_\gamma(S)$. The restriction of $q(S),q_\gamma(S)$ to the special and general fibre are denoted $q(s), q_\gamma(s)$ and $q(\eta), q_\gamma(\eta)$ respectively. We must verify the tropical support of $q(s)$ and $q(\eta)$ coincide. Since being torus fixed is closed, if these tropical supports did not coincide then we could find a cone $\gamma$ of $\Gamma$ with the following property. 

\noindent \textbf{Property $\star$}: There is a one parameter subgroup $T^o$ of $T_\gamma$ which fixes $q_\gamma(s)$ but does not fix $q_\gamma(\eta)$.

\noindent The proof of Lemma \ref{lem:TropSuppConst} is completed with the following steps.  Assume for contradiction that there exists a cone $\gamma$ and torus $T^o$ as in property $\star$. We write $\eta = \mathrm{Spec}(K)$.

\noindent \textbf{Step I}. Construct a second map from the trait $S$ to $\overline{Q}_\gamma$. We write $S=S''$ when referring to this second family to avoid confusion: the special point of $S''$ is $s''$ and the generic point $\eta''$. The construction ensures that $q_\gamma(s'')=q_\gamma(s)$ and $$q_\gamma(\eta''): \mathcal{E}'' \rightarrow \mathcal{F}''_{\eta''}$$ is invariant under the action of $T^o$. Deduce the images of $S$ and $S''$ lie in the same connected component of $\overline{Q}_\gamma$.

\noindent \textbf{Step II}. Use the action of $T^o$ to construct a 
map $S'=\mathrm{Spec}(K[[t]])\rightarrow \overline{Q}_{\gamma}$. The construction has the property that the restriction of the universal surjection to the special fibre fits into a diagram $$\mathcal{E}'' \rightarrow \mathcal{F}'_{s'}\xrightarrow{p} \mathcal{F}''_{\eta''}$$ where the kernel of $p$ is non trivial. Use this to argue that $\mathcal{F}''_{\eta''}$ and $\mathcal{F}'_{s'}$ have different Hilbert polynomials and thus the images of $S''$ and $S'$ do not lie in the same connected component of $\overline{Q}_{\gamma}$.

\noindent \textbf{Step III.} Observe the map $\eta'\rightarrow \overline{Q}_{\gamma}$ admits a lift to a map $R = \mathrm{Spec}(K[T^{\pm 1}]) \rightarrow \overline{Q}_{\gamma}$ and we may factor $$\eta \rightarrow R \rightarrow \overline{Q}_{\gamma}.$$ Thus the images of $S$ and $S'$ lie in the same connected component of ${\overline{Q}_\gamma}$.

The final sentences of each step cannot all be true and we obtain a contradiction.
 \begin{proof}
We execute steps in the proof outlined above. Note $T^o$ specifies a ray $\rho$ in the star fan of $\gamma$. After subdividing $\gamma$ we may assume $\rho$ is a ray in this star fan and that moreover the projection map induces a combinatorially flat map of smooth fans $\mathrm{St}_\gamma \rightarrow \rho$. The closed stratum associated to $\mathrm{St}_\rho$ is denoted $V(\rho)$.

\textbf{Step I}. Transversality means we can restrict the universal surjection $q_\gamma(S)$ to $V(\rho)$ and still have a family flat over $S$. The map $V(\gamma) \rightarrow V(\rho)$ is flat because it is pulled back from a combinatorially flat map of smooth fans. Thus pulling $q_\gamma(S)|_{V(\rho)}$ back to $V(\gamma)$ gives our new map from $S$ to $\overline{Q}_\gamma$.

\textbf{Step II.} The action of $T^o$ on $q$ defines a map $R\rightarrow Q^o$ defining a map $\eta'\rightarrow Q^o$. By properness we can take a limit obtaining $S'\rightarrow \overline{Q}_\gamma$. 

We can understand the surjection of sheaves associated to $s$ on affine patches via Gr\"obner theory. Choose an affine patch $U_i$ of $X_\sigma$. Assume $q$ restricted to the affine patch of $X_\Gamma$ formed by intersecting $U_i$ with the locally closed stratum associated to $\rho$ fits into a short exact sequence of global sections $$ 0 \rightarrow G \rightarrow E \otimes_A A[T, Y_1^{\pm 1},...,Y_n^{\pm 1}]\rightarrow {F}\rightarrow 0.$$ Here we pick coordinates such that $T^o$ acts trivially on the $Y_i$. Associated to $T^o$ is then a term order $w$ assigning each $Y_i$ to $0$ but $T$ to $1$.

With the above setup, $q_\gamma(s')$ on our affine patch has global sections fitting into the short exact sequence 
$$ 0 \rightarrow G' \rightarrow E \otimes_A A[T, Y_1^{\pm 1},...,Y_n^{\pm 1}]\rightarrow {F}_{s'}\rightarrow 0$$
where $G'$ is $\mathrm{in}_w(G)$ for $w$ the term order defined by $T^o$. On the same affine patch we know $q(\eta'')$ is pulled back from a surjection of sheaves on $V(\rho)$ and so fits into the short exact sequence of global sections $$0 \rightarrow G'' \rightarrow E \otimes_A A[T, Y_1^{\pm 1},...,Y_n^{\pm 1}]\rightarrow {F}_{\eta''}\rightarrow 0.$$ Here $G' = \mathrm{in}_{-w}(G)$. Note $G'$ and $G''$ can only coincide if $G$ were generated by elements homogenous in $T$. If $G$ were so generated then by transversality $G$ is generated by polynomials in which $T$ does not appear. Thus $q_\gamma$ restricted to this the preimage of $U_i$ is torus fixed, which cannot happen for all affine patches $U_i$ lest $q$ be torus fixed. 

Hilbert polynomial is additive in short exact sequences. There is a short exact sequence $$\mathrm{ker}(p) \rightarrow  \mathcal{F}'_{s'}\rightarrow \mathcal{F}''_{\eta''}\rightarrow 0$$ and we now know $\mathrm{ker}(p)$ is not the zero sheaf because we saw this on affine patches. Thus the Hilbert polynomials of $\mathcal{F}'_{s'}$ and $\mathcal{F}''_{\eta''}$ differ.

\textbf{Step III.} Immediate from above analysis.
\end{proof}
\subsubsection{Dropping the constant degeneration hypothesis} We extend the analysis of Section \ref{sec:ConstDegen} by dropping the assumption that the image of $W$ in $\mathpzc{W}$ is a single point. We still assume that $W$ admits a strict map to an Artin cone. Fix a piecewise linear subdivision $\mathscr{T}$ of $\mathrm{Trop}(W)\times \mathrm{Trop}(X)$ such that $\Gamma$ is a tropical model of $\mathscr{T}$. We say a point $p$ of $Q^o$ is \textit{compatible with $\mathscr{T}$} if the tropical support associated to $p$ is obtained by restricting $\mathscr{T}$ to the preimage of the appropriate face of $\mathrm{Trop}(W)$.

\begin{proposition}\label{prop:StabilityIsOpen}
    Being compatible with $\mathscr{T}$ is an open condition on $Q^o$.
\end{proposition}

\begin{proof}
The proof of Lemma \ref{lem:TropSuppConst} carries through almost vis a vis. The difference is constructing the map $$S'' \rightarrow \overline{Q}_\gamma$$ in Step I. We explain how to adapt the construction to the present setting. The image of $s$ in $\mathpzc{W}$ is $a^\star\sigma_s$ and $\eta$ maps to $a^\star \sigma_\eta$. There are cones $\gamma_1,...,\gamma_q$ of $\Gamma$ minimal with the property that the interior of $\gamma_i$ maps to the interior of $\sigma_\eta$ and $\gamma$ is a face of $\gamma_i$. 

After possibly subdividing we may assume the star fan of each $\gamma_i$ is equivariant and smooth. For each $\gamma_i$ there is a cone $\rho_i$ of relative dimension one over $\mathrm{Trop}(W)$ corresponding to the direction defined by $T^o$ in the star fan of $\gamma_i$. Subdividing further we may assume the map from $V(\gamma_i)$ to $V(\rho_i)$ is combinatorially flat and thus flat. The the union of the $V(\gamma_i)$ is flat over the union of the $V(\rho_i)$. Indeed for fixed $i$ this follows as in the constant degeneration case; it suffices to check for each $\rho_i$ by the valuative criterion for flatness. 

Note $\rho_i$ are precisely cones which contain a fixed cone $\rho$ lying over $\eta$ in their image. There is now a flat map $V(\gamma) \rightarrow V(\rho)$ and the proof goes through as before.  
\end{proof}
\subsubsection{When $W$ is not atomic} We finish this subsection by removing the hypothesis that $W$ is atomic. Any logarithmic scheme $W$ admits an \'etale cover by atomic logarithmic schemes $W_i\rightarrow W.$ The expanded sheaf can be pulled back to $W_i$ and thus we have already constructed a morphism from each $\mathpzc{W}_i$ to $\mathpzc{Supp}$. The Artin fan of $\mathpzc{W}$ is the colimit of the Artin fans of each $\mathpzc{W}_i$. 

 \begin{proof}[Proof of Proposition \ref{prop:Family gives Family}.] Tropical support is controlled in the logarithmic Quot space trick by the image of $W$ in the relative Quot scheme. We need to show each logarithmic stratum of $W$ is mapped to the locally closed subscheme compatible with $\mathscr{T}(q)$.

If this were not true there would be a locally closed subscheme $W'$ of $W$ which is mapped to $a_{\sigma'}$ for some $\sigma'$ such that the tropical support over $W'$ was not the restriction of $\mathscr{T}(q)$ to $\sigma'$. Pick $\sigma'$ maximal with this property and observe $W'$ must contain a point of $W$ in its closure mapped to $a_{\sigma''}$ such that $\sigma'$ is a proper face of $\sigma''$ (else $W$ is not atomic). This contradicts Proposition \ref{prop:StabilityIsOpen} unless $\sigma'$ is maximal. For $\sigma'$ maximal $V(\sigma')$ is connected (else $W$ is not atomic).
\end{proof}

\begin{remark}\label{rem:NOAFAN}
    Throughout we have assumed that $W$ has locally connected logarithmic strata and thus we can speak of the Artin fan of $W$. This was convenient for stating clean results but not necessary to obtain a natural map $$W \rightarrow \supplog$$ which \'etale locally on $W$ factors through a strict map from $W$ to an Artin fan. As in the case $W$ admits an Artin fan it suffices to work \'etale locally on $W$ so we may assume our logarithmic modification of $W \times X$ is pulled back from a morphism of Artin fans $$\Gamma \rightarrow a^\star\sigma\times \mathpzc{X}.$$ Here $a^\star \sigma$ is any Artin cone such that there is a strict map $W \rightarrow a^\star\sigma$. Replacing $\mathpzc{W}$ by $a^\star \sigma$ the proof of Proposition \ref{prop:StabilityIsOpen} remains valid. However the map from $W$ to $\supplog$ need not factor through $a^\star \sigma$ as tropical support need not be constant. Instead we use Proposition \ref{prop:StabilityIsOpen} to obtain a strict \'etale open cover of $W$ on which Proposition \ref{prop:Family gives Family} holds. The map $W \rightarrow \supplog$ can now be defined as in the proof of Proposition \ref{prop:Family gives Family}.
\end{remark}

\section{Flat limits after Tevelev.}\label{sec:GenTev}
The goal of this section is to develop the techniques needed to show the logarithmic Quot space is proper. Let $\underline{S}$ be a trait with generic point $\underline{\eta}$ and consider a sheaf $\mathcal{F}$ on $X\times S$ which is flat over $\underline{S}$. Define a logarithmic scheme $\eta$ by equipping $\underline{\eta}$ with either of the following logarithmic structures
\begin{enumerate}
\item Case 1: equip $\eta$ with the trivial logarithmic structure.
\item Case 2: equip $\eta$ with logarithmic structure with ghost sheaf $\mathbb{N}$. 
\end{enumerate}
Assume the pullback of $\mathcal{F}$ to $X\times \eta$ is logarithmically flat over $\eta$.
\begin{theorem}\label{Goal:FlatLimits}
    In both Case 1 and Case 2 there is a logarithmic structure on $\underline{S}$ extending the the logarithmic structure on $\eta$ defining a logarithmic scheme $S$ with the following property. There is a logarithmic modification $$\pi_\Gamma: (X\times S)_\Gamma \rightarrow X\times S$$ such that the strict transform of $\mathcal{F}$ under $\pi_\Gamma$ is logarithmically flat and integral over $S$.
\end{theorem}

Theorem \ref{Goal:FlatLimits} is the technical ingredient required to show that the logarithmic Quot space is universally closed. Our argument follows \cite[Theorem 4.6.1, especially Section 7]{MR20}
 	\subsection{Generically trivial logarithmic structure after Tevelev}\label{sec:Transversality}
	We first handle Case 1 where the generic point of $S$ has trivial logarithmic structure.
 \subsubsection{Logarithmic structure on $S$} Define logarithmic scheme $S$ by equipping $\underline{S}$ with the divisorial logarithmic structure from its special point. We are left to find the logarithmic modification $\pi_\Gamma$.

 \subsubsection{Reduction to the toric case} In this section we reduce the proof of Theorem \ref{Goal:FlatLimits} in Case 1 to the situation that $X$ is toric equipped with its toric boundary divisors.
 \begin{lemma}\label{Lem:justdotoriccase}
 To prove Theorem \ref{Goal:FlatLimits}, Case 1 it suffices to check the case $X$ is a toric variety.
 \end{lemma}
 \begin{proof} 
 Being logarithmically flat is a local property so it suffices to ensure logarithmic flatness on an open cover. Observe $\mathbb{A}^n$ may be covered by very affine varieties. It follows that any scheme admits an open cover by very affine varieties. Take such an open cover $\{U_i\hookrightarrow X\}$. Since $D$ is simple normal crossings without loss of generality $D|_{U_i} = V(f_1,...,f_k)$ for some regular sequence $f_i$. The functions $f_1,...,f_k$ define a function from $U_i$ to $\mathbb{A}^k$. Combining this with the fact $U_i$ is very affine we can consider $U_i$ a closed subscheme of $(\mathbb{C}^\star)^\ell \times \mathbb{A}^k$. This embedding has the property that the toric boundary of $\mathbb{A}^k$ pulls back to the divisor $D\cap U_i$.

We now consider a coherent sheaf $\mathcal{F}$ on $U_i$. Any such coherent sheaf can be pushed forward to a coherent sheaf on the toric variety $(\mathbb{C}^\star)^\ell \times \mathbb{A}^k$. Suppose we find a toric modification $U_i(\Gamma)$ of $(\mathbb{C}^\star)^\ell \times \mathbb{A}^k$ such that the strict transform of $\iota_\star \mathcal{F}$ is logarithmically flat over $S$. The strict pullback of a logarithmic modification is a logarithmic modification and thus we obtain a logarithmic modification of $X\times S$.  

Any toric morphism to $\mathbb{A}^1$ is flat and thus every logarithmic modification of $X\times S$ is integral over $S$ (whose log structure corresponds to the cone $\mathbb{R}_{\geq 0}$). Logarithmic flatness of $\mathcal{F}$ over $S$ follows from the same statement for $\iota_\star \mathcal{F}$.
\end{proof}	

\subsubsection{Toric case} We prove Case 1 of Theorem \ref{Goal:FlatLimits} in the case $X$ is toric equipped with divisorial logarithmic structure from its toric boundary. In light of Section \ref{Lem:justdotoriccase} this completes our analysis of Case 1.
	\begin{proposition}\label{prop:GenTev}
		Given a coherent sheaf on $X\times S$ which is logarithmically flat over $X\times \eta$, there is a logarithmic modification $$\pi_\Gamma:(X\times S)_\Gamma \rightarrow X\times S$$ such that the strict transform of $\mathcal{F}$ is logarithmically flat.
	\end{proposition}
	
	Proposition~\ref{prop:GenTev} is a translation of a mild upgrade of Theorem \ref{thm:OldTev}.
	
	\begin{proof} For the strict transform of $\mathcal{F}$ to be logarithmically flat we require first that the morphism of Artin fans $$a^\star(\mathrm{Trop}(X)\times \mathrm{Trop}(S))_\Gamma \rightarrow a^\star\mathrm{Trop}(S)$$ is flat. This is true for any morphism of Artin fans to $\mathrm{Trop}(S) = \mathbb{R}_{\geq 0}$ mapping the generic point to the generic point.   

    Thus it suffices to ensure that the sheaf $\mathcal{F}$ on $X \times S$ is flat in the usual sense over $$X \times S \rightarrow \mathpzc{X}\times \mathpzc{S} \times_\mathpzc{S} S = \mathpzc{X}\times S.$$ The map $$X \times S  \rightarrow [X \times S/\mathbb{G}_m^k] = \mathpzc{X}\times S$$ is a global quotient, where $X$ is a toric variety of dimension $k$ and $\mathbb{G}_m^k$ acts as the dense torus of $X$. 

    We first rephrase the requirement that $\mathcal{F}$ is flat over $\mathpzc{X}\times S$ without the language of stacks. Consider the map $$\Psi: \mathbb{G}_m^k \times (X\times S) \rightarrow X \times S.$$  $$(g,y) \mapsto g^{-1}y.$$ Note $\mathcal{F}$ is flat over $\mathpzc{X}\times S$ if and only if $\Psi^\star \mathcal{F}$ is flat with respect to the projection map $\mathbb{G}_m^k \times X \rightarrow X$. 
    
    We must now show that $\Psi^\star \mathcal{F}$ can be flattened by an equivariant blowup of $X\times S$. Replacing $S$ with $\mathbb{A}^1$ this follows from Theorem \ref{thm:OldTev}. The same proof works for $S$ a trait.
\end{proof}

\subsection{The general case}
We now handle Case 2 where the generic point of $S$ has logarithmic structure with ghost sheaf $\mathbb{N}$. The basic strategy is to reduce to Case 1 which we already know how to handle. 
\subsubsection{Logarithmic structures on $S$}\label{sec:LogStr} The first part of Theorem \ref{Goal:FlatLimits} requires us to specify a logarithmic structure on $S$. A sufficient class of such logarithmic structures extending the logarithmic structure on $\eta$ with ghost sheaf $\mathbb{N}$ are the \textit{logarithmic extensions} introduced in \cite[Section 7.1]{MR20}.

\subsubsection{Construction}
In this subsection we construct a logarithmic structure on $\underline{S}$ defining logarithmic scheme $S$; and a logarithmic modification of $X\times S$. Together these prove Theorem \ref{Goal:FlatLimits}. 

\textbf{Setup, notation and link to Case 1.} We are given a logarithmic modification $(X\times \eta)_\Gamma$ of $X \times \eta$. This is the same data as a polyhedral subdivision $\Gamma_1 \rightarrow \mathrm{Trop}(X)$. Taking the cone over $\Gamma_1$ recovers a subdivision $\Gamma$ of $\mathrm{Trop}(X)\times \mathrm{Trop}(\eta) = \mathrm{Trop}(X)\times \mathbb{R}_{\geq 0}$. 

Each vertex of $\Gamma_1$ specifies a ray $\gamma$ in $\Gamma$ and thus a closed subscheme $V(\gamma)\times \eta$ of $(X\times \eta)_\Gamma$. We now define a new logarithmic structure on $V(\gamma)\times \eta$ and a morphism from the resulting logarithmic scheme $V(\gamma)\times \eta^\dagger$ to $\eta^\dagger$. Here we define ${\eta}^\dagger$ to be the logarithmic scheme obtained by equipping the underlying scheme of $\eta$ with the trivial logarithmic structure.

\textbf{Defining new logarithmic structure.} Let $\mathcal{A}$ be the scheme underlying an Artin fan. The scheme $\mathcal{A}$ carries a natural logarithmic structure as an Artin fan. Thus specifying a map from a scheme $W$ to $\mathcal{A}$ specifies a logarithmic structure on $W$. This is the logarithmic structure pulled back from the natural logarithmic structure on $\mathcal{A}$.

Our strategy for defining a new logarithmic structure on the scheme $\underline{V(\gamma)\times \eta}$ underlying $V(\gamma)\times \eta$ is to give a map to the underlying algebraic stack of an Artin fan. We already have one such map given by restricting the underlying map of schemes $$(X \times S)_\Gamma \rightarrow a^\star \Gamma.$$ We know the image of $V(\gamma)\times \eta$ lies in the closed substack corresponding to the Star fan of $\gamma$ denoted $a^\star \mathrm{St}_\gamma$. This closed substack is itself the stack underlying an Artin fan. The logarithmic structure on $({V(\gamma)\times \eta})^\dagger$ is pulled back from the logarithmic structure of the Artin fan with underlying stack $a^\star \mathrm{St}_\gamma$. We can express this logarithmic scheme as a product $$({V(\gamma)\times \eta})^\dagger= V(\gamma)^\dagger\times \eta^\dagger.$$ 

\textbf{Data of Case 1.} There is a natural map $$(V(\gamma) \times \eta)^\dagger \rightarrow {\eta}^\dagger.$$ On underlying schemes this is the projection map. Since $\eta^\dagger$ has the trivial logarithmic structure, there is a unique way to upgrade this to our morphism of logarithmic schemes.

\textbf{Applying Case 1.} Write $S^\dagger$ for the logarithmic scheme obtained by equipping $S$ with logarithmic structure from its special point. By Proposition \ref{prop:GenTev} there is a logarithmic modification $$\pi^\dagger_{\gamma'}:(V(\gamma)^\dagger\times S^\dagger)_{\gamma'}\rightarrow V(\gamma)^\dagger\times S^\dagger$$ such that the strict transform of $\mathcal{F}$ is logarithmically flat over ${S}^\dagger$. The data of $(V(\gamma)^\dagger\times S^\dagger)_{\gamma'}$ is a subdivision $$\gamma' \rightarrow \mathrm{Trop}(V(\gamma)^\dagger)\times \mathbb{R}_{\geq 0}.$$

\textbf{Upgrading the rank.} We now upgrade $\pi^\dagger_{\gamma'}$ to $$ \pi_\gamma':(V(\gamma)\times {S}_\sigma)_{\gamma'}\rightarrow V(\gamma)\times {S}_\sigma$$ where the logarithmic scheme ${S}_\sigma$ is the logarithmic extension of $\eta$ corresponding to the cone $\mathbb{N}^2$.
Indeed we simply take the fibre product of $\pi_{\gamma'}^\dagger$ with $\mathsf{pt}^\dagger$. We adopt the convention $\sigma = \mathbb{N}^2$ with the first copy of $\mathbb{N}$ being the ghost sheaf at the generic point and the second copy corresponding to the uniformiser of the special point of $\underline{S}$.

\textbf{Gluing data associated to each vertex.} Write $3\varepsilon$ for the rational number which is the smallest lattice distance between two vertices in $\Gamma_1$. For each vertex $\gamma_i$ in $\Gamma_1$ identify a two dimensional cone $\sigma_i$ in $\sigma$ and containing the ray $\{(i,0): i \in \mathbb{N}\}$. 

By our previous discussion, for each vertex $\gamma_i$ of $\Gamma$ we get a subdivision $\gamma_i'$ of $\mathrm{St}(\gamma_i)\times \mathbb{N}^2$ which is the trivial subdivision over the ray $\{(0,i)|i \in \mathbb{N}\}$. By taking a fibre we can think of each point $p$ in $\mathbb{N}^2$ as specifying the data of a polyhedral subdivision of $\mathrm{St}(\gamma_i)$. For $p$ in the $x$ axis the polyhedral subdivision is the same. By shrinking $\sigma_i$ we may assume that whenever the $x$ coordinate is at most one, every vertex in the subdivision $\gamma_i'$ lies within lattice distance $\epsilon$ of the vertex associated to the zero cone in $\mathrm{St}_{\gamma_i}$.

Set $\sigma_0 \subset \sigma$ a smooth subcone contained in each $\sigma_i$. We now define $S = S_{\sigma_0}$. To finish our construction we must specify a subdivision of $\mathrm{Trop}(X)\times \sigma_0.$ A piecewise linear subdivision on an $\varepsilon$ neighbourhood of each cone $\gamma_i$ corresponding to a vertex of $\Gamma_1$ is specified by the subdivision $\gamma_i'$. Each subdivision extends to a subdivision of $\mathrm{Trop}(X)\times \sigma_0$. Take the common refinement $\mathscr{T}$ of the piecewise linear subdivision induced by each $\gamma_i$. Choose any smooth subdivision $\Gamma\rightarrow \mathrm{Trop}(X)\times S_{\sigma_0}$. Shrinking $\sigma_0$ if needed and pulling back the resulting subdivision as in \cite{molcho2019universal}, without loss of generality the image of each stratum of $\Gamma$ is a cone of $\sigma_0$ . 

\subsubsection{Verifying construction works} First observe the map of Artin fans $$a^\star\Gamma \rightarrow a^\star \sigma_0$$ is flat by miracle flatness. By construction the image of each cone is a cone and thus fibre dimension is constant. The base is regular because $\sigma_0$ was chosen to be smooth. This handles being integral and we are left to check the pullback of $\mathcal{F}$ is logarithmically flat.

It remains to check the strict transform $\mathcal{F}_\Gamma$ of $\mathcal{F}$ to $(X \times S)_\Gamma$ is logarithmically flat over $S$. This is the same as checking $\mathcal{F}_\Gamma$ is flat in the usual sense over the codomain of the morphism $$ (X \times S)_\Gamma \rightarrow a^\star \Gamma \times_\mathpzc{S} S.$$ This morphism is of finite presentation over a noetherian base (since $S$ is a trait) so we may appeal to the valuative criterion for flatness \cite[11.8.1]{EGA3}. The target has a cover by schemes $V(\gamma_i)$ where $\gamma_i$ is a cone corresponding to a vertex of $\Gamma_1$. The image of any trait is contained within one of these closed substacks. Thus it suffices to verify the restriction of the map to the preimage of each $V(\gamma_i)$ is flat over $V(\gamma_i)$. This holds by construction.

	\section{The logarithmic Quot space}\label{sec:LogQuot}
	In the sequel we set $X$ a projective (fine and saturated) logarithmic scheme which is logarithmically flat over a point with the trivial logarithmic structure. Fix also a Hilbert polynomial $\Phi$. Let $\mathcal{E}$ be a coherent sheaf on $X$. 

 \begin{definition}\label{defn:LogQuot}
The logarithmic Quot space $\mathrm{Quot}(X,\mathcal{E})$ is the groupoid valued sheaf on the category of logarithmic schemes obtained by sheafifying the presheaf which assigns to a logarithmic scheme $S$ the groupoid of logarithmic surjections of coherent sheaves on $X\times S$ which are flat over $S$.
 \end{definition}

Since logarithmic flatness and being integral are preserved under strict base change, we can think of the logarithmic Quot space as a (not necessarily algebraic) stack in the strict \'etale site. The next two sections constitute a proof of Theorem \ref{Mainthm:LogQuot} and Theorem \ref{thm:LogQuot}.

\subsection{Representability} In Section \ref{sec:DefnTropSupp} we showed that a morphism from $S$ to the presheaf used to define the logarithmic Quot space specifies in particular a map from $\mathpzc{S}$ to the stack of tropical supports $\mathpzc{Supp}.$ Proposition \ref{prop:Family gives Family} shows this assignment descends to define a morphism $$\mathrm{Quot}(X,\mathcal{E}) \rightarrow \mathpzc{Supp}(\mathpzc{X}).$$
Given a tropical model $\supptrop_\Sigma \rightarrow \supptrop$ the associated \textit{proper model} of  $\mathrm{Quot}(X,\mathcal{E})$ is the fibre product $$\mathrm{Quot}_\Sigma(X,\mathcal{E}) = \mathrm{Quot}(X,\mathcal{E})\times_{\supplog}\supplog_\Sigma\xrightarrow{\pi_\Sigma} \mathrm{Quot}(X,\mathcal{E}).$$ Representability properties of the logarithmic Quot space are captured by the proper models.  
 
 \subsubsection{Open cover} Given a cone complex $\supptrop_{\Lambda_2}$ embedded in $\supptrop$ and given moreover a combinatorially flat tropical model of the universal piecewise linear space $$
\begin{tikzcd}
\mathscr{X}_{\Lambda_1} \arrow[r] & X \times \supptrop_{\Lambda_2} \arrow[d] \\
                                  & \supptrop_{\Lambda_2}                   
\end{tikzcd}$$ we apply the Construction \ref{cons:flatmapArtinfan} to define a morphism of piecewise linear spaces whose corresponding diagram of stacks on the category of logarithmic schemes $$\mathcal{X}_{\Lambda_1} \rightarrow \supplog_{\Lambda_2}$$ is flat.
 
 Now set $U_\Lambda$ to be the open substack of the relative Quot scheme $\mathrm{Quot}(\mathcal{X}_{\Lambda_1}/\supplog_{\Lambda_2})$ whose $S$ valued points satisfy two conditions.
\begin{enumerate}[(1)]
    \item \textbf{Logarithmic flatness}: the logarithmic surjection of coherent sheaves $[\pi_\Gamma,q_\Gamma]$ on $S \times X$ is logarithmically flat over $S$.
    \item \textbf{Stability}: The family of tropical supports associated to $[\pi_\Gamma,q_\Gamma]$ defined in Section \ref{sec:TropSuppNotConstDegen} coincides with the family pulled back along the morphism ${S} \rightarrow \supplog_{\Lambda_2}$.
\end{enumerate}
The first condition is open because flatness is an open condition. The fact that stability is open requires more work, and this claim appears as Proposition \ref{prop:StabilityIsOpen}. Note since $\pi:\mathscr{X}_{\Lambda_1} \rightarrow \supptrop_{\Lambda_2}$ is proper, the relative Quot scheme $\mathrm{Quot}(\mathscr{X}_{\Lambda_1}/\supptrop_{\Lambda_2})$ is an algebraic space, and thus $U_\Lambda$ is also an algebraic space. We say $\Lambda$ is \textit{compatible} with $\Sigma$ if $\supptrop_{\Lambda_2}$ is a subcomplex of $\Sigma$.
	\begin{lemma}\label{lem:coverbyUlambda}
		Whenever $\Lambda$ is compatible with $\Sigma$, the natural morphism $$U_\Lambda\rightarrow \mathsf{Quot}_\Sigma(X,\mathcal{E})$$ is an open immersion.
	\end{lemma}
    It follows that $U_\Lambda \rightarrow \mathsf{Quot}(X,\mathcal{E})$ is logarithmic \'etale as it is the composition of logarithmically \'etale morphisms. 
	\begin{proof}
		Consider two tropical models of $\mathscr{X}$, say $\Lambda_1, \Lambda_1'$, which are both integral over a fixed base $\supptrop_{\Lambda_2}$. Write $\overline{\Lambda}$ for the common refinement of $\Lambda_1,\Lambda_1'$ and note there are corresponding maps of universal expansions $$\mathcal{X}_{\overline{\Lambda}}\rightarrow \mathcal{X}_{\Lambda_1},\mathcal{X}_{\overline{\Lambda}}\rightarrow \mathcal{X}_{\Lambda_1'}.$$ 
  
  We now show that the transverse and stable locus in $\mathrm{Quot}(\mathcal{X}_{\Lambda_1}/\supplog_{\Lambda_2},\mathcal{E})$ is canonically identified with the transverse and stable locus in $\mathrm{Quot}(\mathcal{X}_{\Lambda_1'}/\supplog_{\Lambda_2},\mathcal{E})$. Indeed a surjection of sheaves on $S \times_{\supplog_{\Lambda_2}}\mathcal{X}_{\Lambda_1}$ can be pulled back to a surjection of sheaves on $S \times_{\supplog_{\Lambda_2}}\mathcal{X}_{\overline{\Lambda}}.$ Stability ensures this surjection is pulled back from a surjection of sheaves on $S \times_{\supplog_{\Lambda_2}}\mathcal{X}_{\Lambda_1'}.$ Since both ${\Lambda_1}$ and ${\Lambda_1'}$ are models of the universal tropical support, the above operation sends transverse sheaves to transverse sheaves. The restriction of a surjection $q$ of sheaves to strata of relative dimension zero over the base in the tropical support determines $q$. Both $\Lambda_1$ and $\Lambda_1'$ are combinatorially flat over $\Sigma$ so the common refinement does not effect those cones of relative dimension zero over the base which are strata of the tropical support. This ensures the above assignment is a bijection.

  Write $Q^o$ for the logarithmically flat and stable locus in $\mathrm{Quot}(\mathcal{X}_{\Lambda_1}/\supplog_{\Lambda_2},\mathcal{E})$. Note $Q^o$ is an open in the transverse and stable locus of $\mathrm{Quot}(\mathcal{X}_{\Lambda_1}/\supplog_{\Lambda_2},\mathcal{E})$. In light of the previous paragraph we can identify $Q^o$ as an open inside $\mathrm{Quot}(\mathcal{X}_{\Lambda_1'}/\supplog_{\Lambda_2},\mathcal{E})$. 
\end{proof}

\begin{corollary}
    The logarithmic Quot space is a logarithmic space in the sense of \cite[Definition 4.11.1]{MolchoWise}.
\end{corollary}
\begin{proof}
    The logarithmic \'etale morphisms $U_\Lambda \rightarrow \mathrm{Quot}(X,\mathcal{E})$ form the requisite cover.
\end{proof}
\subsubsection{Prorepresentability and cover by proper models}
Denote the set of tropical models $$S_\mathpzc{X} = \{\supptrop_\Sigma(\AX) \rightarrow \supptrop(\AX)\}.$$
 \begin{proposition}
	Taking colimits in the category of stacks over \textbf{LogSch} there is an equality of moduli stacks $$\varinjlim_{\Sigma \in S_\mathpzc{X}} \mathsf{Quot}_\Sigma(X,\mathcal{E})=\mathsf{Quot}(X,\mathcal{E}).$$ 
	\end{proposition}
	\begin{proof}
		The morphisms $\pi_\Sigma$ specify a map $$\varinjlim_{\Sigma \in S_\mathpzc{X}} \mathsf{Quot}_\Sigma(X,\mathcal{E})\rightarrow \mathsf{Quot}(X,\mathcal{E}).$$ We write down an inverse on the level of functors of points. Note the proposition is false if one takes colimits in the category of prestacks. 
		
		A morphism $B \rightarrow \mathsf{Quot}(X,\mathcal{E})$ is the data of an \'etale cover $$\{f_i:U_i\rightarrow B\} \textrm{ where we denote }U_i \cap U_j = U_{ij}$$ and a surjection of expanded logarithmic sheaves $q_i:\mathcal{E} \rightarrow \mathcal{F}$ on $X \times U_i$. Refining the open cover if necessary,  $\mathcal{F}|_{U_i}= [\varphi^{\star}  \mathcal{X}_\Sigma, \mathcal{F}_\Gamma]$ for some choice of $\Sigma$ and $\varphi:U_i \rightarrow \mathsf{Supp}_\Sigma(X)$. Pulling $q_i$ back to a morphism $f_i^{\star} q_i$ of sheaves on $U_i$ we obtain a morphism $$U_i \rightarrow \mathsf{Quot}_{\Sigma_i}(X).$$ Moreover denoting the common subdivision of $\Sigma_i$ and $\Sigma_j$ by $\Sigma_{ij}$ the restriction of $g_i$ to $U_{ij}$ factors $$g_i|_{U_{ij}}:U_{ij} \rightarrow \mathsf{Quot}_{\Sigma_{ij}}(X)\xrightarrow{h_{ij}} \mathsf{Quot}_{\Sigma_i}(X).$$  Since $h_{ij} = h_{ji}$ the compositions  $$U_i\rightarrow \mathsf{Quot}_{\Sigma_i}(X) \rightarrow \varinjlim \mathsf{Quot}_\Sigma^\mathrm{log}(X)$$ glue to define a morphism $$g:B \rightarrow \varinjlim \mathsf{Quot}_\Sigma^\mathrm{log}(X).$$
	\end{proof}

	\begin{theorem}\label{thm:algebraic}
		The model $\mathsf{Quot}_\Sigma(X,\mathcal{E})$ is an algebraic stack of Deligne--Mumford type with logarithmic structure. 
	\end{theorem}
\begin{proof}
    First observe the morphism $U_\Lambda \rightarrow \mathrm{Quot}_{\Lambda_1}(X,\mathcal{E})$ is a strict open immersion as in Lemma~\ref{lem:coverbyUlambda}. We are left to check stabilisers are finite which may be done locally, and thus we check our claim for the $U_\Lambda$. A map from a point to $U_\Lambda$ can be thought of as two pieces of data. First a map from the point $p$ to the Artin stack $\mathrm{Supp}_\Lambda$ and second a surjection of sheaves $q$ on some logarithmic modification $X_\Gamma$. The image of $p$ in $\mathrm{Supp}_\Lambda$ has stabiliser group $\mathbb{G}_m^k$ for some $k$. The stabiliser of $p$ considered a point of $U_\Lambda$ is the subgroup $G$ of $\mathbb{G}_m^k$ whose action fixes $q$. This follows by comparing the definition of tropical support with \cite[Theorem 1.8]{CarocciNabijou1}. Now observe $G$ is an algebraic subgroup of $\mathbb{G}_m^k$ containing no one dimensional subtori by Theorem \ref{lem:defnTropSupp}. All such groups are finite.
\end{proof}
	\subsection{Open subscheme}\label{sec:OpenSubscheme} We observe $\mathrm{Quot}(X,\mathcal{E})^o$ is a open sub--scheme of $\mathrm{Quot}(X,\mathcal{E})$. Indeed $\mathrm{Quot}(X,\mathcal{E})^o$ comes equipped with universal data of a surjection of sheaves $\mathcal{E} \rightarrow \mathcal{F}$ on $\mathrm{Quot}(X,\mathcal{E})^o~\times~X$. This universal surjection of sheaves does not obviously define a logarithmic surjection of coherent sheaves because it is not clear $\mathcal{F}$ is logarithmically flat over $\mathrm{Quot}(X,\mathcal{E})^o$. However we know by Theorem \ref{thm:OldTev} that there is a logarithmic modification of $\mathrm{Quot}(X,\mathcal{E})^o\times X$ such that the strict transform of $\mathcal{F}$ is logarithmically flat. Being integral is automatic because $\mathrm{Quot}^o(X,\mathcal{E})$ has the trivial logarithmic structure, and thus we have constructed a valid logarithmic surjection of coherent sheaves over $\mathrm{Quot}(X,\mathcal{E})^o$.

 The universal property of $\mathrm{Quot}(X,\mathcal{E})$ now defines a map $\iota: \mathrm{Quot}(X,\mathcal{E})^o\rightarrow \mathrm{Quot}(X,\mathcal{E})$. The image is the locus in $\mathrm{Quot}(X,\mathcal{E})$ is the locus with trivial tropical support. This tropical support corresponds to the zero dimensional cone in $\supptrop$ and thus $\iota$ is open. Transversality ensures the map is an injection. Thus we have defined an open immersion.
	\subsection{Universally closed and separated} We have shown the model ${\mathrm{Quot}}_\Sigma(X,\mathcal{E})$ is an algebraic Deligne-Mumford stack $\underline{\mathrm{Quot}}_\Sigma(X,\mathcal{E})$  equipped with a logarithmic structure. 
	\begin{theorem}
		Connected components of the Deligne--Mumford stack $\underline{\mathrm{Quot}}_\Sigma(X,\mathcal{E})$ are universally closed and separated. 
	\end{theorem}
	By \cite[Theorem 2.2.5.2]{MolchoWise} it follows that each model ${\mathrm{Quot}}_\Sigma(X,\mathcal{E})$ satisfies the \textit{unique right lifting property defined} in the same theorem statement. An exercise in abstract nonsense shows ${\mathrm{Quot}}(X,\mathcal{E})$ satisfies the same right lifting property.
	
	We check the valuative criterion. Let $\underline{S} = \mathrm{Spec}(R)$ be a trait with generic point $\eta$ and special point $s$. Consider a commutative square 
	$$
	\begin{tikzcd}
	\eta \arrow[d, hook] \arrow[r] & {\mathsf{Quot}_\Sigma(X, \mathcal{E})} \arrow[d] \\
	S \arrow[r] \arrow[ru, dashed,"f"] & \mathrm{Spec}(\mathbb{C}).                             
	\end{tikzcd}
	$$
	We must check that for any $S$ at most one map $f$ exists. Moreover after replacing $R$ by a ramified base change, the morphism $f$ exists. It is also necessary to check that the logarithmic Quot space is bounded. We defer the boundedness proof to Section.

	\begin{proof}[Proof of valuative criterion]
		We first show existence after a ramified base change. A morphism from $\eta$ to $\mathrm{Quot}$ specifies a morphism $\phi:\eta \rightarrow \supplog_\Sigma$ and a surjection of sheaves on $\mathscr{X}\times_{\supplog_\Sigma} \eta$. Since $\eta$ is a single point the underlying scheme of $\mathscr{X}\times_{\supplog_\Sigma} \eta$ is a product $X_\Gamma \times \eta$ for some scheme $X_\Gamma$. Properness of Grothendieck's Quot scheme $\mathrm{Quot}_\Sigma(\underline{X}_\Gamma,\mathcal{E})$ defines a surjection of sheaves on $S\times X_\Gamma$. By Proposition \ref{Goal:FlatLimits}, possibly after replacing $R$ by a base change, there is a logarithmic scheme $S$ containing $\eta$ as a subscheme, and a logarithmic modification of $X_S$ such that the strict transform of $\mathcal{F}$ is logarithmically flat. The strict transform of $\mathcal{E}$ is already logarithmically flat. By miracle flatness the expansion is integral. Thus we have the data of a logarithmic surjection of coherent sheaves on $X\times S$ which is logarithmically flat over $S$.

        We have thus defined a map from $S$ to $\mathrm{Quot}(X,\mathcal{E})$ and must verify the map factors through $\mathrm{Quot}_\Sigma(X,\mathcal{E})$. This is true after a ramified base change of $R$ provided the cone $\sigma_i$ used to define $S$ is sufficiently small. 
		
		It remains to check uniqueness. Suppose we are given two surjections of logarithmic schemes on $X\times S$ which agree on $X\times \eta$. Modifying the logarithmic structure if necessary, we may choose representatives $$(\pi_\Gamma:X_\Gamma \rightarrow X, q_\Gamma) \textrm{ and } (\pi_\Gamma: X_\Gamma\rightarrow X, q_\Gamma')\textrm{ where both logarithmic modifications are the same}.$$ Note the logarithmic structure on $S$ is one of the $S_\sigma$ introduced in \cite[Section 7.1]{MR20} and discussed in Section \ref{sec:LogStr}. Therefore the fact that common refinements of logarithmic refinements need not be integral is not an issue: simply shrink the cone $\sigma$ and we can be sure that the map is integral.
        By the Quot scheme trick, the surjections $q_\Gamma, q_\Gamma'$ are the data of morphisms to the relative Quot scheme $\mathrm{Quot}(X_\Gamma/S, \mathcal{E})$. This relative Quot scheme is separated (over $S$) by a theorem of Grothendieck. Since $q_\Gamma$ and $q_\Gamma'$ agree on the generic point, separatedness of Grothendieck's Quot scheme implies $q_\Gamma = q_\Gamma'$.
	\end{proof}

\section{Examples}
The goal of this section is to describe examples of the logarithmic Quot space and computations of tropical support. The key takeaway from is that studying 
the logarithmic Quot space does not pose substantially more difficulty than studying Grothendieck's Quot scheme. 

\subsection{Trivial logarithmic structure}
Let $X$ be a logarithmic scheme obtained by equipping a scheme $X$ with the trivial logarithmic structure. Then $X$ is automatically logarithmically flat over a point with the trivial logarithmic structure and any coherent sheaf $\mathcal{E}$ on $X$ is logarithmically flat. 

In this situation logarithmic modifications of $X\times S$ which are integral over $S$ are Kummer and are unimportant for understanding additional logarithmic surjections of coherent sheaves. Thus the logarithmic Quot space of $\mathcal{E}$ on $X$ is Grothendieck's Quot scheme $\mathrm{Quot}(\underline{X},\mathcal{E})$ equipped with the trivial logarithmic structure.

\subsection{Tropical support} We provide two instructive examples of computing the tropical support. Tropical support can be accessed either through torus actions or cohomologically. Here we take the torus action perspective. 

\subsubsection{Points on $\mathbb{A}^2$}
We give an example of computing the tropical support. First consider a logarithmic modification of $\mathbb{A}^2 \times \mathsf{pt}^\dagger$. Such a logarithmic modification is the data of a polyhedral subdivision of $\mathbb{R}^2_{\geq 0}$, see the left of Figure \ref{fig:TropSupp}.

A point of the logarithmic Hilbert scheme of $\mathbb{A}^2$ is specified by a closed subscheme of such an expansion. Two distinct points in a logarithmic modification of $\mathbb{A}^2$ give a length two subscheme of $\mathbb{A}^2$. See the right of Figure \ref{fig:TropSupp}.

\begin{figure}[ht]
	\centering
	\includegraphics[width= 120mm]{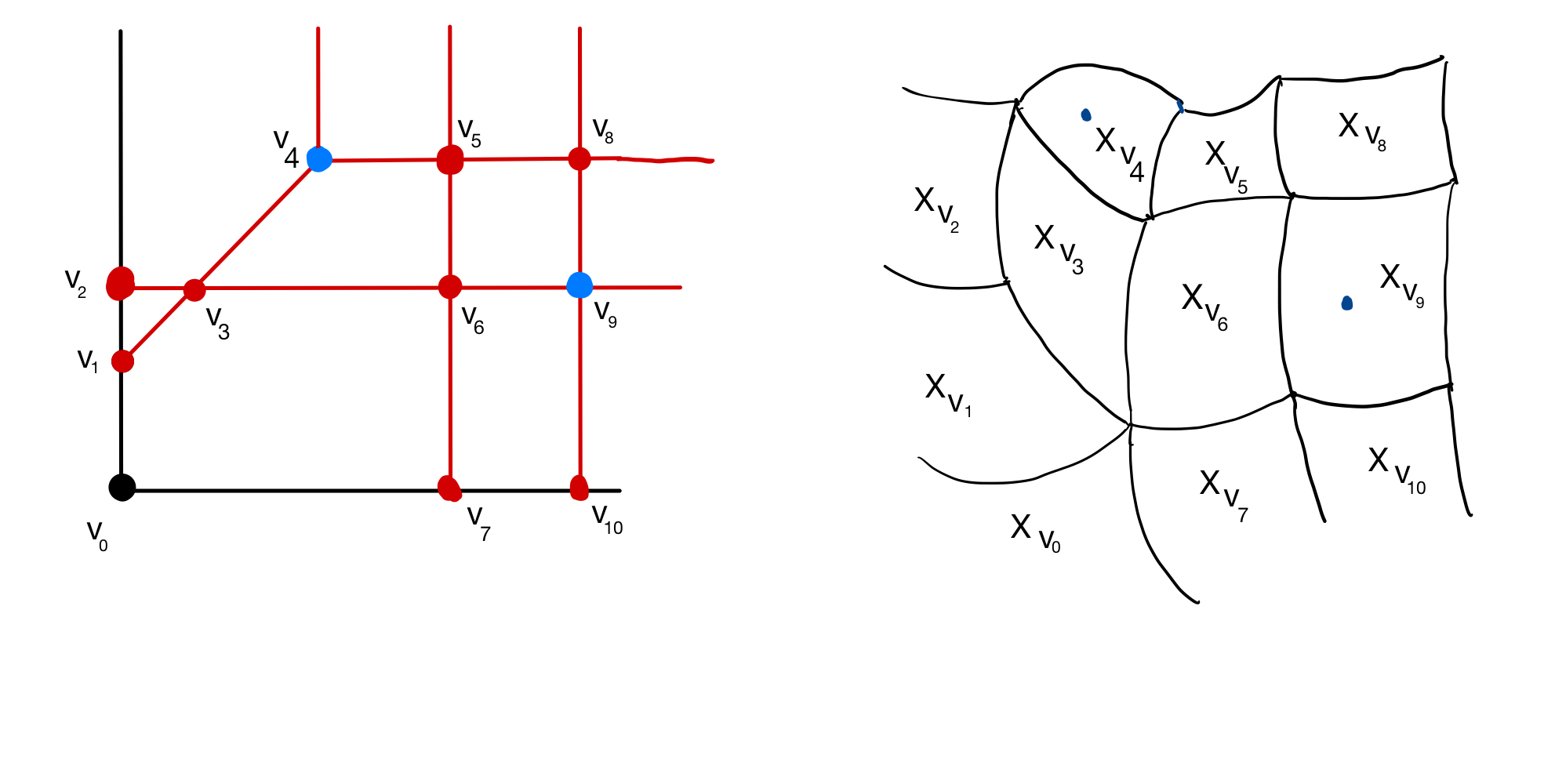}
    \caption{Left a polyhedral subdivision of $\mathbb{R}^2_{\geq 0}$ and right the associated expansion containing a subscheme (purple dots). The tropical support is obtained from the left hand diagram by keeping only data of blue vertices. These are the vertices corresponding to components which contain a point, and are thus fixed by all tori.}\label{fig:TropSupp}
\end{figure}

The tropical support records only the two blue vertices. These vertices correspond to components containing subschemes which are not fixed fixed by action of the two dimensional torus associated to each component. 

\subsubsection{Curves and points in $\mathbb{P}^2$}
We give an example of tropical support for mixed dimensional subschemes. The example in Figure \ref{fig:MixedDim} occurs when studying the moduli space of curves and points in $\mathbb{P}^2$.

\begin{figure}[ht]
	\centering
	\includegraphics[width= 120mm]{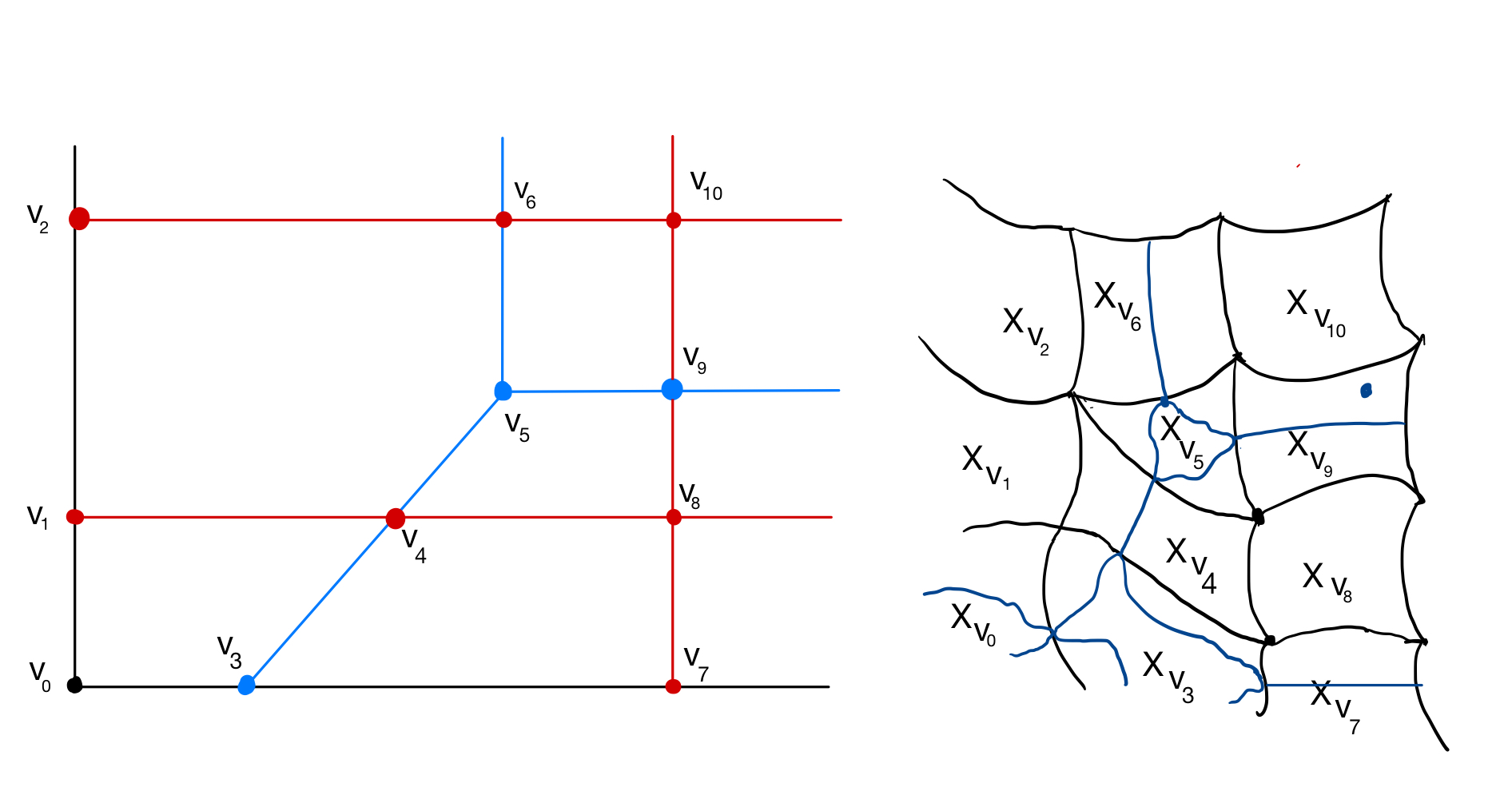}
    \caption{Left a polyhedral subdivision of $\mathbb{R}^2_{\geq 0}$ and right the associated expansion containing a subscheme (purple). The tropical support is obtained from the left hand diagram by keeping only data of blue tropical curve. Note vertex $v_4$ is not seen by the tropical support because the associated subscheme is fixed by a one dimensional torus. By contrast tropical support detects vertex $v_9$ because the subscheme in $X_{v_9}$ is not fixed by any one parameter subgroup of the torus.}\label{fig:MixedDim}
\end{figure}

Once again the blue data is the tropical support. In particular we see a tropical version of an embedded point in a tropical curve. Algebraically this arises from a point in the same component as a tropical curve fixed by a one dimensional torus.

\subsection{The logarithmic linear system}
The logarithmic Hilbert scheme of divisors in a toric variety is the first non--trivial example of a logarithmic Hilbert scheme where the sub-schemes have dimension at least three.

The \textit{logarithmic linear system} of hypersurfaces in a toric variety is a toric stack. For toric surfaces the situation is described in detail in \cite[Sections 1 and 2]{KH20} following an observation of Maulik and Ranganathan \cite{MR20}. The construction is identical for moduli of hypersurfaces in toric varieties of every dimension. The only difference is to increase the dimension of the polytope. The fan of the logarithmic linear system is closely related to the Gelfand--Kapranov--Zelevinsky secondary polytope \cite{GKZ}. 

\subsection{Logarithmic Hilbert scheme of two points on $\mathbb{A}^2$.} The logarithmic strata of the logarithmic Hilbert scheme of two points in $\mathbb{A}^2$ are governed by the possible tropical supports. Up to permuting the $x$ and $y$ axes all possible tropical supports are depicted in Figure \ref{fig:twoptA2}. 

We now describe the isomorphism class of the scheme underlying some of these logarithmic strata. The top logarithmic stratum is the locus of points in $\mathbb{A}^2$ supported away from the toric boundary. The Ghost sheaf of the associated logarithmic stratum is zero.

We now turn our attention to the locus $X_1$ associated to the tropical diagrams in a blue oval. The logarithmic stratum associated to the upper diagram has ghost sheaf $\mathbb{N}^2$. The  underlying locally closed stratum is the stack quotient of the Hilbert scheme of two points on $(\mathbb{C}^\star)^2$ by the action of $(\mathbb{C}^\star)^2$. Note this space is a Deligne--Mumford stack and not a scheme since the ideal $(X^2 - 1,Y-1)$ is fixed by the action sending $X \mapsto -X$ and $Y \mapsto Y$.

Associated to the bottom tropical diagram is a product - one for each dot. Assigned to each dot is the quotient of the Hilbert scheme of one point on $(\mathbb{C}^\star)^2$ by the action of $(\mathbb{C}^\star)^2$. The stratum is thus a product of two copies of $\mathrm{Spec}(\mathbb{C})$ and is thus a single point.

On the level of underlying schemes, the closure of $X_1$ is thus a single point. One can think of $X_1$ as a blowup of $(\mathbb{C}^\star)^2$ in the point $(1,1)$. There are no one point compactifications of this blowup. Indeed blowing down the exceptional curve in such a compactification yields a one point compactification of $(\mathbb{C}^\star)^2$. There are no one point compactifications of $(\mathbb{C}^\star)^2$ which are algebraic spaces. The author thanks Bernd Siebert for a conversation about this example.

\begin{figure}[ht]
	\centering
	\includegraphics[width= 120mm]{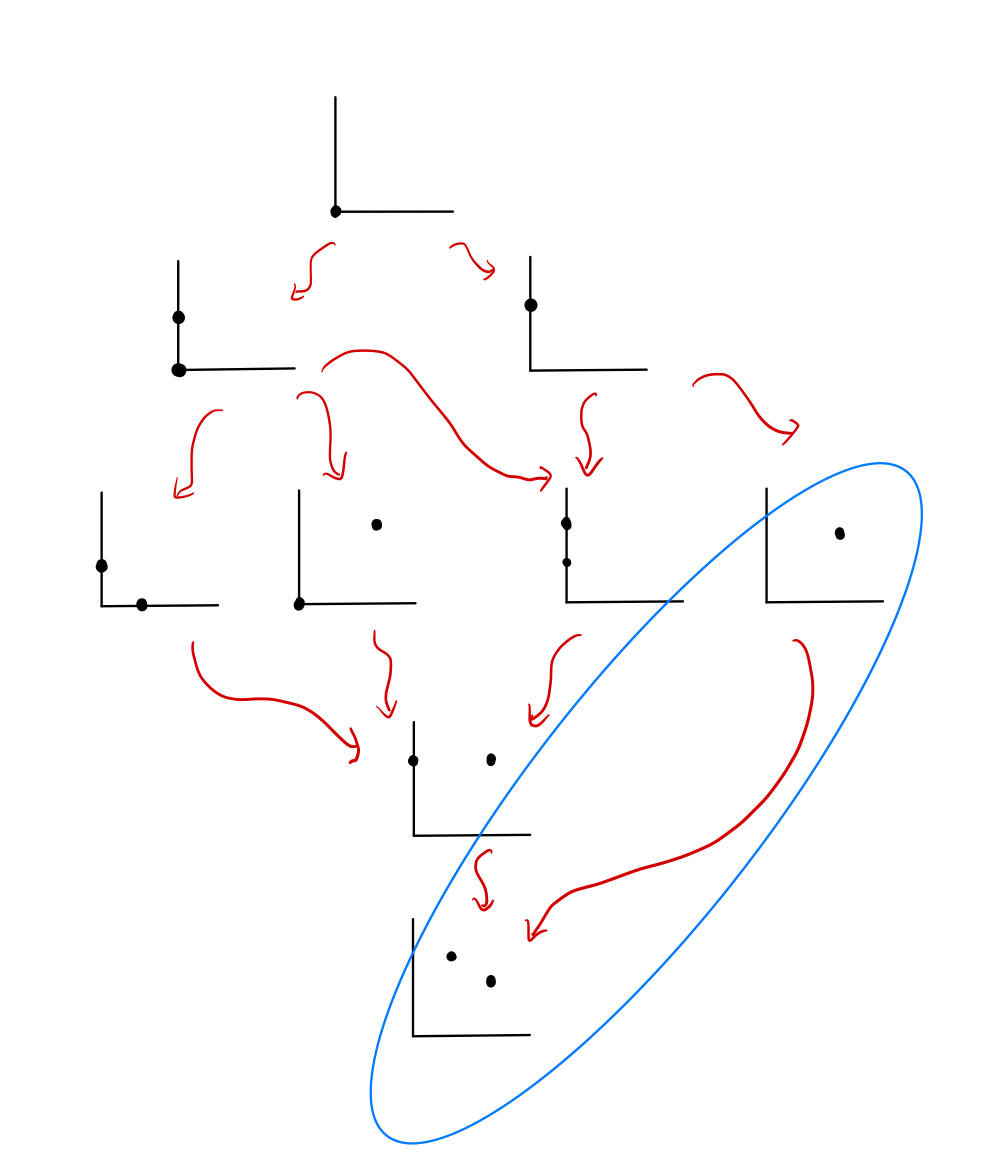}
    \caption{Tropical supports which appear when studying the logarithmic Hilbert scheme of $\mathbb{P}^2$. The height records the rank of the logarithmic stratum (rank zero for the first row and rank four on the bottom). There is a red line from diagram A to diagram B whenever the logarithmic stratum associated to B lies in the closure of the logarithmic stratum associated to A. The blue circle is explained in the text.}\label{fig:twoptA2}
\end{figure}

\subsection{Two choices of subdivision} In the perspective of our paper Maulik and Ranganathan's version of the logarithmic Hilbert scheme of curves requires two choices. They fix both a subdivision of $\mathrm{Supp}$ and a subdivision of the universal family $\mathscr{X}$. The goal of this section is to explain how these choices are connected and the ramifications for choices of tropical model of the logarithmic Quot space. 

Fix a combinatorial type of tropical support in the universal expansion. This fixes a combinatorial type of tropical support. To subivide $\mathscr{X}$, it suffices to choose for each point of $\supptrop$ a polyhedral subdivision of $\mathrm{Trop}(X)$ refining the PL structure induced by $\mathscr{X}$. This combinatorial model must be chosen in a particular way in order to define a valid tropical model of $\mathrm{Trop}(X)$. The key criterion is that the slant of each ray in the one skeleton is locally constant. 

Fixing such a choice on each stratum of $\supptrop$ gives rise to a subdivision of the universal expansion. This subdivision need not be combinatorially flat, but there is a universal way to flatten it (by pushing forward the locally closed stratification as in \cite{molcho2019universal}). This defines a subdivision $\supptrop' \rightarrow \supptrop$ which need not be a tropical model.

Maulik and Ranganathan make a second choice of tropical model $\supptrop_\Sigma \rightarrow \supptrop.$ If one fixes a subdivision of the universal expansion up front then the choice of $\Sigma$ is not arbitrary. Indeed combinatorial flatness implies we can factor $$\supptrop_\Sigma \rightarrow \supptrop' \rightarrow \supptrop.$$ Thus the two choices are linked, although neither choice determines the other. 

	\bibliographystyle{alpha}
	\bibliography{Bibliography2}
\end{document}